\documentclass[11pt,a4paper]{amsart}

\usepackage{verbatim}
\usepackage{amsmath}
\usepackage{mathtools}
\usepackage{amsthm}
\usepackage{amssymb}

\usepackage[bookmarksopen,bookmarksdepth=3]{hyperref}

\newcommand*\fullref[3][\relax]{%
  \ifdefined\hyperref%
    {\hyperref[#3]{#2\penalty 200\ \ref*{#3}#1}}%
  \else%
    {#2\penalty 200\ \relax\ref{#3}#1}%
  \fi%
}

\usepackage{booktabs}

\usepackage{tikz}
\usetikzlibrary{calc}

\usepackage{etoolbox}

\usetikzlibrary{arrows.meta}

\tikzset{
  normalarrow/.style={line width=.6pt},
}

\tikzset{
  normalarrowlabel/.style={
    auto,
    inner sep=.5mm,
    outer sep=0mm,
    font=\footnotesize,
  },
  tinyarrowlabel/.style={
    auto
    inner sep=.2mm,
    outer sep=0mm,
    font=\tiny,
  },
  arrowlabel/.style={
    normalarrowlabel
  },
  marrowlabel/.style={
    normalarrowlabel,
  },
  crystaledges/.style={
    f1/.style={
      ->,
      normalarrow,
      labelled/.style={every to/.style={edge node={node[marrowlabel] {$1$}}}},
      coloured/.style={draw=red},
    },
    f2/.style={
      ->,
      normalarrow,
      labelled/.style={every to/.style={edge node={node[marrowlabel] {$2$}}}},
      coloured/.style={draw=blue},
    },
    f3/.style={
      ->,
      normalarrow,
      labelled/.style={every to/.style={edge node={node[marrowlabel] {$3$}}}},
      coloured/.style={draw=green!50!black},
    },
    f4/.style={
      ->,
      normalarrow,
      labelled/.style={every to/.style={edge node={node[marrowlabel] {$4$}}}},
      coloured/.style={draw=brown},
    },
    f5/.style={
      ->,
      normalarrow,
      labelled/.style={every to/.style={edge node={node[arrowlabel] {$5$}}}},
      coloured/.style={draw=cyan},
    },
    f6/.style={
      ->,
      normalarrow,
      labelled/.style={every to/.style={edge node={node[arrowlabel] {$6$}}}},
      coloured/.style={draw=violet},
    },
    f7/.style={
      ->,
      normalarrow,
      labelled/.style={every to/.style={edge node={node[arrowlabel] {$7$}}}},
      coloured/.style={draw=brown!50!black},
    },
    f8/.style={
      ->,
      normalarrow,
      labelled/.style={every to/.style={edge node={node[arrowlabel] {$7$}}}},
      coloured/.style={draw=magenta},
    },
    f9/.style={
      ->,
      normalarrow,
      labelled/.style={every to/.style={edge node={node[arrowlabel] {$7$}}}},
      coloured/.style={draw=brown!50!black},
    },
    f10/.style={
      ->,
      normalarrow,
      labelled/.style={every to/.style={edge node={node[arrowlabel] {$7$}}}},
      coloured/.style={draw=lime},
    },
    f11/.style={
      ->,
      normalarrow,
      labelled/.style={every to/.style={edge node={node[arrowlabel] {$7$}}}},
      coloured/.style={draw=olive},
    },
    f12/.style={
      ->,
      normalarrow,
      labelled/.style={every to/.style={edge node={node[arrowlabel] {$7$}}}},
      coloured/.style={draw=pink},
    },
    fi/.style={
      ->,
      normalarrow,
      labelled/.style={every to/.style={edge node={node[marrowlabel] {$i$}}}},
      coloured/.style={draw=gray},
    },
    fn2/.style={
      ->,
      normalarrow,
      labelled/.style={every to/.style={edge node={node[marrowlabel] {$n{-}2$}}}},
      coloured/.style={draw=brown},
    },
    fn1/.style={
      ->,
      normalarrow,
      labelled/.style={every to/.style={edge node={node[arrowlabel] {$n{-}1$}}}},
      coloured/.style={draw=cyan},
    },
    fn/.style={
      ->,
      normalarrow,
      labelled/.style={every to/.style={edge node={node[arrowlabel] {$n$}}}},
      coloured/.style={draw=violet},
    },
    df1/.style={
      f1,
      densely dotted,
    },
    df2/.style={
      f2,
      densely dotted,
    },
    df3/.style={
      f3,
      densely dotted,
    },
    dfi/.style={
      fi,
      densely dotted,
    },
    dfn2/.style={
      fn2,
      densely dotted,
    },
    dfn1/.style={
      fn1,
      densely dotted,
    },
    dfn/.style={
      fn,
      densely dotted,
    },
  },
  labelledcrystaledges/.style={
    crystaledges,
    f1/.append style={labelled},
    f2/.append style={labelled},
    f3/.append style={labelled},
    f4/.append style={labelled},
    f5/.append style={labelled},
    f6/.append style={labelled},
    f7/.append style={labelled},
    f8/.append style={labelled},
    f9/.append style={labelled},
    f10/.append style={labelled},
    f11/.append style={labelled},
    f12/.append style={labelled},
    fi/.append style={labelled},
    fn2/.append style={labelled},
    fn1/.append style={labelled},
    fn/.append style={labelled},
  },
  colouredcrystaledges/.style={
    crystaledges,
    f1/.append style={coloured},
    f2/.append style={coloured},
    f3/.append style={coloured},
    f4/.append style={coloured},
    f5/.append style={coloured},
    f6/.append style={coloured},
    f7/.append style={coloured},
    f8/.append style={coloured},
    f9/.append style={coloured},
    f10/.append style={coloured},
    f11/.append style={coloured},
    f12/.append style={coloured},
    fi/.append style={coloured},
    fn2/.append style={coloured},
    fn1/.append style={coloured},
    fn/.append style={coloured},
  },
  labelledcolouredcrystaledges/.style={
    crystaledges,
    f1/.append style={labelled,coloured},
    f2/.append style={labelled,coloured},
    f3/.append style={labelled,coloured},
    f4/.append style={labelled,coloured},
    f5/.append style={labelled,coloured},
    f6/.append style={labelled,coloured},
    f7/.append style={labelled,coloured},
    f8/.append style={labelled,coloured},
    f9/.append style={labelled,coloured},
    f10/.append style={labelled,coloured},
    f11/.append style={labelled,coloured},
    f12/.append style={labelled,coloured},
    fi/.append style={labelled,coloured},
    fn2/.append style={labelled,coloured},
    fn1/.append style={labelled,coloured},
    fn/.append style={labelled,coloured},
  },
  crystalvertex/.style={
    font=\small,
    inner sep=.5mm,
    outer sep=0mm,
  },
  smallcrystalvertex/.style={
    crystalvertex,
    font=\scriptsize,
  },
  bigcrystalvertex/.style={
    crystalvertex,
    inner sep=1mm,
    font=\normalsize,
  },
  crystal/.style={
    x=10mm,
    y=10mm,
    every node/.style={crystalvertex},
    labelledcrystaledges,
  },
  bigcrystal/.style={
    x=15mm,
    y=15mm,
    every node/.style={bigcrystalvertex},
    labelledcrystaledges,
  },
  smallcrystal/.style={
    x=7mm,
    y=7mm,
    every node/.style={smallcrystalvertex},
    colouredcrystaledges,
  },
}

\usetikzlibrary{matrix}

\tikzset{
  pretableaumatrix/.style={
    ampersand replacement=\&,
    matrix of math nodes,
    outer sep=1mm,
    inner sep=0mm,
    anchor=center,
    row sep={between borders,-\pgflinewidth},
    column sep={between borders,-\pgflinewidth},
    dottedentry/.style={densely dotted},
    spaceentry/.style={draw=none,execute at begin node=\null},
  },
  pretableaunode/.style={
    font=\small,
    draw=gray,
    sharp corners,
    rectangle,
    anchor=base,
    text height=3.75mm,
    text depth=1.25mm,
    minimum height=5mm,
    minimum width=5mm,
    inner sep=0mm,
    outer sep=0mm,
  },
  tableaumatrix/.style={
    pretableaumatrix,
    every node/.append style={
      pretableaunode,
    },
  },
  medtableaumatrix/.style={
    pretableaumatrix,
    every node/.append style={
      pretableaunode,
      font=\footnotesize,
      text height=2.75mm,
      text depth=.75mm,
      minimum height=3.5mm,
      minimum width=3.5mm
    },
  },
  smalltableaumatrix/.style={
    pretableaumatrix,
    every node/.append style={
      pretableaunode,
      font=\scriptsize,
      text height=1.85mm,
      text depth=.15mm,
      minimum height=2.5mm,
      minimum width=2.5mm,
    },
  },
  tinytableaumatrix/.style={
    pretableaumatrix,
    every node/.append style={
      pretableaunode,
      font=\tiny,
      text height=1.25mm,
      text depth=.15mm,
      minimum height=1.75mm,
      minimum width=1.75mm
    },
  },
  tableau/.style={
    baseline=-1.25mm,
    every matrix/.style={tableaumatrix},
  },
  medtableau/.style={
    baseline=-1.25mm,
    every matrix/.style={medtableaumatrix},
  },
  smalltableau/.style={
    baseline=-1.25mm,
    every matrix/.style={smalltableaumatrix},
  },
  preshapetableaumatrix/.style={
    pretableaumatrix,
    execute at end cell={\strut},
    every node/.append style={
      draw=black,
      anchor=base,
      inner sep=0mm,
      outer sep=0mm,
    },
    shadedentry/.style={fill=gray},
    darkshadedentry/.style={fill=darkgray},
  },
  medshapetableaumatrix/.style={
    preshapetableaumatrix,
    every node/.append style={
      text height=2.75mm,
      text depth=.75mm,
      minimum height=3.5mm,
      minimum width=3.5mm
    },
  },
  shapetableaumatrix/.style={
    ampersand replacement=\&,
    matrix of math nodes,
    outer sep=0mm,
    inner sep=0mm,
    anchor=base,
    row sep={between borders,-\pgflinewidth},
    column sep={between borders,-\pgflinewidth},
    execute at begin cell={\strut},
    every node/.append style={draw,anchor=base,text height=1mm,text depth=.5mm,minimum size=1.5mm,inner sep=0mm,outer sep=0mm},
  },
  shapetableau/.style={
    every matrix/.style={shapetableaumatrix},
  },
  topalign/.style={
    every matrix/.append style={name=maintableau,anchor=maintableau-1-1.base},
    baseline,
  },
}

\newcommand*\tableau[2][]{\tikz[tableau,#1]\matrix{#2};}

\theoremstyle{definition}
\newtheorem{definition}{Definition}[section]

\newtheorem{example}[definition]{Example}

\theoremstyle{plain}
\newtheorem{corollary}[definition]{Corollary}
\newtheorem{lemma}[definition]{Lemma}
\newtheorem{proposition}[definition]{Proposition}
\newtheorem{theorem}[definition]{Theorem}

\numberwithin{equation}{section}

\newcommand*{\defterm}[1]{\textit{#1}}

\DeclarePairedDelimiter{\set}{\{}{\}}
\DeclarePairedDelimiterX{\gset}[2]{\{}{\}}{\,#1:#2\,}

\DeclarePairedDelimiter{\parens}{\lparen}{\rparen}

\DeclareMathOperator{\im}{im}
\DeclareMathOperator{\dom}{dom}

\newcommand*{\emptyword}{\varepsilon}
\newcommand*{\imreduces}{\rightarrow}

\newcommand*{\reduces}{\rightarrow^*}

\newcommand*{\nset}{\mathbb{N}}
\newcommand*{\zset}{\mathbb{Z}}

\DeclarePairedDelimiterX{\pres}[2]{\langle}{\rangle}{#1\,\delimsize\vert\,\mathopen{}#2}

\newcommand*{\drel}[1]{\mathcal{#1}}

\newcommand*\e{\tilde{e}}
\newcommand*\f{\tilde{f}}
\newcommand*\g{\tilde{g}}

\newcommand*\ecount{\epsilon}
\newcommand*\fcount{\varphi}

\newcommand*\Pl{\mathrm{Pl}}

\newcommand*\lalg[1]{\mathfrak{#1}}

\newcommand*\aA{\mathcal{A}}
\newcommand*\aB{\mathcal{B}}
\newcommand*\aC{\mathcal{C}}
\newcommand*\aD{\mathcal{D}}
\newcommand*\aG{\mathcal{G}}
\newcommand*\aX{\mathcal{X}}

\newcommand*\wt{\mathrm{wt}}

\newcommand*{\rpad}{\delta_{\mathrm{R}}}
\newcommand*{\lpad}{\delta_{\mathrm{L}}}

\newcommand*\lex{\text{lex}}
\newcommand*\lenlex{\text{lenlex}}

\newenvironment{rhoarray}[1]{\bgroup\arraycolsep=0pt\begin{array}{#1}}{\end{array}\egroup}

\newcommand*\dist{\mathrm{dist}}

\renewcommand*\bar[1]{\overline{#1}}
\newcommand*\smashbar[1]{{\mathrlap{\smash{\overline{#1}}}\phantom{#1}}}

\newcommand*\soverbrace[1]{\overbrace{\vphantom{)^k}#1}}

\newcommand*\n{n}

\begin{document}

\title[Crystal monoids \& crystal bases]{\texorpdfstring{Crystal monoids \& crystal bases: rewriting systems and
    biautomatic structures for plactic monoids of types $A_{\n}$, $B_{\n}$, $C_{\n}$, $D_{\n}$, and $G_2$}{Crystal
    monoids \& crystal bases: rewriting systems and biautomatic structures for plactic monoids of types An, Bn, Cn, Dn,
    and G2}}

\author{Alan J. Cain}
\address{%
Centro de Matem\'{a}tica e Aplica\c{c}\~{o}es\\
Faculdade de Ci\^{e}ncias e Tecnologia\\
Universidade Nova de Lisboa\\
2829--516 Caparica\\
Portugal
}
\email{%
a.cain@fct.unl.pt
}
\thanks{The first named author was supported by an Investigador \textsc{FCT} fellowship (\textsc{IF}/01622/2013/\textsc{CP}1161/{\sc
    CT}0001).}

\author{Robert D. Gray}
\address{%
School of Mathematics\\
University of East Anglia\\
Norwich NR4 7TJ\\
United Kingdom
}
\email{%
Robert.D.Gray@uea.ac.uk
}

\thanks{The second named author was partially supported by the {\scshape EPSRC} grant EP/N033353/1 `Special inverse
  monoids: subgroups, structure, geometry, rewriting systems and the word problem'.}

\author{Ant\'{o}nio Malheiro}
\address{%
Centro de Matem\'{a}tica e Aplica\c{c}\~{o}es\\
Faculdade de Ci\^{e}ncias e Tecnologia\\
Universidade Nova de Lisboa\\
2829--516 Caparica\\
Portugal}
\address{%
Centro de \'{A}lgebra da Universidade de Lisboa\\
Av. Prof. Gama Pinto 2\\
1649--003 Lisboa\\
Portugal\\
}
\email{%
ajm@fct.unl.pt
}

\thanks{The first and third named authors were partially supported by by the Funda\c{c}\~{a}o para a Ci\^{e}ncia e a
  Tecnologia (Portuguese Foundation for Science and Technology) through the project {\sc UID}/{\sc MAT}/00297/2013
  (Centro de Matem\'{a}tica e Aplica\c{c}\~{o}es) and the project {\scshape PTDC}/{\scshape MHC-FIL}/2583/2014.}

\thanks{All three authors were partially supported by by the Funda\c{c}\~{a}o para a Ci\^{e}ncia e a Tecnologia
  (Portuguese Foundation for Science and Technology) through the project {\scshape PTDC}/{\scshape MHC-FIL}/2583/2014
  and {\scshape PTDC}/{\scshape MHC-PUR}/31174/2017.}

\thanks{Much of the research leading to this paper was undertaken during visits by the second author to the Centro de
  Matem\'{a}tica e Aplica\c{c}\~{o}es, Universidade Nova de Lisboa. We thank the centre and university for their
  hospitality. These visits were funded by the \textsc{FCT} exploratory project
  \textsc{IF}/01622/2013/\textsc{CP}1161/\textsc{CT}0001 (attached to the first author's fellowship).}

\thanks{The authors thank C\'{e}dric Lecouvey for supplying offprints, Duarte Chambel Ribeiro for pointing out an error,
  and Vanda Martins for dealing with administrative matters arising from the second author's visits to Lisbon.}

\begin{abstract}
  The vertices of any (combinatorial) Kashiwara crystal graph carry a natural monoid structure given by identifying
  words labelling vertices that appear in the same position of isomorphic components of the crystal. Working on a purely
  combinatorial and monoid-theoretical level, we prove some foundational results for these crystal monoids, including the
  observation that they have decidable word problem when their weight monoid is a finite rank free abelian group. The
  problem of constructing finite complete rewriting systems, and biautomatic structures, for crystal monoids is then
  investigated. In the case of Kashiwara crystals of types $A_n$, $B_n$, $C_n$, $D_n$, and $G_2$ (corresponding to the
  $q$-analogues of the Lie algebras of these types) these monoids are precisely the generalised plactic monoids
  investigated in work of Lecouvey. We construct presentations via finite complete rewriting systems for all of these
  types using a unified proof strategy that depends on Kashiwara's crystal bases and analogies of Young tableaux, and on
  Lecouvey's presentations for these monoids. As corollaries, we deduce that plactic monoids of these types have finite
  derivation type and satisfy the homological finiteness properties left and right $\mathrm{FP}_\infty$. These rewriting
  systems are then applied to show that plactic monoids of these types are biautomatic and thus have word problem
  soluble in quadratic time.


\end{abstract}

\maketitle

\section{Introduction}

The Plactic monoid is a fundamental algebraic object
which captures a natural monoid structure carried by the set
of semistandard Young tableaux. It arose originally in the work
of Schensted \cite{schensted_longest} on algorithms for
finding the maximal length of a nondecreasing
subsequence of a given word over the
 ordered alphabet $\aA_n = \set{1 < 2 < \ldots < n}$.
The output of Schensted's algorithm is a tableau and,
by identifying pairs of words that lead to the same tableau,
one obtains the Plactic monoid $\Pl(A_n)$ of rank $n$.
Following this, Knuth \cite{knuth_permutations} found
a finite set of defining relations for the Plactic monoid.
An in-depth systematic study of the
Plactic monoid was then carried out in the work of
Sch\"{u}tzenberger \cite{schtzenberger1977} and
Lascoux and Sch\"{u}tzenberger \cite{lascoux_plaxique}.
Since then, the Plactic monoid and its
corresponding semigroup algebra, the Plactic algebra, have found applications in
various aspects of representation theory and algebraic
combinatorics.
Sch\"{u}tzenberger \cite{schutzenberger_pour} argues that the Plactic monoid ought to be considered as one of the
fundamental monoids in algebra. He gives several reasons to support this claim, including the fact that the Plactic
monoid was used to give the first correct proofs of the Littlewood--Richardson rule for products of Schur functions by
Sch\"{u}tzenberger himself \cite{schtzenberger1977} and independently by Thomas
\cite{thomas_baxter,thomas_schensted}. (For further details on the Littlewood--Richardson rule and the history of
attempts to prove it, see \cite[Section~5.4]{lothaire_algebraic}, \cite[Appendix]{green_polynomial},
\cite[\S~4]{vanleeuwen_littlewood}, and \cite[Chapter~7, Appendix~1]{stanley_enumerative2}.)

Numerous other applications of the Plactic monoid have since been discovered including a combinatorial description of Kostka--Foulkes polynomials
\cite{lascoux_plaxique,lascoux_foulkes},
a noncommutative version of the Demazure character
formula, and of the Schubert polynomials \cite{Lascoux_Schubert, schutzenberger_Schubert}.
The Plactic monoid has motivated a wide range of other
interesting work including the discovery of variations on this monoid
like the shifted \cite{serrano_shifted} and hypoplactic monoids \cite{krob_noncommutative}, Littelmann's
generalization to Plactic algebras for semisimple Lie algebras \cite{littelmann_plactic},
the investigation of the Chinese monoid \cite{cassaigne_chinese},
Hilbert series (growth functions) \cite{duchamp_plactic},
the conjugacy problem \cite{cm_conjugacy},
homogeneous monoids and algebras which include monoids attached to set-theoretic solutions
to Yang--Baxter equations
\cite{Okninski2014, Jespers2015, Cedo2012(2), DehornoyYBE},
semigroup identities \cite{KubatIdentities},
and the theory of quadratic normalization
\cite{DehornoyArxiv}. Some structural results for Plactic algebras were obtained in
\cite{cedo_plactic, kubat_plactic}.
An excellent general introduction to the Plactic monoid is given in the
article of Lascoux, Leclerc and Thibon \cite[Chapter~5]{lothaire_algebraic}.

%
%
%
%
%
%
%
%
%
%
%
%
%
%
%
%
%
%
%
%
%
%
%
%

%
%


%
%
One of the most exciting connections which has recently emerged are
the links between the Plactic monoid and Kashiwara's crystal basis theory. This subject has its origins in the theory of quantum groups \cite{hong_quantumgroups}. The notion of the quantised enveloping algebra, or quantum group, $U_q(\lalg{g})$ associated with a symmetrisable Kac--Moody Lie algebra $\lalg{g}$ was discovered independently by Drinfeld \cite{Drinfeld1985} and Jimbo \cite{Jimbo1985} in 1985 while studying solutions of the quantum Yang--Baxter equations. Kashiwara \cite{kashiwara_crystalizing, kashiwara_crystalbases} introduced crystals in order to give a combinatorial description of modules over $U_q(\lalg{g})$ when $q$ tends to zero. Crystals are extremely useful combinatorial tools for studying representations of these algebras. For example, knowing the crystal of a representation allows one to deduce tensor product and branching rules involving that representation. Since its introduction this important theory has been developed and generalised in multiple directions e.g. to quantum affine algebras, superalgebras and quantum queer superalgebras; see \cite{Kang_walls, Benkart_super, Grantcharov_TAMS, Grantcharov_JEMS}.
%
%
%
%
%
%
%

The connection with the Plactic monoid comes via the study of crystal bases of
$U_q(\lalg{gl}_{n})$-modules. These type-$A_n$ crystals
have vertex set corresponding to all words over the alphabet $\aA_n = \set{1 < 2 < \ldots < n}$, directed edges labelled by colours from the set $I = \{1, 2, \ldots, n-1 \}$ which are determined by the Kashiwara operators $\e_i$ and $\f_i$, and weights coming from the free abelian group $\mathbb{Z}^n$ given by word content (see \fullref{Section}{sec:crystalsdef} for full details of this construction).
An isomorphism between two connected components of the crystal is a weight preserving bijection which maps edges to edges preserving colours. If one defines a relation by saying that two words are equivalent if there is an isomorphism between their respective connected components mapping one vertex to the other then it turns out that this relation on $\aA_n^*$ is equal to the Plactic relation mentioned above.
In this way, the Plactic monoid $\Pl(A_n)$ may be defined in terms of crystals of type $A_n$.
%
%
There are a number of explicit constructions known for crystals of representations of other quantum algebras.  In
addition to type $A_n$, explicit descriptions of crystals are known for simple Lie algebras of types $B_n$, $C_n$,
$D_n$, and the exceptional type $G_2$; see
\cite{hong_quantumgroups,KangMisra,kashiwara_crystalgraphs,Littelmann1995,lecouvey_survey}.  For crystals of each of
these types,
%
%
aspects of theory have been developed. As part of their description of crystals of types $A_n$, $B_n$, $C_n$, and $D_n$, Kashiwara and Nakashima \cite{kashiwara_crystalgraphs}
develop
the correct generalisation of semistandard tableaux for classical
types via the notion of admissible column.
%
%
%
%
For all of these
types, Lecouvey obtained finite presentations via Knuth-type relations
for the corresponding crystal monoids (as defined in \fullref{Section}{sec:relfromcrystals}
below), he also gives Schensted--type insertion algorithms and
establishes a Robinson--Schensted type correspondence in all of these
cases \cite{lecouvey_cn, lecouvey_bndn, lecouvey_survey}.
Bumping and sliding algorithms for $C_n$-tableaux were also independently obtained by Baker \cite{Baker2000}.
Analogous results for infinite rank quantum groups were given by Lecouvey in \cite{lecouvey_infinite}.


%
%

%
%
%
%
%
%

In addition to shedding new light on the connection between the Plactic monoid and the representation theory of Lie algebras, this viewpoint also gives rise to a natural family of monoids arising from crystals, generalising the classical Plactic monoid.
Following Kashiwara \cite{KashiwaraBanff} a crystal is an edge-coloured directed graph satisfying a certain simple set of axioms.
As we shall see in \fullref{Section}{sec:relfromcrystals} below, every abstract combinatorial crystal gives rise to a monoid, in the same way that the classical Plactic monoid arises from $A_n$ above.
%
%
%
%
Examples of Crystal monoids (with weights from a free abelian group) include the classical Plactic monoid $Pl(A_n)$, each of the Plactic-type monoids studied by Lecouvey in \cite{lecouvey_cn, lecouvey_bndn, lecouvey_survey}, and also other important well-studied monoids such as the bicyclic monoid.
%
%
%
%
%
%
%
%
%
%
%
%
%
%
%
%
%
%

In more detail, as mentioned above, in the general abstract definition of
combinatorial crystal (see \fullref{Section}{sec:crystalsdef} below for
a full definition)
the vertices correspond to
words over a finite alphabet $X$, and weight-preserving isomorphisms
between connected components define a congruence $\sim$ on the free
monoid $X^*$.
The corresponding crystal monoid is then the monoid $X^* / {\sim}$
obtained by factoring the free monoid by this congruence.
This connects the theory of Kashiwara crystals directly
to combinatorial semigroup theory (the study of semigroups defined by
generators and relations), combinatorics on words, and formal language
theory. For instance, Lecouvey's results
\cite{lecouvey_cn, lecouvey_bndn, lecouvey_survey}
show in particular that for all classical
types, these crystal monoids $X^* / {\sim}$ are all finitely
presented. Powerful tools exist for studying monoids defined by
presentations in this way, including the theories of (Noetherian and
confluent) string rewriting systems \cite{book_srs} and automata
theory, specifically the theory of automatic groups and monoids
\cite{epstein_wordproc, campbell_autsg}.

The defining property for automatic groups and monoids is the
existence of a rational set of normal forms (with respect to some
finite generating set $A$) such that we have, for each generator in
$A$, a finite automaton that recognizes pairs of normal forms that
differ by multiplication by that generator. It is a consequence of the
definition that automatic monoids (and in particular automatic groups)
have word problem that is soluble in quadratic time
\cite[Corollary~3.7]{campbell_autsg}. Automatic groups have attracted a lot of attention over the last 25~years, in part because of the large number of natural and important
classes of groups that have this property.  The class of automatic
groups includes:
various small cancellation groups \cite{gersten_smallcancellation},
Artin groups of finite and large type \cite{holt_artingroups}, braid
groups, and hyperbolic groups in the sense of Gromov
\cite{gromov_hyperbolic}. In parallel, the theory of automatic monoids
has been extended and developed over recent years.  Classes of monoids
that have been shown to be automatic include divisibility monoids
\cite{picantin_finite}, singular Artin monoids of finite type
\cite{corran_singular}, and monoids arising from confluence monadic rewriting systems \cite{otto_srsauto,c_mrsassm}.  Several complexity and decidability results
for automatic monoids are obtained in \cite{lohrey_decidability}.
Other aspects of the theory of automatic monoids that have been
investigated include connections with the theory of Dehn functions
\cite{otto_dehn} and complete rewriting systems
\cite{otto_automonversus}.

In the cases that they are applicable,
these tools of string rewriting systems and automatic structures
give rise to algorithms for working with the monoids,
which can in particular be used to study decidability and complexity
questions. These are very natural aspects of theory to develop given
the fundamental role that algorithms play in the theory of
Plactic monoids, tableaux and Kashiwara crystals outlined above.
%
%
%
%
%
%
Of course any results about the complexity of algorithms
for working
with these monoids (algorithms that operate on words)
may be translated to results about algorithms
for working with the corresponding tableaux and crystal graphs (see Section~\ref{sec:biautomaticity} for examples of this).
It was precisely these kinds of ideas that motivated the current
authors' paper \cite{cgm_plactic} on the classical Plactic monoid.
It was pointed out
by E. Zelmanov [during his
  plenary lecture at the international conference \emph{Groups and
    Semigroups: Interactions and Computations} (Lisbon, 25--29 July
  2011)]
that since Schensted's
algorithm can be used to show that the Plactic monoid has word problem
that is soluble in
quadratic time, it is natural to ask whether Plactic monoids are
automatic. This is a natural question since (as mentioned above)
all automatic monoids
have word problem decidable in quadratic time. In \cite{cgm_plactic} we
gave an affirmative
answer to this question. We did this by first constructing a finite
complete rewriting system for the Plactic monoid, with respect
to the set of column generators.
%
%
%
Beginning with this finite complete rewriting
system, we then showed that for
Plactic monoids, finite transducers may be constructed to perform left
(respectively right) multiplication by a generator. We then applied this
result to show that Plactic monoids of arbitrary finite rank are
biautomatic (the strongest form of automaticity for monoids).
Other consequences of these results include the fact that
Plactic algebras
of finite rank admit finite Gr\"{o}bner--Shirshov bases,
Plactic monoids of finite rank satisfy the
homological finiteness property $\mathrm{FP}_\infty$, and the
homological finiteness property $\mathrm{FDT}$, and that Plactic algebras
are automaton algebras  in the sense of Ufnarovski; see
\cite{ufnarovski_introduction} or more recently \cite{Okninski2014}.

From the point of view of crystals, these results say that string
rewriting systems and transducers can be used to compute efficiently
with crystals of type $A_n$. Our interest in this paper is to
investigate the extent to which these tools can be applied to other Kashiwara
crystals and crystal monoids.
The results in this article will show that
such tools can be successfully
developed for all of the classical types $A_n$, $B_n$, $C_n$, $D_n$,
and for the exceptional type $G_2$. As in the
case of the classical Plactic monoid the existence of finite complete rewriting systems implies that
these monoids have finite derivation type and satisfy the
homological finiteness properties left and right $\mathrm{FP}_\infty$, and that the corresponding
semigroup algebras are automaton algebras and all admit finite
Gr\"{o}bner--Shirshov bases. Also, the existence of biautomatic structures implies these monoids
all have word problem soluble in quadratic time.

We now give a brief overview of the main ideas, constructions and results that shall be obtained in this paper.
We begin by using some results from the theory of crystal bases to construct finite complete rewriting systems
presenting the Plactic monoids of types $A_n$, $B_n$, $C_n$, $D_n$, and $G_2$.
(We refer the reader forward to \fullref{Subsection}{subsec:srs} for definitions and terminology on rewriting systems.)
We use column generators and our rewriting system has rules that replace an adjacent pair of columns by the unique
tableau that represents their product.  The set of Young tableaux serves as a cross-section of the plactic monoid of type $A_n$: two words in $\aA_n^*$ represent the same element of $\Pl(A_n)$ if and only if they give the same tableau when Schensted's insertion algorithm (see \cite{schensted_longest} and \cite[Ch.~5]{lothaire_algebraic}) is applied to them. The other types of plactic monoids
have analogous (but substantially different) types of tableaux. Any of these tableaux, when read column-by-column from
right to left, yields a word that represents the corresponding element of the monoid.
Thus the columns of a given type are generators for the plactic monoid of that type. Most products of columns are not
tableaux. Following \cite{lecouvey_bndn}, we call an arbitrary product of columns a tabloid. The key to constructing our
rewriting systems and automatic structures is to use column generators and rewrite tabloids to tableaux. More formally,
we consider a pair of columns that form a tabloid that is \emph{not} a tableau. This is the left-hand side of a
rewriting rule. The right-hand side of the corresponding rewriting rule is the unique tableau that represents the same
element of the monoid as this tabloid. Pictorially, rewriting will look like the following:
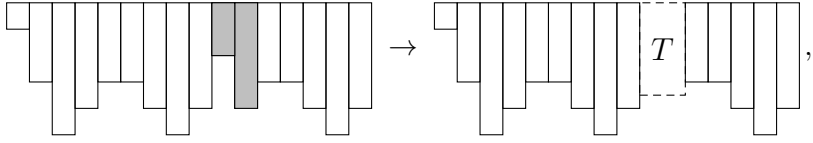
\[
\begin{tikzpicture}[x=3mm,y=3.5mm,baseline=-7mm]
  \foreach \x/\xheight in {0/4,1/5,2/4,3/3,4/3,5/4,6/2,7/4,8/5,9/4,10/3,11/3,12/4,13/5,14/3,15/1} {
    \draw ($ (15,0) - (\x,0) $) rectangle ($ (14,0)-(\x,\xheight) $);
  };
  \draw[fill=lightgray] ($ (15,0) - (5,0) $) rectangle ($ (14,0)-(5,4) $);
  \draw[fill=lightgray] ($ (15,0) - (6,0) $) rectangle ($ (14,0)-(6,2) $);
\end{tikzpicture}
\;\imreduces\;
\begin{tikzpicture}[x=3mm,y=3.5mm,baseline=-7mm]
  \foreach \x/\xheight in {0/4,1/5,2/4,3/3,4/3,7/4,8/5,9/4,10/3,11/3,12/4,13/5,14/3,15/1} {
    \draw ($ (15,0) - (\x,0) $) rectangle ($ (14,0)-(\x,\xheight) $);
  };
  \draw[densely dashed] ($ (15,0) - (5,0) $) rectangle ($ (13,0)-(5,3.5) $);
  \node[font=\large] at (9,-1.75) {$T$};
\end{tikzpicture}\,,
\]
where $T$ is the tableau representing the same element as the two shaded columns. Thus we gradually rewrite a tabloid
towards a product of columns where every adjacent pair of columns forms a tableau; as we shall see, the whole product
then forms a tableau. We prove that this rewriting is terminating by anayzing what shapes of tableaux can result from a product of two columns. For the classical types $A_n$, $B_n$, $C_n$ and $D_n$ this is done by applying the generalized Littlewood--Richardson rule for decomposing tensor products of crystals into connected components (see \cite[Theorem~8.6.6.]{hong_quantumgroups}). The case of $G_2$ is dealt with separately using an analysis of products of columns, working with highest-weight words.

Equipped with our finite complete rewriting systems, we then proceed to prove that the Plactic monoids of types $A_n$,
$B_n$, $C_n$, $D_n$, and $G_2$ are biautomatic.
(We refer the reader forward to \fullref{Subsection}{sec:biautomaticity} for definitions and terminology on automatic semigroups.)
In each case the language of representatives of the biautomatic structure will be the language of irreducible words of the rewriting complete system $(\Sigma, T)$ described above.  To obtain a biautomatic structure, we first investigate what happens when we take a tableau and left multiply by a single generator. We show how the corresponding word over $\Sigma$ can be rewritten by $T$ to an irreducible word by a single left-to-right pass through the word, and that this only changes the length of the word by at most $1$, in all the classical cases $A_n$, $B_n$, $C_n$ and $D_n$, and by at most $2$ in the case of $G_2$. Analogous results are proved for right multiplication by a single generator, although the proofs are more involved than the corresponding results for left multiplication. These results are then used to build biautomatic structures for the plactic monoids of each type. The strategy is to show that the same kind of rewriting occurs when a normal form word, not necessarily of highest weight, is left- or right-multiplied by a generator, and thus that such rewriting can be carried out by a two-tape automaton.

Note that we recover in this paper a new proof of our previous results that classical Plactic monoids (of type $A_n$) can
be presented by finite complete rewriting systems and are biautomatic
\cite{cgm_plactic}. While writing this paper, we came across the work of Hage
\cite{hage_typec}, who independently constructed a finite complete rewriting
system for $\Pl(C_n)$. Hage's approach differs from ours in making use of
Lecouvey's insertion algorithms, whereas we use Lecouvey's presentations. (Hage
does not consider biautomaticity or its consequences.) We should also note that
an alternative approach to obtaining complete rewriting systems for the Plactic
monoids considered in this paper is to apply the results of
Littelmann~\cite[Theorem~B, \S~8]{littelmann_plactic} which he obtained
using his path model. In contrast, as far as the authors are aware, the results
we obtain here are the first to appear in the literature on biautomatic
structures and complexity of the word problem for
plactic monoids of types $A_{\n}$, $B_{\n}$, $C_{\n}$, $D_{\n}$, and $G_2$.
It is important to note that there exist finitely presented monoids which are defined by finite complete rewriting systems but which are not automatic. Indeed, there even exist multihomogeneous finitely presented monoids with this property; see \cite{cgm_homogeneous}. Thus our results on biautomaticity are in no sense immediate consequences of the existence of complete rewriting systems defining these monoids. Indeed, in order to obtain our results on automatic structures, and the corollaries on the complexity of the word problem, we shall need to prove results which give detailed information about how products of columns are rewritten to normal form using the finite complete rewriting systems.

\section{Crystals and plactic monoids}

In this section we will formulate the main concepts that are used throughout the
paper. We will
present Kashiwara's characterization of plactic monoids in
terms of crystal graphs. We first outline a pure
combinatorial abstract theory of crystal monoids, that  avoids delving into
the deep theory underlying crystal graphs, thus providing a general framework
for all the different types of plactic monoids ($A_n$
 $B_n$, $C_n$, $D_n$ and $G_2$). This general theory gives us an abstract
version of known results from
\cite{kashiwara_crystalgraphs,lecouvey_survey,KashiwaraBanff} for the different
types of plactic monoids.
For the
underlying theory of crystal bases we  refer the reader to
\cite{hong_quantumgroups}.

\subsection{Notation}

We denote the empty word (over any alphabet) by $\emptyword$. For an alphabet $X$, we denote by $X^*$ the set of all
words over $X$ including the empty word $\emptyword$. When $X$ is a generating set for a monoid $M$, every element of
$X^*$ can be interpreted either as a word or as an element of $M$. For words $u,v \in X^*$, we write $u=v$ to indicate
that $u$ and $v$ are equal as words and $u=_X v$ to denote that $u$ and $v$ represent the same element of the monoid
$M$. The length of $u \in X^*$ is denoted $|u|$, and, for any $x \in X$, the number of occurences of the symbol $x$ in $u$ is denoted
$|u|_x$.

\subsection{\texorpdfstring{Definition of crystal graph}{Definition of crystal
graph}}
\label{sec:crystalsdef}

For the purposes of this paper, a directed graph with labels from $I$ is a set $V$ of vertices equipped with a set $E$
of triples drawn from $V \times I \times V$. A triple $(v,i,v') \in E$ is interpreted as an edge from the vertex $v$ to
a vertex $v'$ with label $i$. A path starting at $u \in V$ and ending at $w \in V$ is a (possibly empty) sequence of edges $(u,i_0,v_1)$,
$(v_1,i_1,v_2)$, \ldots, $(v_n,i_n,w)$; note that all paths are directed. Notice that vertices and edges may appear
multiple times on a path.

\begin{definition}
A \defterm{crystal basis} is a directed labelled graph with vertex set $X$ and
label set $I$ satisfying the conditions:
\begin{itemize}
\item For all $x \in X$ and $i \in I$, there is at most one edge starting at $x$ labelled by $i$ and at most one edge
  ending at $x$ labelled by $i$.
\item For all $i \in I$, there is no infinite path made up of edges labelled by $i$.
\end{itemize}
\end{definition}
Notice that the second condition implies that a crystal basis cannot contain an
$i$-labelled directed circuit.

(Strictly speaking, such a graph is a graphical description of the representation-theoretic notion of a crystal basis;
see \cite[\S~4.2]{hong_quantumgroups} for details. More precisely, every (integrable highest weight) representation of a
symmetrizable Kac--Moody algebra has a crystal associated to it. However, not every crystal arises from such a
representation. Indeed, there has been research on finding a simple set of local axioms that characterize those crystals
that arise from such representations; see
\cite{stembridge_local,Sternberg2007,Danilov2009}. In fact, the two
conditions above coincide with
axioms (P1) and (P2) in the characterization of the crystal graphs of integrable highest-weight modules for simply-laced
quantum Kac--Moody algebras in \cite{stembridge_local}.)

For each $i\in I$, define partial maps $\e_i$ and $\f_i$ called the
\defterm{Kashiwara operators} on the set $X$ as follows: for each edge $(a,i,b)$, which we will represent graphically as
\[
\begin{tikzpicture}[bigcrystal,labelledcolouredcrystaledges]
  \draw
  (1,0) node (p1) {$a$}
  ++(1,0) node (p2) {$b$};
  \draw[fi] (p1) to (p2);
\end{tikzpicture},
\]
define $\f_i (a) = b$ and $\e_i (b) = a$.

Using the definition of $\e_i$ and $\f_i$, we can build an extended directed
labelled graph:
\begin{definition}
 A \defterm{crystal
  graph} arising from a given crystal basis with
vertex set $X$ and label set $I$, is a directed labelled graph, denoted
$\Gamma_X$,  with vertex set   $X^*$, the free monoid on $X$. The edges are
defined by
partially extending the operators $\e_i$ and $\f_i$ to $X^*$, as follows: for all $u, v \in X^*$ and $i \in I$, define
inductively
\begin{align}
\e_i(uv) &=
\begin{cases}
u\, \e_i(v) & \text{if $\fcount_i(u)<\ecount_i(v)$} \\
\e_i(u)\,v & \text{if $\fcount_i(u) \geq \ecount_i(v)$}
\end{cases}; \label{eq:ei} \\
\f_i(uv) &=
\begin{cases}
\f_i(u)\,v & \text{if $\fcount_i(u)>\ecount_i(v)$} \\
u\,\f_i(v) & \text{if $\fcount_i(u) \leq \ecount_i(v)$}
\end{cases},\label{eq:fi}
\end{align}
where $\ecount_i$ and $\fcount_i$ are auxiliary maps on $X^*$ defined as follows: for $w\in X^*$, let
\begin{align*}
\ecount_i(w) & = \max\gset[\big]{k \in \nset\cup\set{0}}{\text{$\underbrace{\e_i\cdots \e_i}_{\text{$k$ times}}(w)$ is defined}}; \\
\fcount_i(w) & = \max\gset[\big]{k \in \nset\cup\set{0}}{\text{$\underbrace{\f_i\cdots\f_i}_{\text{$k$ times}}(w)$ is defined}}.
\end{align*}
\end{definition}
This extension of the operators $\e_i$ and $\f_i$ to words on
$X^*$ replicates the  properties of the  action of the Kashiwara operators as
in \cite[Theorem~1.1.4]{kashiwara_crystalgraphs}.

For each $i \in I$, define a map $\rho_i : X^* \to \gset{{-}^p{+}^q}{p,q \in
\nset\cup\set{0}}$. (Note that the symbols ${+}$ and
${-}$ here, and in the following discussion, are simply letters in the alphabet $\set{{+},{-}}$.) For a word
$w \in X^*$, define $\rho_i(w)$ to be the word obtained by replacing each symbol $x$ of $w$ by
${-}^{\ecount_i(x)}{+}^{\fcount_i(x)}$, then iteratively deleting subwords ${+}{-}$ until a word of the form
${-}^p{+}^q$ remains.
Note further that each symbol ${+}$ or ${-}$ in the computed word $\rho_i(w)$ is
a symbol that `survives' from the
original replacement of symbols $x$ by ${-}^{\ecount_i(x)}{+}^{\fcount_i(x)}$. Furthermore, each symbol ${+}$ or ${-}$
in $\rho_i(w)$ is contributed by a uniquely determined symbol of $w$ (since two subwords ${+}{-}$
cannot partially overlap with each other).

The following result shows the connection between $\rho_i$ and the action of the operators $\e_i$ and $\f_i$.


For classical Lie algebras the properties on operators given in the following result may be found in \cite{kashiwara_crystalgraphs}. The generalisation below is proved in a similar way, directly from the definitions above, so we omit the proof.

\begin{proposition}
\label{prop:computingef}
Let $w = w_1\cdots w_k$, where $w_h \in X$, and $i \in I$. Then
\begin{enumerate}
\item \begin{enumerate}
  \item $\e_i(w)$ is defined if and only if $\rho_i(w)$ contains at least one symbol ${-}$.
  \item If $\e_i(w)$ is defined, $\e_i(w) = w_1\cdots w_{j-1}\e_i(w_j)w_{j+1}\cdots w_k$, where $w_j$ is the symbol
    that contributed the rightmost symbol ${-}$ in $\rho_i(w)$.
  \item If $\e_i(w)$ is defined, $w = \f_i(\e_i(w))$.
  \end{enumerate}
\item \begin{enumerate}
  \item $\f_i(w)$ is defined if and only if $\rho_i(w)$ contains at least one symbol ${+}$.
  \item If $\f_i(w)$ is defined, $\f_i(w) = w_1\cdots w_{j-1}\f_i(w_j)w_{j+1}\cdots w_k$, where $w_j$ is the symbol
    that contributed the leftmost symbol ${+}$ in $\rho_i(w)$.
  \item If $\f_i(w)$ is defined, $w = \e_i(\f_i(w))$.
  \end{enumerate}
\item $\rho_i(w) = {-}^{\ecount_i(w)}{+}^{\fcount_i(w)}$.
\end{enumerate}
Furthermore, the actions of the operators $\e_i$ and $\f_i$ are
well-defined.
\end{proposition}

The previous proposition gives the following
practical method, first described
in \cite{kashiwara_crystalgraphs}, for computing the actions of $\e_i$ and $\f_i$ on a word $w \in X^*$: Compute
$\rho_i(w)$ by writing down the word obtained by replacing each symbol $x$ by ${-}^{\ecount_i(x)}{+}^{\fcount_i(x)}$ and
then deleting subwords ${+}{-}$. The resulting word will have the form ${-}^{\ecount_i(w)} {+}^{\fcount_i(w)}$. If
$\ecount_i(w)=0$, then $\e_i(w)$ is undefined. If $\ecount_i(w)>0$ then we obtain $\e_i(w)$ by taking the symbol $x$
that contributed the rightmost $-$ of $\rho_i(w)$ and changing it to $\e_i(x)$. If $\fcount_i(w)=0$, then $\f_i(w)$ is
undefined. If $\fcount_i(w)>0$ the we obtain $f_i(w)$ by taking the symbol $x$ that contributed the leftmost $+$ of
$\rho_i(w)$ and changing it to $\f_i(x)$.

%
%
%
%

Notice in particular that if, during the deletion of subwords ${+}{-}$, the word that we obtain begins with ${-}$, then
this symbol ${-}$ will remain in place throughout all subsequent deletions, and so $\ecount_i(w)>0$, and so $\e_i(w)$ is
defined. This observation is important, and we will use it repeatedly throughout the paper. (There is a dual observation
for words ending in $+$ implying that $\f_i(w)$ is defined, but we will not need this.)

In the crystal graph, we have an edge from $w$ to $w'$ labelled by $i$ if and only if $w' = \f_i(w)$ (or, equivalently, $w
= \e_i(w')$). Note that $\ecount_i(u)$ is the length of the longest path consisting of edges labelled by $i$ that ends
at $u$. Dually, $\fcount_i(u)$ is the length of the longest path consisting of edges labelled by $i$ that starts at $u$.

\subsection{Weights}

In our abstract combinatorial setting we have the following definition:
\begin{definition}
A \defterm{weight function} is  a
homomorphism $\wt : X^* \to P$, where $P$ is
some monoid (called the \defterm{weight monoid}) such that there is a partial order $\leq$ on $P$ (not necessarily
compatible with multiplication in $P$) with the following property: for all $u
\in X^*$ and $i \in I$,
\begin{itemize}
 \item  if $\e_i(u)$ is
defined, then $\wt(u) < \wt\parens[\big]{\e_i(u)}$; and
\item  if $\f_i(u)$ is defined, then
$\wt\parens[\big]{\f_i(u)} < u$.
\end{itemize}
\end{definition}

Let $u,v \in X^*$. The word $u$ has \defterm{higher weight} than the word $v$ (or, equivalently, the word $v$ has
\defterm{lower weight} than the word $u$) if $\wt(v) < \wt(u)$. Thus the operators $\e_i$, when defined, always yield a
word of higher weight, and the operators $\f_i$, when defined, always yield a word of lower weight.

The abstract definitions of weight monoid and weight functions given here
are more general than in the literature. In the context of Lie algebras
representations the weight is a linear map from the vertices of the crystal
components (which is identified with the set of words from $X^*$) to the weight
lattice generated by the fundamental weights $\Lambda_1,\ldots, \Lambda_n$. This
weight lattice can be identified (up to isomorphism) with $\zset^{n}$. For the
root system of type $A_n$, the partial order on $\zset^{n}$ is the so-called
\emph{dominance order} on the set of partitions (see
\cite[\S~I.1]{Macdonald1979}).

In the remainder of the paper, we will not need to explicitly compare orders: we
simply use the fact that $\e_i$, when
defined, raise weight, and $\f_i$, when defined, lowers weight.

In the crystal graph $\Gamma_X$, a vertex that has maximal weight within a
particular component is called a
\defterm{highest-weight vertex}. (In the specific crystal graphs we consider later, it will turn out that each component
contains a \emph{unique} highest-weight vertex.)


\begin{lemma}[{\cite[Lemma 5.3.1]{lecouvey_survey}}]
\label{lem:highestweightfactors}
For any words $w_1,w_2 \in X^*$, the word $w_1 w_2$ is a vertex
of highest weight of a connected component of the crystal graph $\Gamma_{X} $ if and only if:
\begin{enumerate}
 \item $ w_1$ is a vertex of highest weight (that is, $\ecount_i( w_1) = 0$);
 \item for all $i = 1,\ldots,n$ we have $\ecount_i(w_2)\leq \fcount_i (w_1)$ .
\end{enumerate}
\end{lemma}

\subsection{Relations from crystal graphs}
\label{sec:relfromcrystals}

For any word $w \in X^*$, let $B(w)$ be the connected component of the crystal graph containing the vertex $w$. A
\defterm{crystal isomorphism} is a bijection $\fcount$ between two connected components $B(w)$ and $B(w')$ that maps
directed edges labelled by $i$ to directed edges labelled by $i$ (in the sense that if $(x,i,y)$ is an edge in $B(w)$,
then $(\fcount(x),i,\fcount(y))$ is an edge in $B(w')$), sends non-edges to non-edges, and preserves weights (in the sense
that $\wt(u) = \wt(\fcount(u))$ for any $u \in B(w)$). If there is a crystal isomorphism between $B(w)$ and $B(w')$, we say
that $B(w)$ and $B(w')$ are isomorphic.

We say $u \in B(w)$ and $v \in B(w')$ lie in the \defterm{same position} of isomorphic components $B(w)$ and $B(w')$ if
there is an isomorphism between $B(w)$ and $B(w')$ that maps $u$ to $v$; this is
denoted by $u \sim v$. This general abstract setting is sufficient to
obtain a congruence.
For classical
crystals this result is well-known (see \cite{lecouvey_survey}
for a survey).

\begin{proposition}
  \label{prop:isomdefinescong}
  The relation $\sim$ is a congruence on the free monoid $X^*$.
\end{proposition}

\begin{definition}
  Let $X$ be an alphabet forming the vertex set of a crystal basis, $\wt: X^* \rightarrow P$ a weight function, and
  $\sim$ the congruence on $X^*$ that relates two words if they lie in the same position of isomorphic components of the
  crystal graph $\Gamma_X$. Then we call $X^*/{\sim}$ the \emph{crystal monoid} determined by the crystal $\Gamma_X$
  with weight function $\wt$ and weight monoid $P$.
\end{definition}

Note that if multiplication in $P$ is algorithmically computable, then the weights of words in $X^*$ are computable. If
the crystal basis is finite (and so the crystal monoid is finitely generated), then it is possible to compute the
connected component of any word in $X^*$. If both these conditions hold, then we can decide whether two components are
isomorphic, and thus check whether two words are $\sim$-related. In short, we have the following:

\begin{proposition}
  \label{prop:wordproblemcrystal}
  If a crystal monoid arises from a finite crystal basis, and has a weight monoid in which multiplication is computable,
  then it has soluble word problem.
\end{proposition}

In particular, when the weight monoid $P$ is (isomorphic to) a free
abelian group of finite rank (which it will be in all the specific
examples we consider below) then \fullref{Proposition}{prop:wordproblemcrystal} applies, and the crystal monoid will
have soluble word problem. Notice, however, that this result says nothing about the \emph{complexity} of the word problem. We
will see that $\Pl(A_n)$ and the plactic monoids of other types, which we will define shortly, are all biautomatic and
thus have word problem soluble in quadratic time \cite[Corollary~3.7]{campbell_autsg}.

\subsection{\texorpdfstring{Crystal graphs of types $A_n$, $B_n$, $C_n$, $D_n$ and $G_2$}{Crystal graphs of types Bn, Cn, Dn and G2}}

The plactic monoid (of type $A_n$) parameterizes representations of the
$q$-analogue of the universal enveloping algebra of the semisimple Lie algebras of type $A_n$. There are analogous
plactic monoids of types $B_n$, $C_n$, $D_n$, and $G_2$, parameterizing representations of the $q$-analogues of the
universal enveloping algebras of the semisimple Lie algebras of the
corresponding types.

In our combinatorial abstract framework, a crystal graph is constructed
from a crystal basis  following the rules given by the
 action of classical Kashiwara operators. In turn, for all the classical types,
this action is given by the simple tensor rule, and thus
the crystal graph within this combinatorial abstract setting corresponds to
the classical crystal graph arising as  tensor powers of the corresponding
basis.

If we fix the  crystal basis of type $A_n$ to be the irreducible root
system of type $A_n$ of the representation of the $q$-analogue of the
classical Lie Algebra of that type, and weight function as in
\cite[\S~3.3]{lecouvey_survey}, we will obtain
the plactic monoid of
type $A_n$.
All other plactic monoid types arise as
crystal monoids in the same way, but starting from different
crystal bases and definitions of the weight function.


For all the classical types,
the weight functions arise from the root systems
of the corresponding Lie algebras as detailed in
\cite[\S~3.3]{lecouvey_survey}.

\subsubsection{\texorpdfstring{Type $A_n$}{Type An}}

For type $A_n$ we consider the ordered alphabet
\[
\aA_n = \set{1 < 2 < \ldots < n}.
\]
The crystal basis for type $A_n$ is:
\begin{equation}
\label{eq:an:crystalbasis}
\begin{tikzpicture}[bigcrystal,labelledcolouredcrystaledges,baseline=(p1.base)]
  \draw
  (1,0) node (p1) {$1$}
  ++(1,0) node (p2) {$2$}
  ++(1,0) node (pdots) {$\mathrlap{\phantom{1}}\ldots$}
  ++(1,0) node (pn) {$n{-}1$}
  ++(1,0) node (pnplus1) {$n$};
  \draw[f1] (p1) to (p2);
  \draw[f2] (p2) to (pdots);
  \draw[fn2] (pdots) to (pn);
  \draw[fn1] (pn) to (pnplus1);
\end{tikzpicture}
\end{equation}
This graph has vertex set $\aA_n$ and labels from the set $\set{1,\ldots,n-1}$.
The resulting graph is the \defterm{crystal graph of type $A_n$}, denoted
$\Gamma_{A_n}$, and the monoid that arises is the \defterm{plactic
  monoid of type $A_n$}, denoted $\Pl(A_n)$.

\subsubsection{\texorpdfstring{Type $B_n$}{Type Bn}}

For type $B_n$ we consider the ordered alphabet
\[
\aB_n = \set[\big]{1 < 2 < \ldots < n < 0 < \bar{n} < \ldots < \bar{2} < \bar{1}}.
\]
Note that $0$ is greater than $n$. The crystal basis for type $B_n$ is:
\[
\begin{tikzpicture}[bigcrystal,labelledcolouredcrystaledges]
  \draw
  (1,0) node (p1) {$1$}
  ++(1,0) node (p2) {$2$}
  ++(1,0) node (pdots) {$\mathrlap{\phantom{1}}\ldots$}
  ++(1,0) node (pn) {$\mathrlap{\phantom{1}}n$}
  ++(1,0) node (p0) {$0$}
  ++(1,0) node (pnbar) {$\mathrlap{\phantom{1}}\smashbar{n}$}
  ++(1,0) node (pdotsbar) {$\mathrlap{\phantom{1}}\ldots$}
  ++(1,0) node (p2bar) {$\smashbar{2}$}
  ++(1,0) node (p1bar) {$\smashbar{1}$};
  \draw[f1] (p1) to (p2);
  \draw[f2] (p2) to (pdots);
  \draw[fn1] (pdots) to (pn);
  \draw[fn] (pn) to (p0);
  \draw[fn] (p0) to (pnbar);
  \draw[fn1] (pnbar) to (pdotsbar);
  \draw[f2] (pdotsbar) to (p2bar);
  \draw[f1] (p2bar) to (p1bar);
\end{tikzpicture}
\]
The resulting graph is the \defterm{crystal graph of type $B_n$}, denoted
$\Gamma_{B_n}$, and the monoid that arises is the \defterm{plactic
  monoid of type $B_n$}, denoted $\Pl(B_n)$.


\subsubsection{\texorpdfstring{Type $C_n$}{Type Cn}}

For type $C_n$ we consider the ordered alphabet
\[
\aC_n = \set[\big]{1 < 2 < \ldots < n < \bar{n} < \bar{n-1} < \ldots < \bar{1}}.
\]
The crystal basis for type $C_n$ is:
\[
\begin{tikzpicture}[bigcrystal,labelledcolouredcrystaledges]
  \draw
  (1,0) node (p1) {$1$}
  ++(1,0) node (p2) {$2$}
  ++(1,0) node (pdots) {$\mathrlap{\phantom{1}}\ldots$}
  ++(1,0) node (pn) {$\mathrlap{\phantom{1}}n$}
  ++(1,0) node (pnbar) {$\mathrlap{\phantom{1}}\smashbar{n}$}
  ++(1,0) node (pdotsbar) {$\mathrlap{\phantom{1}}\ldots$}
  ++(1,0) node (p2bar) {$\smashbar{2}$}
  ++(1,0) node (p1bar) {$\smashbar{1}$};
  \draw[f1] (p1) to (p2);
  \draw[f2] (p2) to (pdots);
  \draw[fn1] (pdots) to (pn);
  \draw[fn] (pn) to (pnbar);
  \draw[fn1] (pnbar) to (pdotsbar);
  \draw[f2] (pdotsbar) to (p2bar);
  \draw[f1] (p2bar) to (p1bar);
\end{tikzpicture}
\]
The resulting graph is the \defterm{crystal graph of type $C_n$}, denoted
$\Gamma_{C_n}$, and the monoid that arises is the \defterm{plactic
  monoid of type $C_n$}, denoted $\Pl(C_n)$.


\subsubsection{\texorpdfstring{Type $D_n$}{Type Dn}}

For type $D_n$ we consider the ordered alphabet
\[
\aD_n = \set[\big]{1 < 2 < \ldots < n-1 < \begin{array}{c}\bar{n}\\n\end{array} < \bar{n-1} < \ldots < \bar{2} < \bar{1}};
\]
note that $n$ and $\bar{n}$ are incomparable and that $n-1 < n < \bar{n-1}$ and $n-1 < \bar{n} < \bar{n-1}$. The crystal
basis for type $D_n$ is:
\[
\begin{tikzpicture}[bigcrystal,labelledcolouredcrystaledges]
  \draw
  (1,0) node (p1) {$1$}
  ++(1,0) node (p2) {$2$}
  ++(1,0) node (pdots) {$\mathrlap{\phantom{1}}\ldots$}
  ++(1,0) node (pnminus1) {$\mathrlap{\phantom{1}}n-1$}
  ++(.707,-.707) node (pnbar) {$\smashbar{n}$}
  ++(0,1.414) node (pn) {$n$}
  ++(.707,-.707) node (pnminus1bar) {$\mathrlap{\phantom{1}}\smashbar{n-1}$}
  ++(1,0) node (pdotsbar) {$\mathrlap{\phantom{1}}\ldots$}
  ++(1,0) node (p2bar) {$\smashbar{2}$}
  ++(1,0) node (p1bar) {$\smashbar{1}$};
  \draw[f1] (p1) to (p2);
  \draw[f2] (p2) to (pdots);
  \draw[fn2] (pdots) to (pnminus1);
  \draw[fn1] (pnminus1) to (pn);
  \draw[fn] (pnminus1) to (pnbar);
  \draw[fn] (pn) to (pnminus1bar);
  \draw[fn1] (pnbar) to (pnminus1bar);
  \draw[fn2] (pnminus1bar) to (pdotsbar);
  \draw[f2] (pdotsbar) to (p2bar);
  \draw[f1] (p2bar) to (p1bar);
\end{tikzpicture}
\]
The resulting graph is the \defterm{crystal graph of type $D_n$}, denoted
$\Gamma_{D_n}$, and the monoid that arises is the \defterm{plactic
  monoid of type $D_n$}, denoted $\Pl(D_n)$.


\subsubsection{\texorpdfstring{Type $G_2$}{Type G2}}

For type $G_2$ we consider the ordered alphabet
\[
\aG_2 = \set[\big]{1 < 2 < 3 < 0 < \bar{3} < \bar{2} < \bar{1}}.
\]
The crystal basis for type $G_2$ is:
\begin{equation}
\label{eq:g2:crystalbasis}
\begin{tikzpicture}[bigcrystal,labelledcolouredcrystaledges,baseline=(p1.base)]
  \foreach \s in {1,2,3} { \node (p\s) at ($ (\s,0) $) {$\s$}; };
  \node (p4) at (4,0) {$0$};
  \foreach \s in {5,6,7} {
    \path let \n1 = {int(8-\s)} in node (p\s) at ($ (\s,0)$) {$\smashbar{\n1}$};
  };
  \foreach \s/\slabel in {1/1,2/2,3/1,4/1,5/2,6/1} {
    \draw[f\slabel] let \n1 = {int(\s+1)} in (p\s) to (p\n1);
  }
\end{tikzpicture}
\end{equation}
The resulting graph is the \defterm{crystal graph of type $G_2$},
denoted $\Gamma_{G_2}$, and the monoid that arises is the \defterm{plactic
  monoid of type $G_2$}, denoted $\Pl(G_2)$.


\subsection{\texorpdfstring{Properties of crystal graphs of types $A_n$, $B_n$, $C_n$, $D_n$ and $G_2$}{Properties of crystal graphs of types An, Bn, Cn, Dn and G2}}

Let $X$ be one of the types $A_n$, $B_n$, $C_n$, $D_n$ or $G_2$, and let $\aX$ be the corresponding alphabet $\aA_n$,
$\aB_n$, $\aC_n$, $\aD_n$ or $\aG_2$. As described above, we have a crystal graph $\Gamma_{X}$ and a plactic monoid
$\Pl(X)$ of each of the given types. For clarity and brevity in explanations, define, for all $x,y \in \aX$ with
$x \leq y$,
\[
\aX[x,y] = \gset{z \in \aX}{x \leq z \leq y}.
\]
Recall that the Kashiwara operators $\e_i$ and $\f_i$ respectively raise and lower weights whenever they are defined.

An important and non-obvious fact for us will be that each connected component of a crystal graph $\Gamma_X$ contains a
\emph{unique} highest-weight vertex \cite[\S~3.1]{lecouvey_survey}. (It is not true for crystal monoids in general that
the connected components of the crystal  have unique highest-weight vertices.)
For any word $w \in \aX^*$, denote by
$w^0$ the unique highest-weight vertex in $B(w)$. Thus there exist $i_1, \ldots, i_r \in \set{1,\ldots,n}$ such that
$w^0 = \e_{i_1} \ldots \e_{i_r}(w)$, or, equivalently $w = \f_{i_r} \ldots \f_{i_1}(w^0)$.

Notice that for $\Gamma_{X}$, we have $u \sim v$ if and only if $u^0 \sim v^0$ and there exist
$i_1, \ldots, i_r \in \set{1,\ldots,n}$ such that
\[
u = \f_{i_r} \cdots \f_{i_1}(u^0) \text{ and } v = \f_{i_r} \cdots \f_{i_1}(v^0).
\]

\section{Tableaux and tabloids}


In this section we give the necessary background on tableaux theory for plactic
monoids of types $A_n$, $B_n$, $C_n$,
$D_n$, and $G_2$, that will be frequently used in the sequel; see
\cite{kashiwara_crystalgraphs} and \cite{lecouvey_survey} for further details.

\subsection{Young tableaux and columns}
\label{subsec:tableaux_cols}

A Young diagram  $Y$ (of shape $\lambda$) associated to a partition  $\lambda=(\lambda_1,\ldots,\lambda_k)$ is a finite array of left-justified boxes whose $i$-th row has length $\lambda_i$.  A \defterm{Young tableau} $T$ of shape $\lambda$ is  a filling of a  Young diagram by  symbols from the fixed alphabet  such that
(i) the entries of any column strictly increase from top to bottom, and
(ii) the entries along each row weakly increase from left to right.

A \defterm{column} (of type $A_n$) is a tableau of column shape $\lambda=(1,\ldots,1)$:
\[
\tikz[tableau]\matrix {x_1 \\ x_2 \\ |[dottedentry]| \null \\x_k\\};
\]

A \defterm{column} of type $B_n$, $C_n$ and $D_n$ is, respectively, a Young diagram of column shape of the form
\[
\tikz[tableau]\matrix { |[minimum height=10mm]| {\beta_+} \\ |[minimum height=10mm]| {\beta_0}\,  \\ |[minimum height=10mm]| {\beta_-}\\};\,, \qquad
\tikz[tableau]\matrix { |[minimum height=10mm]| \gamma_+ \\ |[minimum height=10mm]| \gamma_-\\};\,,\qquad \mbox{and}\qquad
\tikz[tableau]\matrix { |[minimum height=10mm]| \delta_+ \\ |[minimum height=10mm]| \delta \\ |[minimum height=10mm]| \delta_-\\};\,,
\]
where
\begin{itemize}
\item $\beta_+$ is filled with symbols from $\aB_n[1,n]$, and is strictly increasing from top to bottom;
\item $\beta_0$ is filled with symbols $0$;
\item $\beta_-$ is filled with symbols from $\aB_n[\bar{n},\bar{1}]$, and is strictly increasing from top to bottom;
\item $\gamma_+$ is filled with symbols from $\aC_n[1,n]$, and is strictly increasing from top to bottom;
\item $\gamma_-$ is filled with symbols from $\aC_n[\bar{n},\bar{1}]$, and is strictly increasing from top to bottom;
\item $\delta_+$ is filled with symbols from $\aD_n[1,n-1]$, and is strictly increasing from top to bottom;
\item $\delta$ is filled with symbols $n$ and $\bar{n}$, with different symbols in vertically adjacent cells.
\item $\delta_-$ is filled with symbols from $\aD_n[\bar{n-1},\bar{1}]$, and is strictly increasing from top to bottom.
\end{itemize}
A \defterm{column} of type $G_2$ is a Young tableau with entries from $\aG_2$, of column shape, of one of the
following three forms:
\[
\tikz[tableau]\matrix {a \\};,\quad
\tikz[tableau]\matrix {a \\ b \\};\text{ with $a < b$,}\quad\text{or }
\tikz[tableau]\matrix {0 \\ 0 \\};.
\]

The \defterm{height} $h(\beta)$ of a column $\beta$ (of any type) is the number of boxes in the column. The
\defterm{reading} $w(\beta)$ of a column is the word obtained by reading the sequence of symbols in the boxes from top
to bottom. We identify a column with its reading. A word is a \defterm{column word} if it is the reading of a (necessarily
unique) column.

\subsubsection{Admissible columns}
\label{subsubsec:admissible_cols}

Let $\beta$ be a column (of any type) and let $z\leq n$. We denote by $N_\beta(z)$ the number of symbols $x$ in $\beta$
such that $x\leq z$ or $\bar{z}\leq x$.

A column $\beta$ is \defterm{admissible} if each of the following conditions is satisfied:
\begin{enumerate}
 \item $N_\beta(z) \leq z$, for any $z\leq n$; 
 \item if $\beta$ is of type $B_n$ and $0$ is in $\beta$, then $h(\beta)\leq n$;
 \item if $\beta=\tikz[tableau]\matrix {a \\ b \\};$ is of type $G_2$ and height $2$, then
   \[
   \begin{cases}
     \dist(a,b) \leq 2 & \text{for $a \in \set{1,0}$,} \\
     \dist(a,b) \leq 3 & \text{otherwise,}
   \end{cases}
   \]
   where $\dist(a,b)$ is the number of arrows between $a$ and $b$ in the crystal basis \eqref{eq:g2:crystalbasis} for $G_2$.
\end{enumerate}
Note that all columns of type $A_n$ are admissible.

The following is a complete list of all twenty-one admissible columns of type $G_2$:
\begin{equation}
\label{eq:g2admissiblecolumns}
\begin{aligned}
\set[\Bigg]{
&\tikz[tableau]\matrix {1\\};,\;
\tikz[tableau]\matrix {2\\};,\;
\tikz[tableau]\matrix {3\\};,\;
\tikz[tableau]\matrix {0\\};,\;
\tikz[tableau]\matrix {\bar{3}\\};,\;
\tikz[tableau]\matrix {\bar{2}\\};,\;
\tikz[tableau]\matrix {\bar{1}\\};,\;
\tikz[tableau]\matrix {1\\2\\};,\;
\tikz[tableau]\matrix {1\\3\\};,\;
\tikz[tableau]\matrix {2\\3\\};,\;
\tikz[tableau]\matrix {2\\0\\};,\\
&\phantom{\tikz[tableau]\matrix {1\\};,}\;\,
\tikz[tableau]\matrix {2\\\bar{3}\\};,\;
\tikz[tableau]\matrix {0\\\bar{3}\\};,\;
\tikz[tableau]\matrix {3\\\bar{3}\\};,\;
\tikz[tableau]\matrix {3\\0\\};,\;
\tikz[tableau]\matrix {3\\\bar{2}\\};,\;
\tikz[tableau]\matrix {0\\\bar{2}\\};,\;
\tikz[tableau]\matrix {\bar{3}\\\bar{2}\\};,\;
\tikz[tableau]\matrix {\bar{3}\\\bar{1}\\};,\;
\tikz[tableau]\matrix {\bar{2}\\\bar{1}\\};,\;
\tikz[tableau]\matrix {0\\0\\};\,
}.
\end{aligned}
\end{equation}

An \defterm{admissible column word} is a word that is the reading of a (necessarily unique) admissible column.

\subsubsection{\texorpdfstring{The functions $\ell$ and $r$}{The functions l and r}}

We say that a column $\beta$ contains a pair $(z,\bar{z})$ if both symbols $z$ and $\bar{z}$ appear in $\beta$, or if $\beta$ is of type
$B_n$ and $0$ appears in $\beta$. In the following paragraphs we define partial functions $\ell$ and $r$ on the set of
columns of some type. The resulting columns $\ell(\beta)$ and $r(\beta)$, when defined, do not contain pairs $(z,\bar{z})$. For
simplicity and uniformity, for columns of type $A_n$ we define $r(\beta)=\ell(\beta)=\beta$.

Let $\beta$ be a column of type $B_n$ or $C_n$ and let $I_\beta = \set{ z_s < \ldots < z_{r+1} < z_r=0, \ldots, z_1=0 }$
be the set of symbols $z$ for which $\beta$ contains the pair $(z,\bar{z})$.  We say that a column $\beta$ of type $B_n$ or
$C_n$ can be \defterm{split} if there exists a set $J_\beta$ of symbols $t_s < \cdots < t_1$ such that
\begin{itemize}
\item $t_1$ is maximal such that $t_1 < z_1$ and the symbols $t_1$ and $\bar{t_1}$ do not appear in $\beta$;
\item for $i=2, \ldots, s$, the symbol $t_i$ is maximal such that $t_i < \min\set{t_{i-1},z_i}$, $t_i \not\in \beta$, and $\overline{t_i} \not\in \beta$.
\end{itemize}

If $\beta$ can be split, $r(\beta)$ is obtained from $\beta$ by replacing $\overline{z_i}$ with $\overline{t_i}$ for each $i$, and $\ell(\beta)$ is
obtained from $\beta$ by replacing $z_i$ with $t_i$ for each $i$, always reordering to obtain a column if necessary
(c.f. \cite[Example 3.1.7]{lecouvey_bndn}).



The operators $r$ and $\ell$ defined for columns of type $B_n$ can be extended to columns of type $D_n$ as follows: for
any $D_n$ column $\beta$, let $\beta_0$ be the column obtained by replacing all subwords $\bar{n}n$ by $00$ in
$\beta$. Note that $\beta_0$ is always a $B_n$ column. Let $r(\beta)$ and $\ell(\beta)$ be $r(\beta_0)$ and
$\ell(\beta_0)$ (as defined for type $B_n$ columns). Observe that if $\beta$ is a type $D_n$ column that does not
contain a subword $\bar{n}n$, it is also a $B_n$ column and $\beta_0 = \beta$ and so the definitions of $r(\beta)$
and $\ell(\beta)$ coincide regardless of whether $\beta$ is viewed as a column of type $B_n$ or $D_n$.

A column $\beta$ of type $B_n$, $C_n$ or $D_n$ is admissible if and only if both $r(\beta)$ and $\ell(\beta)$ are
defined \cite[Proposition 4.3.3]{lecouvey_survey} (see also
\cite{Sheats1999} for type $C_n$). This fact will be important
in the definition of tableaux in the
following subsection.

\subsection{Tabloids and tableaux}

Let $X$ be one of the types $A_n$, $B_n$, $C_n$, $D_n$ or $G_2$. A \defterm{tabloid} of type $X$ is a
sequence of admissible columns $\beta_r,\ldots, \beta_1$ of type $X$, which we write in a planar form by writing each
column vertically beside each other in the order $\beta_r,\ldots, \beta_1$ from left to right.


For brevity, we also use the inline form
$\tikz[tableau]\matrix{\beta_r \& |[dottedentry]| \& \beta_1\\};$ to denote the tableau with columns $\beta_r,\ldots, \beta_1$. The \defterm{reading}
$w(T)$ of a tabloid $T=\tikz[tableau]\matrix{\beta_r \& |[dottedentry]| \& \beta_1\\};$ is the word $w(\beta_1)\cdots w(\beta_r)$. Note that the
columns of the tabloid are read from rightmost to leftmost, and each column is read from top to bottom.


Note that different tabloids may have the same reading.


For any word $u \in \aX^*$ there is at least one tabloid whose reading is $u$: if $u = u_1\cdots u_k$, where $u_i \in \aX$, then the
tabloid $\tableau{u_k \& |[dottedentry]| \& u_1\\}$ has reading $u$. (Notice that each column $\tableau{u_i\\}$ (of height
$1$) is admissible.)

We now define a relation $\preceq$ on the sets of admissible columns of each type. For types $A_n$, $B_n$, $C_n$, and
$D_n$, the definition proceeds as follows: for two admissible columns $\beta_1$ and $\beta_2$, define
\begin{itemize}
\item $\beta_2 \leq \beta_1$ if $h(\beta_2) \geq h(\beta_1)$ and the rows of the tabloid \tikz[tableau]\matrix{\beta_2 \& \beta_1\\}; are
  weakly increasing from left to right;
\item $\beta_2 \preceq \beta_1$ if $r(\beta_2) \leq \ell(\beta_1)$.
\end{itemize}
Note that for any admissible column $\beta$, we have $\ell(\beta) \leq \beta \leq r(\beta)$; hence $\beta_2 \preceq \beta_1$ implies $\beta_2 \leq \beta_1$.

For type $G_2$, the definition is more complicated: for columns $\beta_1$ and $\beta_2$, define
\begin{align*}
\tikz[tableau]\matrix{a\\}; \preceq \tikz[tableau]\matrix{b\\}; \iff{}& (a \leq b) \land \bigl((a,b) \neq (0,0)\bigr) \\
\tikz[tableau,topalign]\matrix{a\\b\\}; \preceq \tikz[tableau]\matrix{c\\}; \iff{}& (a \leq c) \land \bigl((a,c) \neq (0,0)\bigr) \\
\tikz[baseline=-1mm]\draw[white] (0,5mm);\smash{\tikz[tableau]\matrix{a\\b\\}; \preceq \tikz[tableau]\matrix{c\\d\\};} \iff{}&
(a \leq c) \land \bigl((a,c) \neq (0,0)\bigr) \\
&\land (b \leq d) \land \bigl((b,d) \neq (0,0)\bigr) \\
&\land \bigl(a \in \set{2,3,0} \implies \dist(a,d) \geq 3\bigr) \\
&\land \bigl(a = \bar{3} \implies \dist(a,d) \geq 2\bigr)
\end{align*}

Note that the relation $\preceq$ is transitive and antisymmetric, but is not reflexive in general.

Let $\beta_1$, $\beta_2$ be columns of type $D_n$ such that $h(\beta_2)\geq h(\beta_1)$. We say that the tabloid $\tikz[tableau]\matrix{\beta_2 \& \beta_1\\};$ contains an \defterm{$a$-configuration}, with $a\notin \set{\bar{n},n}$, if:
\begin{itemize}
 \item $a=x_p$, $\bar{n}=x_r$ are symbols of $\beta_2$ and $\bar{a}=y_s$, $n=y_q$ symbols of $\beta_1$; or
 \item $a=x_p$, ${n}=x_r$ are symbols of $\beta_2$ and $\bar{a}=y_s$, $\bar{n}=y_q$ symbols of $\beta_1$
\end{itemize}
where the integers $p,q,r,s$ are such that $p\leq q < r\leq s$. Denote by $\mu(a)$ the integer defined by $\mu(a)=s-p$.

A \defterm{tableau} of type $A_n$, $B_n$, $C_n$ or $G_2$ is a tabloid $\tikz[tableau]\matrix{\beta_r \& |[dottedentry]| \& \beta_1\\};$
such that $\beta_{i+1} \preceq \beta_{i}$ for all $i=1,\ldots,r-1$.  A tableau of type $D_n$ is a tabloid
$\tikz[tableau]\matrix{\beta_r \& |[dottedentry]| \& \beta_1\\};$ such that $\beta_{i+1} \preceq \beta_{i}$ and the tabloid $\tikz[tableau]\matrix{r(\beta_{i+1}) \& \ell
  (\beta_{i}) \\};$ does not contain an $a$-configuration with $\mu(a) = n- a$, for all $i=1,\ldots,r-1$.

\begin{lemma}
  \label{lem:efonfactors}
  Let $T = \tikz[tableau]\matrix{\beta_m \& |[dottedentry]| \& \beta_1\\};$ be a tabloid of type $A_n$, $B_n$, $C_n$,
  $D_n$, and $G_2$. Let $i \in \set{1, \ldots, n}$. Let $u_j = w(\beta_j)$ for $j \in {1,\ldots,m}$, so that
  $w(T) = u_1\cdots u_m$. Suppose $v = \f_i(w(T))$ (respectively, $v = \e_i(w(T))$) is defined. Factor $v$ as
  $v = v_1\cdots v_m$, where $|v_j| = |u_j|$. Then:
  \begin{enumerate}
  \item There exists some $k \in \set{1, \ldots, m}$ such that $v_j=u_j$ for $j \neq k$ and $v_k = \f_i(u_k)$ (respectively, $v_k = \e_i(u_k)$).
  \item Each word $v_j$ is an admissible column word, and so $v$ is the reading of the tabloid
    $\tikz[tableau]\matrix{\gamma_m \& |[dottedentry]| \& \gamma_1\\};$.
  \item For all $j\in \set{1,\ldots, m-1}$, we have $\beta_{j+1}\preceq \beta_{j}$ if and only if $\gamma_{j+1}\preceq \gamma_{j}$. In
    particular, $T$ is a tableau if and only if $\tikz[tableau]\matrix{\gamma_m \& |[dottedentry]| \& \gamma_1\\};$ is a tableau.
  \end{enumerate}
\end{lemma}

\begin{proof}
  See \cite{kashiwara_crystalgraphs} for types $A_n$, $B_n$, $C_n$, and $D_n$; see \cite{lecouvey_survey} for type $G_2$.
\end{proof}

In light of the preceding lemma, we can think of applying the operators $\e_i$ and $\f_i$ to a tabloid $T$: using the
notation of the lemma, $\f_i(T)$ (respectively, $\e_i(T)$), when defined, is the tabloid
$\tikz[tableau]\matrix{\gamma_m \& |[dottedentry]| \& \gamma_1\\};$. Note that $\f_i$ and $\e_i$ preserve shapes of
tabloids and preserve the $\preceq$ relation between adjacent columns, and in particular preserve tableaux. Thus the
words in a given connected component are readings of tabloids with the same shape.
%
%
Furthermore, iterated application of this lemma shows that in a
given connected component of one of the crystal graphs, either every word is the reading of a tableau or no word is the
reading of a tableau. In a connected component where every word is the reading of a tableau, all the corresponding
tableaux have the same shape. (However, it is not true in general that two same-shape tabloids belong to the same
component.)

We can now say that a tabloid $T$ has highest weight if $\e_i(T)$ is undefined
for all $i$. Note that this is equivalent
to the word $w(T)$ being of highest weight. Furthermore, we have the following characterization of highest weight
tableaux:

\begin{lemma}
  \label{lem:highestweightableauchar}
  Let $X$ be one of the types $A_n$, $B_n$, $C_n$, $D_n$, and $G_2$. An $X$ tableau has highest weight if and only if it
  has $i$-th row filled with $i$, for $i=1,\ldots,n$, except that in the $D_n$
case the $n$-th row can instead be
  filled with $\bar{n}$.
\end{lemma}

\begin{proof}
  See \cite{kashiwara_crystalgraphs} for types $A_n$, $B_n$, $C_n$, and $D_n$; see \cite{lecouvey_survey} for type $G_2$.
\end{proof}

Note also that \fullref{Lemma}{lem:highestweightableauchar} can be recovered easily using the definition of the
operators $\e_i$ and the relation $\preceq$.


\begin{theorem}[{\cite{lecouvey_survey}}]
  \label{thm:tableauxcrossection}
  Let $X$ be one of the types $A_n$, $B_n$, $C_n$, $D_n$, and $G_2$, and let $\aX$ be the corresponding alphabet
  $\aA_n$, $\aB_n$, $\aC_n$, $\aD_n$ or $\aG_2$. Then for any $u \in \aX^*$, there is a unique tableau $P(u)$ such that
  $u \sim_X w(P(u))$. Thus the set of tableaux form a cross-section of the monoid $\Pl(X) = \aX^*/{\sim_X}$.
\end{theorem}

\subsection{Presentations for plactic monoids}
\label{subsec:presentations}


The classical plactic monoid $\Pl(A_n) = \aA_n^*/{\sim_{A_n}}$ is presented
by the so-called \defterm{Knuth relations} \cite{knuth_permutations}. Similarly,
all other types of plactic monoids are presented by certain defining relations
as described in \cite[\S~5.1]{lecouvey_survey}. In particular, for the cases
$B_n$, $C_n$ and $D_n$ the Knuth relations are also part of the given defining
relations.



In order to facilitate the reading of this article, we shall give more
details on some of the
defining relations that appear in cases $B_n$, $C_n$ and $D_n$. These relations
are labelled in \cite[\S~5.1]{lecouvey_survey} as $CR^{X}$, for
$X\in\{B,C,D\}$. We shall refer to the set of these relation as
$\drel{R}_5^{X}$, where $X$ is one of the types $B_n$, $C_n$ and $D_n$.
For our purposes we use the convention that $\bar{0}=0$ and that $\bar{\bar{z}}=z$.

A relation from $\drel{R}_5^{B_n}$ is defined as follows: let $w = w(C)$ be a
non-admissible column word for which
each strict factor (that is, a factor of $w$ not equal to $w$) is
admissible; let $z$ be the smallest (with respect to $ <$) unbarred symbol of
$w$ such that the
pair $(z,\bar{z})$ occurs in $w$ and $N_C(z) > z$, otherwise set $z = 0$. Let $\widetilde{w}$ be the column word
obtained by erasing the pair $(z,\bar{z})$ in $w$ if $z\leq n$ and erasing $0$ otherwise. The relation
$\drel{R}_5^{B_n}$ consists of all such pairs $(w,\tilde{w})$. (See
\cite[Definition 3.2.2]{lecouvey_bndn}.)

Both sets $\drel{R}_5^{C_n}$ and $\drel{R}_5^{D_n} $  are equal to
$\drel{R}_5^{B_n}$
except that we naturally exclude defining relations that involve $0$. (See
\cite[Definitions 5.1.2 and 5.1.3]{lecouvey_survey}.)

We now state the following auxiliary results that we will use in the sequel:

\begin{lemma}[{[Commuting columns lemma (CCL)]}]
\label{lem:commutingcolumns}
Let $X$ be one of the types $A_n$, $B_n$, $C_n$, $D_n$ and let $\aX$ be the corresponding alphabet $\aA_n$, $\aB_n$,
$\aC_n$ or $\aD_n$.  Let $\alpha,\beta \in \aX[1,n]^*$ be words that are readings of columns (that is, strictly
increasing words) such that every symbol of $\alpha$ appears in $\beta$. Then $\alpha\beta =_{\Pl(X)} \beta\alpha$.
\end{lemma}

\begin{proof}
  This follows directly from the Knuth relations; one can also use
Schensted's
  insertion algorithm for $\Pl(A_n)$ (see \cite[Chapter~5]{lothaire_algebraic}) and note that the required defining
  relations also appear in the presentations for the other types of plactic monoid.
\end{proof}

\begin{lemma}
  \label{lem:deleting}
  Let $\aX$ be one of the alphabets $\aB_n$, $\aC_n$ and $\aD_n$.
  Consider a word $w = 12\cdots q \bar{x_1}\,\bar{x_2}\cdots\bar{x_k}$ for some
$q \in \aX[1,n]$, $\bar{x_1}, \ldots,\bar{x_k} \in
  \aX[\bar{q},\bar{1}]$ and $1 \leq x_k < x_{k-1} < \ldots < x_1 \leq q$. Then
$w =_{\Pl(X)} u$, where $u$ is the word
  obtained from $12\cdots q$ by deleting the symbols $x_1,x_2,\ldots,x_k$. In particular, $u$ is either empty or is an
  admissible column containing fewer than $q$ symbols.
\end{lemma}

\begin{proof}
  Let $u^{(i)}$ be the word obtained by deleting $x_1, \ldots, x_i$ from $1 2 \cdots q$. Then $u^{(i)} \bar{x_{i+1}}
  =_{\Pl(X)} u^{(i+1)}$ is an $\drel{R}_5^X$ relation. Note that $w =
u^{(0)}$. By induction, therefore, $u = u^{(k)}$ is a
  column with $u^{(0)} \bar{x_1} \ldots \bar{x_k} =_{\Pl(X)} u^{(k)}$. Clearly
$|u|$ is less than $q$. Since $u$
  contains only symbols from $\aX[1,q]$, it follows that $N_u(z) \leq z$ for all $z$ and so $u$ is an admissible
  column if it is non-empty.
\end{proof}

We also give details of the presentation defining $\Pl(G_2)$ as it will be
frequently mentioned in the following sections. Another reason is because  we
give this presentation in a slightly different way from Lecouvey
\cite[Definition~5.1.4]{lecouvey_survey}. Note, that the sets of defining
relations $\drel{R}_1^{G_2}$, $\drel{R}_2^{G_2}$,
$\drel{R}_3^{G_2}$, and $\drel{R}_4^{G_2}$, defined below, still
correspond to the crystal isomorphisms identified by Lecouvey, and
hence these relations generate the same congruence as those of Lecouvey.

Giving a presentation for $\Pl(G_2)$ requires the auxiliary partial map $\Theta$ on $\aG_2^2$ defined as per the following table:
\[
\begin{array}{r|cccccccccccccc}
w & 21 & 31 & 01 & \bar{3}1 & \bar{3}2 & \bar{2}1 & \bar{2}2 & \bar{1}1 & \bar{1}2 & \bar{2}3 & \bar{1}3 & \bar{1}0 & \bar{1}\bar{3} & \bar{1}\bar{2} \\
\hline
\vbox to 4mm{} w\Theta & 12 & 13 & 23 & 20 & 2\bar{3} & 30 & 3\bar{3} & 00 & 0\bar{3} & 3\bar{2} & 0\bar{2} & \bar{3}\bar{2} & \bar{3}\bar{1} & \bar{2}\bar{1}
\end{array}
\]
The monoid $\Pl(G_2)$ is presented by $\pres{\aG_2}{\drel{R}_1^{G_2} \cup \drel{R}_2^{G_2} \cup \drel{R}_3^{G_2} \cup \drel{R}_4^{G_2}}$ (see \cite[Definition 5.1.4]{lecouvey_survey}), where
\begin{align*}
\drel{R}_1^{G_2} ={}& \set[\big]{(10,1),(1\bar{3},2),(1\bar{2},3),(2\bar{2},0),(2\bar{1},\bar{3}),(3\bar{1},\bar{2}),(0\bar{1},\bar{1})}, \\
\drel{R}_2^{G_2} ={}& \set[\big]{(1\bar{1},\emptyword)}, \\
\drel{R}_3^{G_2} ={}& \gset[\big]{(abc,a(bc)\Theta)}{ab \in \im\Theta,bc\in\dom\Theta} \\
&\cup \gset[\big]{(abc,(ab)\Theta^{-1}c)}{ab \in \im\Theta, b \geq c, bc \neq 00, bc \notin \dom\Theta}, \\
\drel{R}_4^{G_2} ={}& \set[\big]{(123,110)} \\
&\cup \gset[\big]{((abc,(ab)\Theta^{-1}c)}{ab \in \im\Theta, bc \in \im\Theta, abc\neq123}.
\end{align*}

\section{Basic two-column lemmata}
\label{sec:twocolumnlemmata}

As described in the strategic overview of our proofs in the Introduction, this section examines
products of two admissible columns that do not form a tableau. In order to prove that the rewriting system we will
construct is terminating, we have to know about the shape of the tableau that
result from this product. Informally, we
will show that the resulting tableau either:
\begin{enumerate}
\item Has fewer entries than the original two columns.
\item Has the same number of entries but only one column.
\item Has the same number of entries, two columns, and has a shorter rightmost column.
\end{enumerate}
The results are given formally in the following two subsections as \fullref{Lemmata}{lem:an:palphabeta}, \ref{lem:combined:palphabeta}
and \ref{lem:g2:palphabeta}.

To construct finite complete rewriting systems for the Plactic monoids only the
basic two-column lemmata, as stated in this section, are needed. In order to
establish our biautomaticity results a far more detailed understanding is needed
of how products of columns behave. These details will be given in the
(non-basic) two-column lemmata in
\fullref{Section}{sec:biautomaticity_lemmata}.

 We consider first the classical type  $A_n$, and in a combined way the types  $B_n$, $C_n$, and
$D_n$, reflecting the increasing order of complexity of the arguments. Type $G_2$ is considered last, because it uses a
rather different approach from the other types.

\subsection{Proving the basic two-column lemmata}

The following result was originally proved in \cite[Lemma~5.7]{cgm_plactic}. We present an alternative proof which uses the Littlewood--Richardson rule for decomposing tensor products of crystals into a disjoint union of connected components; see \cite[Theorem~7.4.6.]{hong_quantumgroups}.

Here we use the reference \cite{hong_quantumgroups}. Full details of the proofs of these results are not given in \cite{hong_quantumgroups} but may be found in the original paper of Nakashima on this topic \cite{nakashima1993}.

%

In the following proofs $n$ will be fixed, and by a partition we shall mean a
sequence $\lambda = (\lambda_1, \lambda_2, \ldots, \lambda_n)$ where $\lambda_1
\geq \lambda_2 \geq \ldots \geq \lambda_n \geq 0$.

\begin{lemma}[Two-column lemma for type $A_n$]
  \label{lem:an:palphabeta}
  If $\alpha, \beta \in \aA_n^*$ are columns such that $\beta \not\preceq \alpha$
then the tableau $P(\alpha\beta)$ contains $|\alpha\beta|$ symbols, and consists of either one column or two columns, the rightmost of
  which contains fewer than $|\alpha|$ symbols.
\end{lemma}
\begin{proof}
We follow the notation and terminology of \cite[Chapter~7]{hong_quantumgroups}.
 Let $\alpha,\beta \in \aA_n^*$ be columns where $|\alpha|=k$ and $|\beta|=l$,
and such that
 $\beta \not\preceq \alpha$.
%
%
Let $Y$ be the Young diagram with a single column of height $k$ and let $Y'$ be the Young diagram of a single column of height $l$.
%
%
Let $B(Y)$ be the connected component of the crystal graph of all
admissible columns
with shape $Y$. Note that $\alpha$ belongs to $B(Y)$.  Similarly we
define $B(Y')$ and note that $\beta$ belongs to $B(Y')$.

Now by
\cite[Theorem~7.4.6.]{hong_quantumgroups} the tensor product of these two
crystal components decomposes as the disjoint union of connected connected
components
\[
B(Y) \otimes B(Y') \cong \bigoplus_{x_1 x_2 \ldots x_l \in B(Y')} B(Y[x_1, x_2, \ldots, x_l]).
\]
By definition $Y[j]$ denotes the diagram obtained by adding a box to the $j$th
row of $Y$, and $Y[j_1, \ldots, j_r]$ is defined inductively to be the diagram
obtained from $Y[j_1, \ldots, j_{r-1}]$ by adding a box at the $j_r$th row.
Here $B(Y[j_1, \ldots, j_r])$ is defined to be $\varnothing$ if at least one of
the $Y[j_1, \ldots, j_q]$ with $q \leq r$ is not a Young diagram; see
\cite[page 165]{hong_quantumgroups}.

In our case, $Y$ is a column of height
$k$ and $x_1 x_2 \ldots x_l$ is a reading of a tableau of shape $Y'$, that is,
$x_1 < x_2 < \ldots < x_l$ is a strictly increasing sequence from $\aA_n$. It
follows that $B(Y[x_1, x_2, \ldots, x_l]) \neq \varnothing$ if and only if
$x_1, x_2, \ldots, x_l$ is a sequence of the form $1, 2, \ldots, a, k+1, k+2,
\ldots, k+(l-a)$ for some non negative integer $a$ such that $a \leq k$ and
$k+l-a\leq n$. Note that  $a$ can possibly be $0$ meaning that the
sequence starts at $k+1$. The Young diagram $Y[1, 2, \ldots, a, k+1,
k+2, \ldots, k+(l-a)]$ has shape $\nu^{(a)} = (\nu_1, \nu_2, \ldots, \nu_n)$
where
\[
\nu_i = \begin{cases}
2 & \mbox{for $1 \leq i \leq a$} \\
1 & \mbox{for $a < i \leq k+l-a$} \\
0 & \mbox{for $i > k+l-a$.}
\end{cases}
\]
The product $\alpha \beta$ must belong to one of the  connect components of
$B(Y) \otimes B(Y')$. Therefore $P(\alpha \beta)$ is a tableau of shape
$\nu^{(a)}$ for some value of $a \leq k$.

 If $a=k$ then $P(\alpha \beta)$ is a
tableau with two columns, with the right column being of height $k$ and the
left column being of height $l$. It is straightforward to show that this is
only possible if
$\beta \preceq \alpha$ (a general argument for this, which also applies in this
 case, may be found in the proof of Lemma~\ref{lem:combined:palphabeta}
below). Since $\beta \not\preceq \alpha$ by assumption, it follows that
$P(\alpha \beta)$ has shape $\nu^{(a)}$ for some $a < k$. But then the shape
$\nu^{(a)}$ has one column or two columns the rightmost of which has $a < k =
|\alpha|$ symbols. This completes the proof.
%
%
%
%
%
\end{proof}

One benefit of reproving Lemma~\ref{lem:an:palphabeta} using the above method is that it can be generalised to the other classical types by applying the generalized Littlewood--Richardson rule for decomposing tensor products of crystals into a disjoint union of connected components; see \cite[Theorem~8.6.6.]{hong_quantumgroups}.

\begin{lemma}[Two-column lemma for types $B_n$, $C_n$ and $D_n$]
  \label{lem:combined:palphabeta}
  Let $X$ be one of the types $B_n$, $C_n$ or $D_n$, and let $\aX$ be the corresponding alphabet from
$\aB_n$, $\aC_n$ or $\aD_n$. If $\alpha, \beta \in \aX^*$ are admissible columns
such that $\beta \not\preceq \alpha$
then the tableau $P(\alpha\beta)$ contains at most $|\alpha\beta|$ symbols, and is either empty or consists of either one column or two
  columns, the rightmost of which contains fewer than $|\alpha|$ symbols.
\end{lemma}
\begin{proof}
We follow the notation and terminology of \cite[Chapter~8]{hong_quantumgroups}. The proof is similar to that of Lemma~\ref{lem:an:palphabeta} but we apply \cite[Theorem~8.6.6]{hong_quantumgroups} in place of \cite[Theorem~7.4.6]{hong_quantumgroups}.

Let $\alpha$ and $\beta$ be admissible columns of one of the types $B_n$, $C_n$
or $D_n$, where $|\alpha|=k$ and $|\beta|=l$. Let $Y$ be the Young diagram with
a single column of height $k$. It follows from
\cite[Theorem~8.6.6]{hong_quantumgroups} that the shape of the tableau
$P(\alpha\beta)$ must be given by a Young diagram of the form
$Y[x_1, x_2, \ldots, x_l]$
where $x_1, x_2, \ldots , x_l$ is a reading of an admissible column of height $l$.

Recall (see \fullref{Subsection}{subsec:tableaux_cols}  above) that readings of the columns
types $B_n$, $C_n$ and $D_n$ have, respectively, the forms $\beta_+ \beta_0 \beta_{-} \in \aB_n^*$, $\gamma_+ \gamma_{-} \in \aC_n^*$ and $\delta_+ \delta \delta_{-} \in \aD_n^*$ where these words satisfy the admissibility conditions given in the \fullref{Subsubsection}{subsubsec:admissible_cols}.
%
%
In particular $\delta$ is filled with symbols $n$ and $\bar{n}$, with different
symbols in vertically adjacent cells, and $\beta_0$ is filled with the symbol
$0$. The other sections are filled with strictly increasing sequences with
respect to the orderings of the vertices in the respective crystal bases. Given
a Young diagram $Y$ of shape $(\lambda_1, \lambda_2, \ldots, \lambda_n)$,
$Y[j]$ is defined (see \cite[page 205]{hong_quantumgroups}) by
\begin{align*}
Y[ \; j \; ] & = (\lambda_1, \ldots, \lambda_j + 1, \ldots, \lambda_n)  &
\mbox{for $j=1, \ldots, n$} \\ \tag{$\diamond$}
Y[ \; \bar{j} \; ] & = (\lambda_1, \ldots, \lambda_j - 1, \ldots, \lambda_n)  &
\mbox{for $j=1, \ldots, n$} \\
Y[0] &=
\begin{cases}
Y & \mbox{if $\lambda_n > 0$} \\
(\lambda_1, \ldots, \lambda_{n-1}, -\infty) & \mbox{if $\lambda_n=0$.}
\end{cases}
\end{align*}
In general $Y[j]$ will itself not be a Young diagram. Set $B(Z)=\varnothing$ if
$Z$ is not a Young diagram. (Note that in \cite[page 205]{hong_quantumgroups}
the authors work with generalised Young diagrams, but for our purposes Young
diagrams suffice since we are not concerned with Plactic monoids associated with
spin representations in this paper.) Then $Y[x_1, x_2, \ldots, x_l]$ is defined
inductively by $Y[x_1, x_2, \ldots, x_{l-1}][x_l]$, where if any of the
intermediate stages $Y[x_1, x_2, \ldots, x_q]$ is itself not a Young diagram
then we set $B(Y[x_1, x_2, \ldots, x_l]) = \varnothing$.

Now $Y[x_1, x_2, \ldots,
x_l]$ is obtained by starting with a column of height $k$ which is a Young
diagram with shape $(\lambda_1, \lambda_2, \ldots, \lambda_n)$ where $\lambda_i
= 1$ for $i \leq k$ and is equal to $0$ otherwise. Then for each of the symbols
$x_1, x_2, \ldots, x_l$ and so on from our column reading we carry out one of
the operations in ($\diamond$). If at any stage the symbol $-\infty$ appears the
process halts and we set $B(Y[x_1, x_2, \ldots, x_l])=\varnothing$, so we can
assume that is not the case. This means that whenever the symbol $0$ is read in
this process, the Young diagram remains unchanged. Also, in the $D_n$ case when
the $\delta$ portion of the word is read, this is an alternating sequence of $n$
and $\bar{n}$ which will ultimately either add $1$ to $\lambda_n$ or subtract
$1$ from $\lambda_n$.

Considering each of the three cases it is straightforward
to see that in the end, if $B(Y[x_1, x_2, \ldots, x_l]) \neq \varnothing$, then $Y[x_1, x_2, \ldots, x_l]$
must be a Young diagram with shape $\nu^{(a,b)} = (\nu_1, \nu_2, \ldots, \nu_n)$
for some $a,b \geq 0$, with $2a+b \leq k+l$ and $a \leq k$, where
\[
\nu_i = \begin{cases}
2 & \mbox{for $1 \leq i \leq a$} \\
1 & \mbox{for $a < i \leq a+b$} \\
0 & \mbox{for $i > a+b$.}
\end{cases}
\]
Note that $a=0$, and $b=0$, are both possible here.

Suppose that $a=k$. In this case, the diagram $Y[x_1, x_2, \ldots, x_l]$ has
more boxes than $Y$, and so, from the definitions of $Y[j]$,
$Y[\bar{j}]$, $Y[0]$, necessarily one of the $x_i$ belongs to $\{1,\ldots,n\}$.
Furthermore, since $a=k\geq 1$, then $x_i=1$ for some $i\in\{1,\dots,l\}$.
Since $x_1\dots x_l$ is an admissible column, if $x_i<x_j$ then $i<j$,
for all $i,j\in \{1,\ldots, l\}$ (The converse also holds except in case
$D_n$).  Thus $x_1=1$.
Now, suppose that for some $s\in \{1,\ldots, l\}$, we have $x_s=\bar{t}$ for
some $t\in\{1,\dots,n\}$. Let $s$ be minimal under such conditions. From the
definitions of $Y[j]$,
$Y[\bar{j}]$, and since $Y[x_1, x_2, \ldots, x_l]$ is a Young diagram of shape
$\nu^{(k,b)}$, there exists some $i\in\{1,\ldots, l\}$, for which $x_i=t$.
Because $x_1\dots x_l$ is an admissible column either $t$ appears to the left
of $\bar{t}$ in $x_1\dots x_l$,  or $t=n$ and $\bar{t}t$ is a factor of
$x_1\dots x_l$ (this situation can only occur in case $D_n$). Since $x_1=1$
and $Y[x_1, x_2, \ldots, x_l]$ is a Young diagram, in the first case we would
have $1 2\dots t u \bar{t}$, for some word $u$, as a prefix of $x_1\dots x_l$,
and in the second case $1\dots (n-1)\bar{n}n$ as a prefix of $x_1\dots x_l$. In
both cases, this contradicts the fact that $x_1\dots x_l$ is an admissible
column. So none of the $x_i$'s is a barred symbol. Now suppose that $x_r=0$
(only possible in case $B_n$) for some $r\in \{1,\dots,l\}$ and choose $r$ to
be minimal in those conditions. Because $x_1\dots x_l$ is an admissible column
and $Y[x_1, x_2, \ldots, x_l]$ is a Young diagram then  $x_1\dots x_l$ has
the form $12\dots (r- 1)0\dots 0$. Also, since $Y[x_1, x_2,
\ldots, x_l]$ is a Young diagram and by the definition of $Y[0]$ we
necessarily have $r-1=n$. We get a contradiction since $12\dots n0\dots 0$ is
not an admissible column of type $B_n$.
It follows that $[x_1, x_2, \ldots, x_l] = [1,2,\ldots, l]$, and that $Y[x_1,
x_2, \ldots, x_l]$
has shape $\nu^{(k,l-k)}$.

From the above, it is then immediate that $Y[x_1, x_2, \ldots, x_l]$
has shape $\nu^{(k,l-k)}$
if and only if $[x_1, x_2, \ldots, x_l] = [1,2,\ldots, l]$. Let $Y'$ be the Young diagram with a single column of height $l$. It follows that in
the decomposition
\[
B(Y) \otimes B(Y') \cong \bigoplus_{x_1 x_2 \dots x_l \in B(Y')} B(Y[x_1, x_2, \ldots, x_l])
\]
into connected components given by \cite[Theorem~8.6.6]{hong_quantumgroups} the component $B(Y[1,2,\ldots, l])$ occurs exactly once and thus no other connected component of $B(Y) \otimes B(Y')$ is isomorphic to $B(Y[1,2,\ldots, l])$. Since $\alpha \beta$ belongs to the connected component  $B(Y[1,2,\ldots, l])$ of $B(Y) \otimes B(Y') $ it follows that $P(\alpha \beta) = \alpha' \beta'$ where $\beta'$ and $\alpha'$ are admissible columns with  $|\alpha'| = k = |\alpha|$ and $|\beta'| = l = |\beta|$. But then $\alpha' \beta'$ belongs to $ B(Y) \otimes B(Y')$ and must also belong to the same connected component $B(Y[1,2,\ldots, l])$. Then from $ \alpha' \beta' = P( \alpha \beta)$ it follows that $\alpha \beta$ and $\alpha' \beta'$ have the same position in this connected component and hence $\alpha' \beta'$ and $ \alpha \beta$ are identical as words which in turn implies that $\beta' = \beta$ and $\alpha' = \alpha$. This implies that $\beta \preceq \alpha$.

Since by assumption $\beta \not\preceq \alpha$, it follows from the arguments
above that $a < k$. Then $Y[x_1, x_2, \ldots, x_l]$
has shape $\nu^{(a,b)}$, with $a<k$, which by inspection satisfies the
conclusions given in the statement of the lemma, completing the proof.
\end{proof}

%
%

\subsection{\texorpdfstring{Two column lemma for type $G_2$}{Two column lemma for type G2}}

The proof for $G_2$ uses a rather different approach from the other types. As in the previous subsection, our aim is to
learn about the shape of the tableau $P(\alpha\beta)$, where $\alpha$ and $\beta$ are admissible $G_2$ columns and
$\beta \not\preceq \alpha$; for the conclusion, see \fullref{Lemma}{lem:g2:palphabeta}. Recall that there are only
finitely many admissible $G_2$ columns (which are listed in \eqref{eq:g2admissiblecolumns}). Thus, our approach is
simply to characterize the finitely many possibilities for $\alpha$ and $\beta$ when $\alpha\beta$ is highest weight in
\fullref{Lemma}{lem:g2:alphabetachar}, and then to compute $P(\alpha\beta)$ in each case and derive the conclusion about
products of arbitrary pairs of admissible columns in \fullref{Lemma}{lem:g2:palphabeta}.

\begin{lemma}
  \label{lem:g2:alphabetachar}
  Let $\alpha$ and $\beta$ be admissible $G_2$ column words such that $\beta \not\preceq \alpha$ and $\alpha\beta$ is a
  highest-weight word.  Either:
  \begin{enumerate}
  \item $\alpha = 1$ and $\beta \in \set{2,0,\bar{1},23,00}$; or
  \item $\alpha = 12$ and $\beta \in \set{1,3,\bar{2},13,30,3\bar{3},\bar{2}\bar{1}}$.
  \end{enumerate}
\end{lemma}

\begin{proof}
  Since $\alpha\beta$ is of highest weight, by
\fullref{Lemma}{lem:highestweightfactors}, $\alpha$ is a highest weight column
  (and thus a highest-weight tableau). The highest weight admissible columns of lengths $1$ and $2$ are $1$ and $12$, so
  either $\alpha = 1$ or $\alpha =12$.
  \begin{enumerate}
  \item Suppose $\alpha = 1$. Let $\beta = x\beta'$, where $x \in \aG_2$. If $x = 1$, then
    $\beta \preceq\alpha$, which is a contradiction. Furthermore,
    \begin{align*}
      x = 3 \implies& \rho_{2}(\alpha\beta) = \rho_2(13\beta') = {-}\rho_2(\beta') = {-}\cdots;\\
      x = \bar{3} \implies& \rho_{1}(\alpha\beta) = \rho_1(1\bar{3}\beta') = {+}{-}{-}\rho_1(\beta') = {-}\cdots;\\
      x = \bar{2} \implies& \rho_{2}(\alpha\beta) = \rho_2(1\bar{2}\beta') = {-}\rho_2(\beta') = {-}\cdots.
    \end{align*}
    In each case, the supposition contradicts $\alpha\beta$ being of highest weight. So $x$ must be $2$, $0$, or
    $\bar{1}$; if $|\beta| = 1$, these are the possibilities for $\beta$.

    Suppose now that $|\beta| = 2$. This cannot occur when $x=\bar{1}$, for no admissible column begins with
    $\bar{1}$. The admissible column words of length $2$ beginning with $2$ and $0$ are $23$, $20$, $2\bar{3}$, and $00$, $0\bar{3}$
    and $0\bar{2}$. Furthermore,
    \begin{align*}
      \beta=20 \implies & \rho_1(\alpha\beta) = \rho_1(120) = {+}{-}{-}{+} = {-}{+}, \\
      \beta=2\bar{3} \implies & \rho_1(\alpha\beta) = \rho_1(12\bar{3}) = {+}{-}{-}{-} = {-}{-}, \displaybreak[0]\\
      \beta=0\bar{3} \implies & \rho_1(\alpha\beta) = \rho_1(10\bar{3}) = {+}{-}{+}{-}{-} = {-}, \\
      \beta=0\bar{2} \implies & \rho_2(\alpha\beta) = \rho_2(10\bar{2}) = {-},
    \end{align*}
    each of which contradicts $\alpha\beta$ being of highest weight. The remaining possibilities are $\beta = 23$ and $\beta=00$.
  \item Suppose $\alpha = 12$. Let $\beta = x\beta'$, where $x \in \aG_2$. Then
    \begin{align*}
      x=2\implies & \rho_1(\alpha\beta) = \rho_1(122\beta') = {+}{-}{-}\rho_1(\beta') = {-}\cdots, \\
      x=0\implies & \rho_1(\alpha\beta) = \rho_1(120\beta') = {+}{-}{-}{+}\rho_1(\beta') = {-}{+}\cdots, \displaybreak[0]\\
      x=\bar{3}\implies & \rho_1(\alpha\beta) = \rho_1(12\bar{3}\beta') = {+}{-}{-}{-}\rho_1(\beta') = {-}{-}\cdots, \\
      x=\bar{1}\implies & \rho_1(\alpha\beta) = \rho_1(12\bar{1}\beta') = {+}{-}{-}\rho_1(\beta') = {-}\cdots,
    \end{align*}
    each of which contradicts $\alpha\beta$ being of highest weight. So $x$ must be $1$, $3$, or $\bar{2}$. If
    $|\beta|=1$, these are the possibilities for $\beta$.

    Suppose now that $|\beta|=2$. The admissible column words of length $2$ beginning with $1$, $3$, and $\bar{2}$ are
    $12$, $13$, $30$, $3\bar{3}$, $3\bar{2}$, $\bar{2}\bar{1}$. Note first that $\beta \neq 12$ since $\beta \not\preceq
    \alpha$. Furthermore
    \[
    \beta=3\bar{2} \implies \rho_2(\alpha\beta) = \rho_2(123\bar{2}) = {+}{-}{-} = {-},
    \]
    which contradicts $\alpha\beta$ being of highest weight. The remaining possibilities for $\beta$ are $13$, $30$,
    $3\bar{3}$, and $\bar{2}\bar{1}$. \qedhere
  \end{enumerate}
\end{proof}

\begin{lemma}[Two-column lemma for type $G_2$]
  \label{lem:g2:palphabeta}
  Let $\alpha$ and $\beta$ be admissible $G_2$ columns with $\beta \not\preceq \alpha$. Then either:
  \begin{itemize}
  \item $P(\alpha\beta)$ contains fewer that $|\alpha\beta|$ symbols,
  \item $P(\alpha\beta)$ contains exactly $|\alpha\beta|$ symbols and has at most one column,
  \item $P(\alpha\beta)$ contains exactly $|\alpha\beta|$ symbols and has exactly two columns, the rightmost of which
    contains fewer than $|\alpha|$ symbols.
  \end{itemize}
\end{lemma}

\begin{proof}
  Since the Kashiwara operators preserve shapes of tabloids and also preserves whether the $\preceq$ relation holds
  between adjacent columns, we can assume that $\alpha\beta$ has highest weight. Using
  \fullref{Lemma}{lem:g2:alphabetachar}, we systematically enumerate the possible words $\alpha\beta$ and calculate
  their corresponding tableaux. The results are shown in \fullref{Table}{tbl:g2:palphabeta}.

  \begin{table}[t]
    \centering
    \caption{Case analysis for the proof of \fullref{Lemma}{lem:g2:palphabeta}.}
    \label{tbl:g2:palphabeta}
    \begin{tabular}{c@{\kern 1mm}rrl@{\kern 1mm}c@{\kern 1mm}c}
      \toprule
      \textit{Shape of} & & & & & \textit{Shape of} \\
      \tikz[tableau]\matrix{\beta \& \alpha\\}; & $\alpha$ & $\beta$ & \textit{Defining relations applied} & $P(\alpha\beta)$ & \textit{$P(\alpha\beta)$} \\
      \midrule
      \tikz[shapetableau]\matrix{\null \& \null\\}; & $1$ & $2$ & --- & $12$ & \tikz[shapetableau]\matrix{\null\\\null\\}; \\
      \tikz[shapetableau]\matrix{\null \& \null\\}; & $1$ & $0$ & $10 =_{\drel{R}_1^{G_2}} 1$ & $1$ & \tikz[shapetableau]\matrix{\null\\}; \\
      \tikz[shapetableau]\matrix{\null \& \null\\}; & $1$ & $\bar{1}$ & $1\bar{1} =_{\drel{R}_2^{G_2}} \emptyword$ & $\emptyword$ & \\
      \tikz[shapetableau]\matrix{\null \& \null\\ \null\\}; & $1$ & $23$ & $123 =_{\drel{R}_4^{G_2}} 110 =_{\drel{R}_1^{G_2}} 11$ & $11$ & \tikz[shapetableau]\matrix{\null\&\null\\}; \\
      \tikz[shapetableau]\matrix{\null \& \null\\ \null\\}; & $1$ & $00$ & $100 =_{\drel{R}_1^{G_2}} 10 =_{\drel{R}_1^{G_2}} 1$ & $1$ & \tikz[shapetableau]\matrix{\null\\}; \\
      \tikz[shapetableau]\matrix{\null \& \null\\ \& \null\\}; & $12$ & $1$ & $121 =_{\drel{R}_3^{G_2}} 112$ & $112$ & \tikz[shapetableau]\matrix{\null\&\null\\\null\\}; \\
      \tikz[shapetableau]\matrix{\null \& \null\\ \& \null\\}; & $12$ & $3$ & $123 =_{\drel{R}_4^{G_2}} 110 =_{\drel{R}_1^{G_2}} 11$ & $11$ & \tikz[shapetableau]\matrix{\null\&\null\\}; \\
      \tikz[shapetableau]\matrix{\null \& \null\\ \& \null\\}; & $12$ & $\bar{2}$ & $12\bar{2} =_{\drel{R}_1^{G_2}} 10 =_{\drel{R}_1^{G_2}} 1$ & $1$ & \tikz[shapetableau]\matrix{\null\\}; \\
      \tikz[shapetableau]\matrix{\null \& \null\\\null \& \null\\}; & $12$ & $13$ & $1213 =_{\drel{R}_3^{G_2}} 1123 =_{\drel{R}_4^{G_2}} 1110 =_{\drel{R}_1^{G_2}} 111$ & $111$ & \tikz[shapetableau]\matrix{\null\&\null\&\null\\}; \\
      \tikz[shapetableau]\matrix{\null \& \null\\ \null \& \null\\}; & $12$ & $30$ & $1230 =_{\drel{R}_4^{G_2}} 1100 =_{\drel{R}_1^{G_2}} 110 =_{\drel{R}_1^{G_2}} 11$ & $11$ & \tikz[shapetableau]\matrix{\null\&\null\\}; \\
      \tikz[shapetableau]\matrix{\null \& \null\\\null \& \null\\}; & $12$ & $3\bar{3}$ & $123\bar{3} =_{\drel{R}_4^{G_2}} 110\bar{3} =_{\drel{R}_1^{G_2}} 11\bar{3} =_{\drel{R}_1^{G_2}} 12$ & $12$ & \tikz[shapetableau]\matrix{\null\\\null\\}; \\
      \tikz[shapetableau]\matrix{\null \& \null\\\null \& \null\\}; & $12$ & $\bar{2}\bar{1}$ & $12\bar{2}\bar{1} =_{\drel{R}_1^{G_2}} 10\bar{1} =_{\drel{R}_1^{G_2}} 1\bar{1} =_{\drel{R}_2^{G_2}} \emptyword$ & $\emptyword$ & \\
      \bottomrule
    \end{tabular}
  \end{table}

  In each case, we get a tableau that contains fewer that $|\alpha\beta|$ symbols (and that is in some case empty), and
  in the cases when the number of symbols in the tableau is equal to $|\alpha\beta|$, either the tableau contains only
  one column, or else contains two columns and the number of symbols in the rightmost column is less than $|\alpha|$.
\end{proof}

\section{Constructing the rewriting system}
\label{sec:rewriting}

We now turn to actually constructing the finite complete rewriting systems presenting $\Pl(A_n)$, $\Pl(B_n)$,
$\Pl(C_n)$, $\Pl(D_n)$, and $\Pl(G_2)$. The constructions can be carried out in parallel, because the only differences
are the appeals to the different lemmata from \fullref{Section}{sec:twocolumnlemmata}. We first of all recall the
necessary definitions about rewriting systems in \fullref{Subsection}{subsec:srs}; for further background, see
\cite{book_srs} or \cite{baader_termrewriting}. For background on semigroup presentations generally, see
\cite{ruskuc_phd} or \cite{higgins_techniques}.

\subsection{Preliminaries}
\label{subsec:srs}

Let $\leq$ be a total order on an alphabet $A$. Define a total order $\leq_\lex$ on $A^*$ by $w \leq_\lex w'$ if and only
if either $w$ is proper prefix of $w'$ or if $w=paq$, $w'=pbr$ and $a \leq b$ for some $p,q,r \in A^*$, and $a,b \in
A$. The order $\leq_\lex$ is the \defterm{lexicographic order induced by $\leq$}. Notice that $\leq_\lex$ is not a
well-order, but that it is left compatible with concatenation. Define also a total order $\leq_\lenlex$ on $A^*$ by
\[
w \leq_\lenlex w' \iff (|w| < |w'|) \lor \bigl((|w| = |w'|) \land (w \leq_\lex w')\bigr).
\]
The order $\leq_\lenlex$ is the \defterm{length-plus-lexicographic order induced by $\leq$}. The order $\leq_\lenlex$ is a
well-order and is left compatible with concatenation.

A \defterm{string rewriting system}, or simply a \defterm{rewriting system}, is a pair $(A,{R})$, where $A$ is a
finite alphabet and ${R}$ is a set of pairs $(\ell,r)$, usually written $\ell \imreduces r$, known as
\defterm{rewriting rules} or simply \defterm{rules}, drawn from $A^* \times A^*$. The single reduction relation
$\imreduces_{{R}}$ is defined as follows: $u \imreduces_{{R}} v$ (where $u,v \in A^*$) if there exists a
rewriting rule $(\ell,r) \in {R}$ and words $x,y \in A^*$ such that $u = x\ell y$ and $v = xry$. That is, $u
\imreduces_{{R}} v$ if one can obtain $v$ from $u$ by substituting the word $r$ for a subword $\ell$ of $u$, where
$\ell \imreduces r$ is a rewriting rule. The reduction relation $\reduces_{{R}}$ is the reflexive and transitive
closure of $\imreduces_{{R}}$. The process of replacing a subword $\ell$ by a word $r$, where $\ell \imreduces r$ is
a rule, is called \defterm{reduction} by application of the rule $\ell \imreduces r$; the iteration of this process is
also called reduction. A word $w \in A^*$ is \defterm{reducible} if it contains a subword $\ell$ that forms the
left-hand side of a rewriting rule in ${R}$; it is otherwise called \defterm{irreducible}.

The rewriting system $(A,{R})$ is \defterm{finite} if both $A$ and ${R}$ are finite. The rewriting system $(A,{R})$ is
\defterm{noetherian} if there is no infinite sequence $u_1,u_2,\ldots \in A^*$ such that $u_i \imreduces_{{R}} u_{i+1}$
for all $i \in \nset$. That is, $(A,{R})$ is noetherian if any process of reduction must eventually terminate with an
irreducible word. The rewriting system $(A,{R})$ is \defterm{confluent} if, for any words $u, u',u'' \in A^*$ with
$u \reduces_{{R}} u'$ and $u \reduces_{{R}} u''$, there exists a word $v \in A^*$ such that $u' \reduces_{{R}} v$ and
$u'' \reduces_{{R}} v$. A rewriting system that is both confluent and noetherian is \defterm{complete}. If $(A,{R})$ is
a complete rewriting system, then for every word $u$ there is a unique irreducible word $w$ such that
$u \reduces_{{R}} w$; this word is called the \defterm{normal form} of $u$. If $(A,{R})$ is complete, then the language
of normal form words forms a cross-section of the monoid: that is, each element of the monoid presented by
$\pres{A}{{R}}$ has a unique normal form representative.

\subsection{Construction}
\label{subsec:rewritingconstruction}

Let $X$ be one of the types $A_n$, $B_n$, $C_n$, $D_n$, and $G_2$, and let $\aX$ be the corresponding alphabet from
$\aA_n$, $\aB_n$, $\aC_n$, $\aD_n$, or $\aG_2$. Let
\[
\Sigma = \gset{c_\sigma}{\text{$\sigma$ is an admissible $X$ column}}.
\]
Note that $\Sigma$ is finite since there are finitely many admissible $X$ columns.

Let $T$ consist of the following rewriting rules:
\begin{align}
c_\sigma c_\tau &\imreduces \emptyword &&\text{$\tau\not\preceq\sigma$ and $P(\sigma\tau)$ is empty,} \label{eq:length0}\\
c_\sigma c_\tau &\imreduces c_\upsilon &&\text{$\tau\not\preceq\sigma$ and $P(\sigma\tau)$ is the $1$-col.~tableau $\tikz[tableau]\matrix{\upsilon\\};$,} \label{eq:length1}\\
c_\sigma c_\tau &\imreduces c_\upsilon c_\phi &&\text{$\tau\not\preceq\sigma$ and $P(\sigma\tau)$ is the $2$-col.~tableau $\tikz[tableau]\matrix{\phi\&\upsilon\\};$,} \label{eq:length2}\\
c_\sigma c_\tau &\imreduces c_\upsilon c_\phi c_\chi&&\text{$\tau\not\preceq\sigma$ and $P(\sigma\tau)$ is the $3$-col.~tableau $\tikz[tableau]\matrix{\chi\&\phi\&\upsilon\\};$.}\label{eq:length3}
\end{align}
Note that since $P(\sigma\tau)$ is a tableau, the subscripts $\upsilon$, $\phi$, and $\chi$ are always admissible columns.

Note that if $\sigma,\tau$ are admissible columns with $\tau \not\preceq \sigma$, then $P(\sigma\tau)$ has at most three
columns by \fullref{Lemmata}{lem:an:palphabeta}, \ref{lem:combined:palphabeta}, and \ref{lem:g2:palphabeta} (that is, by the two-column lemmata for type $A_n$, types $B_n$, $C_n$,
$D_n$, and type $G_2$). Thus every such pair of columns gives rise to a rewriting rule in $T$. (Note that rules of the form
$c_\sigma c_\tau \imreduces c_\upsilon c_\phi c_\chi$ only arise when $X=G_2$, because $P(\sigma\tau)$ has at most two
columns in the other cases.) Finally, note that $T$ is finite since there are finitely many possibilities for $\sigma$
and $\tau$, and the right-hand side of each rule is uniquely determined by the left-hand side.

The idea is that a word $c_{\beta^{(1)}}\cdots c_{\beta^{(m)}}$ corresponds to the tabloid $\tikz[tableau]\matrix{\beta^{(m)} \& |[dottedentry]| \& \beta^{(1)}\\};$, and that if this tabloid is not a tableau, then there are two adjacent columns
between which the relation $\preceq$ does not hold. These columns (as represented by some subword $c_\sigma c_\tau$ with
$\tau \not\preceq\sigma$) are rewritten to a tableau (as represented by a word in $\Sigma^*$). Thus, in terms of words in
$\Sigma^*$, tabloids are rewritten to become more `tableau-like', and the irreducible words correspond to tableaux.

\begin{lemma}
  \label{lem:noetherian}
  The rewriting system $(\Sigma,T)$ is noetherian.
\end{lemma}

\begin{proof}
  Let $\trianglelefteq$ be any total order on $\Sigma$ that extends the partial order induced by lengths of columns, in
  the sense that $|\sigma| \leq |\tau| \implies c_\sigma \trianglelefteq c_\tau$ for any two admissible columns $\sigma$
  and $\tau$.

  Let the map $L : \Sigma^* \to \nset \cup \set{0}$ send each word to the sum of the lengths of the subscripts of its
  symbols: that is,
  \[
  L\bigl(c_{\sigma^{(1)}}c_{\sigma^{(2)}}\cdots c_{\sigma^{(h)}}\bigr) = \sum_{i=1}^h |\sigma^{(i)}|.
  \]

  Define a total order $\sqsubset$ on $\Sigma^*$ by
  \begin{align*}
    u \sqsubset v \iff{}& \bigl(L(u) < L(v)\bigr) \\
    &\lor \Bigl(\bigl(L(u) = L(v)\bigr) \land \bigl(u \trianglelefteq_{\lenlex} v\bigr)\Bigl).
  \end{align*}
  That is, $\sqsubset$ first orders by the total number of symbols in the tabloid to which a word corresponds, then by
  the length of the word, and then lexicographically based on the ordering $\trianglelefteq$ of $\Sigma$. Note that
  $\sqsubset$ is compatible with multiplication in the free monoid $\Sigma^*$.

  Let $c_\alpha c_\beta$ be the left-hand side of a rewriting rule and let $w$ be its right-hand side. So
  $\beta \not\preceq \alpha$. Consider two cases:
  \begin{itemize}
  \item For $X = A_n$ (respectively, $X = B_n$, $C_n$, or $D_n$), \fullref{Lemma}{lem:an:palphabeta} (respectively,
    \ref{lem:combined:palphabeta}) shows that $P(\alpha\beta)$ contains at most $|\alpha\beta|$ symbols (so that
    $L(w) \leq L(c_\alpha c_\beta)$), and consists of at most two columns (so that $|w| \leq |\alpha\beta|$) and the
    rightmost column contains fewer than $|\alpha|$ symbols (so that $w \vartriangleleft_{\lex} \alpha\beta$). Thus
    $w \sqsubset \alpha\beta$.
  \item For $X = G_2$, \fullref{Lemma}{lem:g2:palphabeta} shows that $P(\alpha\beta)$ contains most $|\alpha\beta|$
    symbols (so $L(w) \leq L(c_\alpha c_\beta)$), and that, if $P(\alpha\beta)$ contains exactly $|\alpha\beta|$ symbols
    (so that $L(w) = L(c_\alpha c_\beta)$), then it either consists of one column (so that $|w| < |\alpha\beta|$ and so
    $w \vartriangleleft_{\lenlex} \alpha\beta$) or it consist of two columns and the rightmost column contains fewer
    than $|\alpha|$ symbols (so that $|w| = |\alpha\beta|$ and $w \vartriangleleft_{\lex} \alpha\beta$, and hence
    $w \trianglelefteq_{\lenlex} \alpha\beta$). Thus $w \sqsubset \alpha\beta$.
  \end{itemize}
  Since $\sqsubset$ is compatible with multiplication in the free monoid $\Sigma^*$, rewriting a word always decreases it
  with respect to $\sqsubset$. Since there are no infinite $\sqsubset$-descending chains, any process of rewriting must
  terminate. Hence $(\Sigma,T)$ is noetherian.
\end{proof}

\begin{lemma}
  The rewriting system $(\Sigma,T)$ is confluent.
\end{lemma}

\begin{proof}
  Let $u \in \Sigma^*$ and let $u'$ and $u''$ be words with $u \reduces u'$ and $u \reduces u''$.  By
  \fullref{Lemma}{lem:noetherian}, there are irreducible words $w' = c_{\beta^{(1)}}\cdots c_{\beta^{(k)}}$ and
  $w'' = c_{\gamma^{(1)}}\cdots c_{\gamma^{(m)}} \in \Sigma^*$ such that $u' \reduces w'$ and $u'' \reduces w''$.  Since
  $w'$ is irreducible, it does not contain the left-hand side of any rule in $T$. Thus, by the comments after the
  definition of $T$, we have $\beta^{(j+1)} \preceq \beta^{(j)}$ for $j = 1,\ldots,k-1$. That is,
  $\tikz[tableau]\matrix{\beta^{(k)} \& |[dottedentry]| \& \beta^{(1)}\\};$ is a tableau. Similarly,
  $\tikz[tableau]\matrix{\gamma^{(m)}\&|[dottedentry]|\&\gamma^{(1)}\\};$ is a tableau (with $m$ columns). But the
  readings of these tableau (that is, $\beta^{(1)}\cdots\beta^{(k)}$ and $\gamma^{(1)}\cdots\gamma^{(m)}$) are equal
  in $\Pl(X)$, and tableaux form a cross-section of $\Pl(X)$ by
\fullref{Theorem}{thm:tableauxcrossection}. Hence
  $k = m$ and $\beta^{(j)} = \gamma^{(j)}$ for $j=1,\ldots,k$, and so $w' = w''$. Thus $v = w' = w''$ is a word
  such that $u' \reduces v$ and $u'' \reduces v$. Therefore $(\Sigma,T)$ is confluent.
\end{proof}

\begin{theorem}
  For any $X \in \set{A_n,B_n,C_n,D_n,G_2}$, there is a finite complete rewriting system $(\Sigma,T)$ that presents
  $\Pl(X)$.
\end{theorem}

\begin{proof}
  Construct the finite complete rewriting system $(\Sigma,T)$ as above. It remains to prove that $\pres{\Sigma}{T}$
  presents $\Pl(X)$. To this end, let $\pres{\aX}{\drel{R}^X}$ be the presentation for $\Pl(X)$ as described in
  \cite[\S~5.1]{lecouvey_survey} and also in \fullref{Subsection}{subsec:presentations} for type $G_2$. We are going to
  prove that $\pres{\Sigma}{T}$ and $\pres{\aX}{\drel{R}^X}$ present the same monoid.

  First notice that if $\sigma = \sigma_1\cdots\sigma_k$ is an admissible column, where $\sigma_i \in
  \aX$, then a sequence of applications of rules from $T$ of type \eqref{eq:length1} lead from $c_{\sigma_1}\cdots c_{\sigma_k}$
  to $c_{\sigma_1\cdots\sigma_k}$:
  \begin{align*}
    c_{\sigma_1}c_{\sigma_2}c_{\sigma_3}\cdots c_{\sigma_{k-1}}c_{\sigma_k} &\imreduces_T c_{\sigma_1\sigma_2}c_{\sigma_3}\cdots c_{\sigma_{k-1}}c_{\sigma_k} \\
                                                                &\qquad\vdots \\
                                                                &\imreduces_T c_{\sigma_1\sigma_2\cdots\sigma_{k-1}}c_{\sigma_k} \\
                                                                &\imreduces_T c_{\sigma_1\sigma_2\cdots\sigma_{k-1}\sigma_k}.
  \end{align*}
  Thus we can apply Tietze transformations to $\pres{\Sigma}{T}$ to replace each symbol $c_{\sigma_1\cdots\sigma_k}$
  with $c_{\sigma_1}\cdots c_{\sigma_k}$ and then remove the generators $c_{\sigma_1\cdots\sigma_k}$ with $k > 1$. The
  result of this is a new presentation $\pres{\Sigma'}{T'}$ where the generating symbols in $\Sigma'$ are $c_x$ for
  $x \in \aX$, so we can replace each $c_x$ by $x$ to obtain a new presentation $\pres{\aX}{T''}$. It remains to show that
  every defining relation in $T''$ is a consequence of those in $\drel{R}^X$ and vice versa.

  Note that $T''$ can be obtained from $T$ by replacing each symbol $c_{\sigma_1\cdots\sigma_k}$ by
  $\sigma_1\cdots\sigma_k$. Thus every defining relation in $T''$ is of the form $(u,v)$, where $u$ is the reading of a
  two-column tabloid and $v$ is the reading of a tableau, and $u =_{\Pl(X)} v$. Since $\pres{\aX}{\drel{R}^X}$ presents
  $\Pl(X)$, the defining relation $(u,v)$ is a consequence of $\drel{R}^X$.

  On the other hand, let $(u,v)$ be a defining relation in $\drel{R}^X$. By inspection of the definition of $\drel{R}^X$
  in \cite[\S~5.1]{lecouvey_survey} and \fullref{Subsection}{subsec:presentations}, $v$ is the reading of a tableau, and
  $P(u) = v$. Suppose this tableau is $\tikz[tableau]\matrix{\beta^{(m)} \& |[dottedentry]| \& \beta^{(1)}\\};$, where
  $\beta^{(1)}$, \ldots, $\beta^{(m)}$ are admissible columns of type $X$. Suppose $u = u_1\cdots u_t$, and note that
  every symbol $u_i$ is an admissible column of type $X$. Since $P(u) = v$, the word $c_{u_1}\cdots c_{u_t}$ rewrites to
  $c_{\beta^{(1)}}\cdots c_{\beta^{(m)}}$ under the rewriting system $(\Sigma,T)$. Fix a sequence of rewriting
  $c_{u_1}\cdots c_{u_t} \reduces_T c_{\beta^{(1)}}\cdots c_{\beta^{(m)}}$.  Replacing each symbol
  $c_{\sigma_1\cdots\sigma_k}$ by $\sigma_1\cdots\sigma_k$ throughout this sequence of rewriting yields a sequence from
  $u = u_1\cdots u_t$ to $\beta^{(1)}\cdots \beta^{(m)} = v$ where every step is an application of a relation from
  $T''$. Hence $(u,v)$ is a consequence of $T''$.

  Since every defining relation in $T''$ is a consequence of those in $\drel{R}^X$ and vice versa, $\pres{\aX}{T''}$ and
  $\pres{\aX}{\drel{R}^X}$ present the same monoid, and thus $\pres{\Sigma}{T}$ presents $\Pl(X)$.
\end{proof}

The following corollary is immediate \cite{squier_finiteness}:

\begin{corollary}
The Plactic monoids of types $A_n$, $B_n$, $C_n$, $D_n$, and $G_2$ have finite derivation type.
\end{corollary}

By a result originally proved by Anick in a different form
\cite{anick_homology}, but also proved by various other authors (see
\cite{brown_geometry,cohen_stringrewriting}):

\begin{corollary}
The  Plactic monoids of types $A_n$, $B_n$, $C_n$, $D_n$, and $G_2$ are of type right and left $\mathrm{FP}_\infty$.
\end{corollary}

\section{Biautomaticity lemmata}
\label{sec:biautomaticity_lemmata}

In this section, we lay the groundwork for constructing biautomatic structures for plactic monoids in
\fullref{Section}{sec:biautomaticity}.

The language of representatives of the biautomatic structure will be the language
of irreducible words of the rewriting system $(\Sigma,T)$ constructed in \fullref{Section}{sec:rewriting}. To prove that
this gives us a biautomatic structure, we must understand how products of the form $c_xc_{\beta^{(1)}}\cdots
c_{\beta^{(\ell)}}$ and $c_{\beta^{(1)}}\cdots c_{\beta^{(\ell)}}c_x$ rewrite, where $c_{\beta^{(1)}}\cdots
c_{\beta^{(\ell)}}$ is an irreducible word and $c_x \in \Sigma$ is such that $|x| = 1$. It will suffice to consider the
situations where $x\beta^{(1)}\cdots\beta^{(\ell)}$ and $\beta^{(1)}\cdots\beta^{(\ell)}x$ are highest weight words,
because, as we shall see, the rewriting of $c_xc_{\beta^{(1)}}\cdots c_{\beta^{(\ell)}}$ and $c_{\beta^{(1)}}\cdots
c_{\beta^{(\ell)}}c_x$ proceeds `in the same way' in the general case.

\subsection{Two-column lemma for biautomaticity}

Let $X$ be one of the types $A_n$, $B_n$, $C_n$, or $D_n$, and let $\aX$ be the corresponding alphabet from
$\aA_n$, $\aB_n$, $\aC_n$, or $\aD_n$.

\begin{lemma}
  \label{lem:combined:alpha}
  Let $\alpha,\beta \in \aX_n^*$ be admissible $X$ columns such that $\beta \not\preceq \alpha$ and $\alpha\beta$ is a
  word of highest weight.
  \begin{enumerate}
  \item If $X = A_n$, $B_n$, or $C_n$, then $\alpha = 1\cdots p$ for some $p \in \aX[1,n]$.
  \item If $X = D_n$, then either $\alpha = 1\cdots p$ for some $p \in \aD[1,n]$ or $\alpha = 1\cdots(n-1)\bar{n}$.
  \end{enumerate}
\end{lemma}

\begin{proof}
  By \fullref{Lemma}{lem:highestweightfactors}, $\alpha$ is a highest weight column (and thus a highest-weight tableau), and
  thus has the required form by \fullref{Lemma}{lem:highestweightableauchar}.
\end{proof}

\begin{lemma}
  \label{lem:combined:hatbeta}
  Let $\alpha,\beta \in \aX_n^*$ be admissible $X$ columns such that $\beta \not\preceq \alpha$ and $\alpha\beta$ is a
  word of highest weight. Suppose the first symbol of $\beta$ is $1$. Let $\hat\beta$ be the maximal prefix of $\beta$
  whose symbols form an interval of $\aX[1,n-1]$ (viewed as an ordered set). Then $P(\alpha\beta)$ consists of two
  columns and the rightmost column of $P(\alpha\beta)$ is $\hat\beta$.
\end{lemma}

\begin{proof}
  By \fullref{Lemma}{lem:combined:alpha}, $\alpha = 1\cdots p$ for some $p$. Thus both $\alpha$ and $\beta$ contain
  $1$. Since $\alpha$ and $\beta$ are admissible columns, both containing $1$, neither contains $\bar{1}$. It follows
  that $P(\alpha\beta)$, as a tableau of the same weight as $\alpha\beta$, also contains two symbols $1$ and so has two
  columns. Suppose $P(\alpha\beta) = \tableau{\delta \& \gamma\\}$. By \fullref{Lemma}{lem:combined:palphabeta},
  $\gamma$ contains fewer than $|\alpha|$ symbols. That is, $\gamma$ contains at most $n-1$ symbols. Thus by
  \fullref{Lemma}{lem:highestweightableauchar}, since $P(\alpha\beta)$ is highest-weight, $\gamma = 1\cdots s$ for some
  $s \in \aX[1,n-1]$. Furthermore, again by \fullref{Lemma}{lem:highestweightableauchar}, if $X = A_n$, $B_n$, or $C_n$,
  then $\delta = 1\cdots t$ for some $t \in \aX[1,n]$ with $t \geq s$ (since $\delta \preceq \gamma$), while if
  $X = D_n$, then either $\delta = 1\cdots t$ for some $t \in \aX[1,n]$ with $t \geq s$ or else
  $\delta = 1\cdots(n-1)\bar{n}$. Thus $\delta$ also contains each symbol in $\aX[1,s]$ and so $\gamma\delta$ contains
  two of each symbol in $\aX[1,s]$. Since $\alpha\beta$ has the same weight as $\gamma\delta$, it follows that both
  $\alpha$ and $\beta$ contain each symbol from $\aX[1,s]$. This shows that $\hat\beta$ contains $1\cdots s$ as a
  prefix; it remains to prove that $\hat\beta$ contains no more symbols. If $s = n-1$, this is immediate by the
  definition of $\hat\beta$, so assume henceforth that $s < n-1$.

  Suppose, with the aim of obtaining a contradiction, that $\hat\beta \neq 1\cdots s$. Then, since $1\cdots s$ is a
  prefix of $\hat\beta$, it follows that $\hat\beta$ contains the symbol $s+1$. (Note that $s+1 < n$.)

  Consider $\rho_s(\gamma\delta)$. The symbol $s$ in $\gamma$ contributes ${+}$ to $\rho_s(\gamma\delta)$. If
  $\delta = 1\cdots t$ for $t \in \aX[1,n]$, then $\delta$ contributes ${+}$ (if $t = s$) or ${+}{-}$ (if $t > s$). If
  $\delta = 1\cdots (n-1)\bar{n}$, then $\delta$ contributes ${+}{+}$. In any case, $\rho_s(\gamma\delta)$ contains at least
  one ${+}$ and so $\f_s(\gamma\delta)$ is defined. Since $\alpha\beta$ and $\gamma\delta$ lie in isomorphic crystal
  components, $\f_s(\alpha\beta)$ is also defined and so $\rho_s(\alpha\beta)$ contains at least one ${+}$.

  Notice that the word $\alpha$ (which, as noted previously, is of the form $1\cdots p$) contains strictly more than $s$
  symbols and so must contain $s+1$. Thus in the calculation of $\rho_s(\alpha\beta)$ the symbols $s$ and $s+1$ in
  $\alpha$ contribute a ${+}$ and a ${-}$, which are deleted. In the word $\beta$, the symbols $s$ and $s+1$ also
  contribute a ${+}$ and a ${-}$, which are deleted. The word $\beta$ cannot contain $\bar{s+1}$, since it is
  admissible, so no other symbols can contribute a ${+}$. Hence $\rho_s(\alpha\beta)$ contains no symbols ${+}$. This is
  a contradiction, and so $\hat\beta = 1\cdots s$. This completes the proof.
\end{proof}

\subsection{Transducers}

This subsection briefly recalls the definition of a transducer and the relation it recognizes; for further background,
see \cite[Chapter~IV]{sakarovitch_automata} or \cite{berstel_transductions}.

Informally, a transducer is a (possibly non-deteministic) finite automaton that reads symbols from two tapes (possibly at
varying `speeds') and thus recognizes a binary relation between the sets of words over the two tape alphabets. More
formally, a transducer is a tuple $(Q,X,Y,I,F,\delta)$, where $Q$ is a finite set of states, $X$ and $Y$ are two finite
alphabets, $I$ is a set of distinguished \defterm{initial states}, $F$ is a set of distinguished \defterm{final states},
and $\delta$ is a finite subset of $Q \times X^* \times Y^* \times Q$ called the \defterm{transition relation}. When in
a state $q$, it can transition to a state $r$ while reading words $x \in X^*$ and $y \in Y^*$ from its top and bottom
input tapes if and only if $(q,x,y,r)$ is in $\delta$. (Note that either or both of $x$ and $y$ can be the empty word.)

The transducer accepts the contents of its input if it can start in some state in $I$, read the whole content of its
input tapes and end in a state in $F$. More formally, it accepts $(u,v) \in X^* \times Y^*$ if and only if there exist
factorizations $u = x_1\cdots x_k$ and $v = y_1\cdots y_k$, where $x_i \in X^*$ and $y_i \in Y^*$ and a sequence of
states $q_0,\ldots,q_k$ such that $q_0 \in I$, $q_k \in F$, and $(q_{i-1},x_i,y_i,q_i) \in \delta$ for $i = 1,\ldots,k$.

The transducer is thought of as a finite directed graph with vertex set $Q$ and, for each $(q,x,y,r) \in \delta$, an
edge from $q$ to $r$ labelled by $(x,y)$, for some words $x \in X^*$ and $y \in Y^*$. A pair $(u,v)$ is accepted if
there is a path from some vertex in $I$ to some vertex in $F$ such that $(u,v)$ is the product in $X^* \times Y^*$ of
the labels on that path.

Note that the set of pairs in $X^* \times Y^*$ accepted by the transducer forms a binary relation between $X^*$ and
$Y^*$, called the relation \defterm{recognized} by the transducer. A relation between $X^*$ and $Y^*$ recognized by a
transducer is said to be \defterm{rational} (see \cite[Subsection~IV.1.2]{sakarovitch_automata}).

As usual in the theory of automatic groups and semigroups, we will not describe transducers and automata by giving the
complete formal definition as tuples of sets and relations; the problem with this is that the technical details become
overpowering and obscure the fundamental ideas. Instead, we will give a somewhat higher level description of how the
transducers and automata `function'. For instance, we will sometimes speak of a transducer or automaton reading a
symbol, `storing' it in its state, and later `checking' that symbol. This means that, on reading the symbol, the
transition relation must take the transducer or automaton to a state that somehow determines the stored symbol (for
instance, states might be tuples and some component of the tuple might be the relevant symbol). `Checking' the stored
symbol means that the transducer or automaton enters a failure state if the stored symbol (as determined by the state)
is not as required.

\subsection{Left-multiplication by transducer}
\label{subsec:leftmult}

\subsubsection{\texorpdfstring{$A_n$, $B_n$, $C_n$, $D_n$}{An, Bn, Cn, Dn}}

Let $X$ be one of the types $A_n$, $B_n$, $C_n$, and $D_n$ and let $\aX$ be the corresponding alphabet from $\aA_n$,
$\aB_n$, $\aC_n$, or $\aD_n$. In these cases, the rewriting that occurs on left-multiplication by a generator is very
similar, and so we treat these cases in parallel. The goal is to prove \fullref{Lemma}{lem:anbncndn:leftmultbytrans},
which contains all the information we need for the eventual proof of biautomaticity.

We emphasize that in the following analysis, \fullref{Commuting columns lemma}{lem:commutingcolumns} is used only as an
auxiliary result to prove facts about words, and is not in any way treated as a rewriting rule.

Let $x \in \aX$ and let $\beta^{(1)},\ldots,\beta^{(m)}$ be admissible $X$ columns satisfying
$\beta^{(i+1)} \preceq \beta^{(i)}$ for $i=1,\ldots,m-1$ (that is,
$\tikz[tableau]\matrix{\beta^{(m)} \& |[dottedentry]| \& \beta^{(1)}\\};$ is a tableau), such that
$x\beta^{(1)}\cdots\beta^{(m)}$ is a highest-weight word. Recall that $x\beta^{(1)}\cdots\beta^{(h)}$ is a
highest-weight word for all $h \leq m$ by \fullref{Lemma}{lem:highestweightfactors}. In particular, $x$ is a
highest-weight word and so $x = 1$. The aim is to examine how the corresponding word over $\Sigma$ (that is,
$c_1 c_{\beta^{(1)}}\cdots c_{\beta^{(m)}}$) is rewritten by $T$ to an irreducible word. We are going to prove that this
rewriting involves a single left-to-right pass through the word and that it only changes the length of the word by at
most $1$.

The tabloid corresponding to $c_1 c_{\beta^{(1)}}\cdots c_{\beta^{(m)}}$ has the following form:
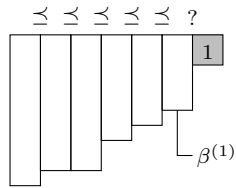
\[
\begin{tikzpicture}[x=4mm,y=-2mm]
  \draw (0,0) rectangle (1,10);
  \draw (1,0) rectangle (2,9);
  \draw (2,0) rectangle (3,9);
  \draw (3,0) rectangle (4,7);
  \draw (4,0) rectangle (5,6);
  \draw (5,0) rectangle (6,5);
  \draw[fill=lightgray] (6,0) rectangle (7,2);
  \foreach \x in {1,2,...,5} {
    \node[font=\scriptsize,anchor=south] at (\x,0) {$\preceq$};
  };
  \node[font=\scriptsize,anchor=south] at (6,0) {$?$};
  \path (7,0) ++(-.5,0) node[font=\scriptsize,anchor=north] {$1$};
  %
  \path[draw] (6,8) node[font=\scriptsize,anchor=west,inner sep=.5mm] {$\beta^{(1)}$} -| (5.5,5);
\end{tikzpicture}
\]
(The symbol $?$ indicates that either $\preceq$ or $\not\preceq$ may hold between these columns.)

First, it is possible that $\preceq$ holds between $1$ and $\beta^{(1)}$. In this case,
$c_1 c_{\beta^{(1)}}\cdots c_{\beta^{(m)}}$ is irreducible and so no rewriting occurs. So assume that $\preceq$ does not
hold between $1$ and $\beta^{(1)}$. Then a rewriting rule applies to $c_1c_{\beta^{(1)}}$. By
\fullref{Lemma}{lem:combined:palphabeta}, $P(1\beta^{(1)})$ has at most two columns. Further, again by
\fullref{Lemma}{lem:combined:palphabeta}, if it has exactly two columns, its right-hand column would be strictly
shorter than the one-symbol column $1$, which is impossible. Thus $P(1\beta^{(1)})$ has at most one column.

If $P(1\beta^{(1)})$ has zero columns (that is, is empty), then the rewriting rule that applies is
$c_1c_{\beta^{(1)}} \to \emptyword$. That is,
\[
c_1 c_{\beta^{(1)}}c_{\beta^{(2)}}\cdots c_{\beta^{(m)}} \imreduces c_{\beta^{(2)}}\cdots c_{\beta^{(m)}}.
\]
Since $\beta^{(i+1)} \preceq \beta^{(i)}$ for $i=2,\ldots,m-1$, the word $c_{\beta^{(2)}}\cdots c_{\beta^{(m)}}$ is
irreducible and no further rewriting occurs.

So assume $P(1\beta^{(1)})$ has exactly one column $\gamma^{(1)}$. Since $\gamma^{(1)}\beta^{(2)}\cdots\beta^{(m)}$ is
highest-weight, \fullref{Lemma}{lem:highestweightfactors} and \fullref{Lemma}{lem:highestweightableauchar} show in
particular that either $\gamma^{(1)} = 1\cdots s$ for some $s \in \aX[1,n]$, or $X = D_n$ and
$\gamma^{(1)} = 1\cdots (n-1)\bar{n}$. Furthermore, since $\gamma^{(1)}$ is one
of these forms and $\gamma^{(1)}$ and
$1\beta^{(1)}$ have the same weight, it follows that $\gamma^{(1)}$ contains at least two symbols. That is,
$\gamma^{(1)}$ contains $2$ or $\gamma^{(1)} = 1\bar{2}$, with the latter possible only when $X = D_n$ and $n=2$.

We now need to know about the column $\beta^{(2)}$:

\begin{lemma}
  \label{lem:leftmult:beta21}
  The column $\beta^{(2)}$ begins with $1$.
\end{lemma}

\begin{proof}
  If $\beta^{(2)} \preceq \gamma^{(1)}$, then since $\gamma^{(1)}$ begins with $1$, so does $\beta^{(2)}$. So assume
  $\beta^{(2)} \not\preceq \gamma^{(1)}$. Consider separately the cases $X = A_n, B_n, C_n, D_n$:
  \begin{itemize}
  \item Suppose $X = A_n$. Since $\gamma^{(1)}$ and $1\beta^{(1)}$ have the same weight and $\gamma^{(1)}$ contains a
    symbol $2$, it follows that the column $\beta^{(1)}$ begins with $2$. Since $\beta^{(2)} \preceq \beta^{(1)}$, the
    column $\beta^{(2)}$ must begin with either $1$ or $2$. With the aim of obtaining a contradiction, suppose it begins
    with $2$. Then $\rho_1(1\beta^{(1)}\beta^{(2)}) = {+}{-}{-}\cdots = {-}\cdots$, which contradicts
    $x\beta^{(1)}\cdots\beta^{(m)}$ being of highest weight. Thus $\beta^{(2)}$ begins with $1$.

  \item Suppose $X = B_n$ or $X =C_n$. Since $\gamma^{(1)}$ and $1\beta^{(1)}$ have the same weight and $\gamma^{(1)}$ contains a
    symbol $2$, it follows that $\beta^{(1)}$ begins with $2$ and cannot contain $\bar{2}$ or $\bar{1}$. Since
    $\beta^{(2)} \preceq \beta^{(1)}$, the column $\beta^{(2)}$ must begin with either $1$ or $2$. With the aim of
    obtaining a contradiction, suppose it begins with $2$. Then $\rho_1(\beta^{(1)})$ is ${-}$ (the $2$ at the start
    contributes ${-}$; there cannot be a further ${-}$ since there is no symbol $\bar{1}$). Hence
    $\rho_1(1\beta^{(1)}\beta^{(2)}) = {+}{-}{-}\cdots = {-}\cdots$, which contradicts $x\beta^{(1)}\ldots\beta^{(m)}$
    being of highest weight. Therefore $\beta^{(2)}$ must begin with $1$.

  \item Suppose $X = D_n$. Consider three sub-cases separately:
    \begin{itemize}
    \item Suppose that $n=2$ and $\gamma^{(1)} = 12$. Since $\gamma^{(1)}$ and $1\beta^{(1)}$ have the same
      weight, it follows that the admissible column $\beta^{(1)}$ is $2$. Then $\ell(\beta^{(1)}) = 2$. Hence either
      $\beta^{(2)} = r(\beta^{(2)}) = 2$ or $\beta^{(2)} = r(\beta^{(2)}) = 1$. With the aim of obtaining a
      contradiction, suppose the former. Then $\rho_1(1\beta^{(1)}\beta^{(2)}) = {+}{-}{-}\cdots = {-}\cdots$, which
      contradicts $x\beta^{(1)}\ldots\beta^{(m)}$ being of highest weight. Therefore $\beta^{(2)} = 1$.

    \item Suppose that $n=2$ and $\gamma^{(1)} = 1\bar{2}$. The reasoning showing that $\beta^{(2)} = 1$ proceeds in
      precisely the same way as the previous sub-case, replacing $2$ with $\bar{2}$ and $\rho_1$ with $\rho_2$.

    \item Suppose that $n > 2$. Then $\gamma^{(1)}$ contains a symbol $2$. Since $\gamma^{(1)}$ and $1\beta^{(1)}$ have
      the same weight, it follows that $\beta^{(1)}$ begins with $2$ and does not contain $\bar{2}$ or $\bar{1}$. The
      reasoning reasoning showing that $\beta^{(2)}$ begins with $1$ now proceeds in precisely the same way as for
      $X = B_n$ or $X = C_n$. \qedhere
    \end{itemize}
  \end{itemize}
\end{proof}

As a consequence of \fullref{Lemma}{lem:leftmult:beta21} and the fact that the first row of the columns $\beta^{(i)}$
must form a non-decreasing sequence since $\preceq$ holds between adjacent columns, it follows that all columns to the
left of $\beta^{(2)}$ also begin with $1$. That is, the tabloid corresponding to
$c_1 c_{\beta^{(1)}}\cdots c_{\beta^{(m)}}$ has the following form:
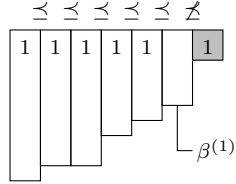
\[
\begin{tikzpicture}[x=4mm,y=-2mm]
  \draw (0,0) rectangle (1,10);
  \draw (1,0) rectangle (2,9);
  \draw (2,0) rectangle (3,9);
  \draw (3,0) rectangle (4,7);
  \draw (4,0) rectangle (5,6);
  \draw (5,0) rectangle (6,5);
  \draw[fill=lightgray] (6,0) rectangle (7,2);
  \foreach \x in {1,2,...,5} {
    \node[font=\scriptsize,anchor=south] at (\x,0) {$\preceq$};
  };
  \node[font=\scriptsize,anchor=south] at (6,0) {$\not\preceq$};
  \foreach \x in {1,2,...,5} {
    \path (\x,0) ++(-.5,0) node[font=\scriptsize,anchor=north] {$1$};
  };
  \path (7,0) ++(-.5,0) node[font=\scriptsize,anchor=north] {$1$};
  \path[draw] (6,8) node[font=\scriptsize,anchor=west,inner sep=.5mm] {$\beta^{(1)}$} -| (5.5,5);
\end{tikzpicture}
\]
In the situation we are considering, $P(1\beta^{(1)})$ is a single column $\gamma^{(1)}$, so after rewriting
$c_1 c_{\beta^{(1)}}\cdots c_{\beta^{(m)}} \imreduces c_{\gamma^{(1)}}c_{\beta^{(2)}}\cdots c_{\beta^{(m)}}$, the
corresponding tabloid has the form
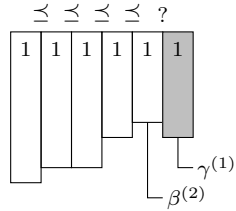
\[
\begin{tikzpicture}[x=4mm,y=-2mm]
  \draw (0,0) rectangle (1,10);
  \draw (1,0) rectangle (2,9);
  \draw (2,0) rectangle (3,9);
  \draw (3,0) rectangle (4,7);
  \draw (4,0) rectangle (5,6);
  \draw[fill=lightgray] (5,0) rectangle (6,7);
  \foreach \x in {1,2,...,4} {
    \node[font=\scriptsize,anchor=south] at (\x,0) {$\preceq$};
  };
  \node[font=\scriptsize,anchor=south] at (5,0) {$?$};
  \foreach \x in {1,2,...,5} {
    \path (\x,0) ++(-.5,0) node[font=\scriptsize,anchor=north] {$1$};
  };
  \path (6,0) ++(-.5,0) node[font=\scriptsize,anchor=north] {$1$};
  \path[draw] (5,11) node[font=\scriptsize,anchor=west,inner sep=.5mm] {$\beta^{(2)}$} -| (4.5,6);
  \path[draw] (6,9) node[font=\scriptsize,anchor=west,inner sep=.5mm] {$\gamma^{(1)}$} -| (5.5,7);
\end{tikzpicture}
\]
If $\beta^{(2)} \preceq \gamma^{(1)}$, then $c_{\gamma^{(1)}}c_{\beta^{(2)}}\cdots c_{\beta^{(m)}}$ is
irreducible and no further rewriting occurs. So assume $\beta^{(2)} \not\preceq \gamma^{(1)}$. Note that
$\gamma^{(1)}\beta^{(2)}\ldots\beta^{(m)}$ is also a highest weight word.

For $j = 2,\ldots,m$, define $\hat\beta^{(j)}$ to be the longest contiguous prefix of $\beta^{(j)}$ containing only
symbols from $\aX[1,n-1]$. Note that because $\beta^{(j+1)} \preceq \beta^{(j)}$, each symbol of $\beta^{(j+1)}$
is less than or equal to the symbol of $\beta^{(j)}$ in the same row. Thus the prefix $\hat\beta^{(j+1)}$ must be at
least as long as the prefix $\hat\beta^{(j)}$, and so $\hat\beta^{(j+1)} \preceq \hat\beta^{(j)}$. So the situation is
as follows, where the horizontal lines in each column indicate the end of $\hat\beta^{(j)}$:
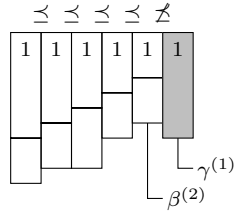
\[
\begin{tikzpicture}[x=4mm,y=-2mm]
  \draw (0,0) rectangle (1,7) rectangle (0,10);
  \draw (1,0) rectangle (2,6) rectangle (1,9);
  \draw (2,0) rectangle (3,5) rectangle (2,9);
  \draw (3,0) rectangle (4,4) rectangle (3,7);
  \draw (4,0) rectangle (5,3) rectangle (4,6);
  \draw[fill=lightgray] (5,0) rectangle (6,7);
  \foreach \x in {1,2,...,4} {
    \node[font=\scriptsize,anchor=south] at (\x,0) {$\preceq$};
  };
  \node[font=\scriptsize,anchor=south] at (5,0) {$\not\preceq$};
  \foreach \x in {1,2,...,5} {
    \path (\x,0) ++(-.5,0) node[font=\scriptsize,anchor=north] {$1$};
  };
  \path (6,0) ++(-.5,0) node[font=\scriptsize,anchor=north] {$1$};
  \path[draw] (5,11) node[font=\scriptsize,anchor=west,inner sep=.5mm] {$\beta^{(2)}$} -| (4.5,6);
  \path[draw] (6,9) node[font=\scriptsize,anchor=west,inner sep=.5mm] {$\gamma^{(1)}$} -| (5.5,7);
\end{tikzpicture}
\]

Since $\beta^{(2)}$ begins with $1$, by \fullref{Lemma}{lem:combined:hatbeta}, $P(\gamma^{(1)}\beta^{(2)})$ has two
columns, the rightmost of which is $\hat\beta^{(2)}$. Let $\gamma^{(2)}$ be the left column of
$P(\gamma^{(1)}\beta^{(2)})$. So we have
\[
c_{\gamma^{(1)}}c_{\beta^{(2)}}\ldots c_{\beta^{(m)}} \imreduces c_{\hat\beta^{(2)}}c_{\gamma^{(2)}}c_{\beta^{(3)}}\ldots c_{\beta^{(m)}}
\]
If $\beta^{(3)}\preceq\gamma^{(2)}$, the word $c_{\hat\beta^{(2)}}c_{\gamma^{(2)}}c_{\beta^{(3)}}\ldots c_{\beta^{(m)}}$
is irreducible. So suppose $\beta^{(3)} \not\preceq \gamma^{(2)}$. We claim $\gamma^{(2)}\beta^{(3)}$ is a highest
weight word. This follows since $\hat\beta^{(2)}$ is a prefix of both $\gamma^{(2)}$ and $\beta^{(3)}$ (since it is a
prefix of $\hat\beta^{(3)}$) and so commutes with both by the \fullref{Commuting columns
  lemma}{lem:commutingcolumns}. Thus
$\hat\beta^{(2)}\gamma^{(2)}\beta^{(3)} =_{\Pl(X)} \gamma^{(2)}\beta^{(3)}\hat\beta^{(2)}$ and so, by
\fullref{Lemma}{lem:highestweightfactors}, $\gamma^{(2)}\beta^{(3)}$ is highest weight. Thus, again by
\fullref{Lemma}{lem:combined:hatbeta}, $P(\gamma^{(2)}\beta^{(3)})$ has two columns, the rightmost of which is
$\hat\beta^{(3)}$.
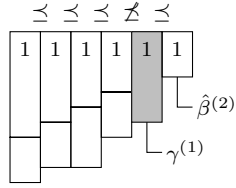
\[
\begin{tikzpicture}[x=4mm,y=-2mm]
  \draw (0,0) rectangle (1,7) rectangle (0,10);
  \draw (1,0) rectangle (2,6) rectangle (1,9);
  \draw (2,0) rectangle (3,5) rectangle (2,9);
  \draw (3,0) rectangle (4,4) rectangle (3,7);
  \draw[fill=lightgray] (4,0) rectangle (5,6);
  \draw (5,0) rectangle (6,3);
  \foreach \x in {1,2,3,5} {
    \node[font=\scriptsize,anchor=south] at (\x,0) {$\preceq$};
  };
  \node[font=\scriptsize,anchor=south] at (4,0) {$\not\preceq$};
  \foreach \x in {1,2,...,5} {
    \path (\x,0) ++(-.5,0) node[font=\scriptsize,anchor=north] {$1$};
  };
  \path (6,0) ++(-.5,0) node[font=\scriptsize,anchor=north] {$1$};
  \path[draw] (6,5) node[font=\scriptsize,anchor=west,inner sep=.5mm] {$\hat\beta^{(2)}$} -| (5.5,3);
  \path[draw] (5,8) node[font=\scriptsize,anchor=west,inner sep=.5mm] {$\gamma^{(1)}$} -| (4.5,6);
\end{tikzpicture}
\]

Continuing in this way, we inductively obtain a sequence of admissible columns $\gamma^{(2)},\ldots,\gamma^{(k)}$ for
some maximal $k \leq m$ such that the following hold for $j=1,\ldots,k-1$:
\begin{itemize}
\item $\beta^{(j+1)} \not\preceq \gamma^{(j)}$.
\item $\gamma^{(j)}\beta^{(j+1)}$ is highest weight. This follows since $\hat\beta^{(2)},\ldots,\hat\beta^{(j)}$ are all
  prefixes of $\gamma^{(j)}$ and $\beta^{(j+1)}$, so commute with both by the \fullref{Commuting columns
    lemma}{lem:commutingcolumns}, and thus
  \[
    \hat\beta^{(2)}\cdots\hat\beta^{(j)}\gamma^{(j)}\beta^{(j+1)} =_{\Pl(X)} \gamma^{(j)}\beta^{(j+1)}\hat\beta^{(2)}\cdots\hat\beta^{(j)}
  \]
  and so, by \fullref{Lemma}{lem:highestweightfactors}, $\gamma^{(2)}\beta^{(3)}$ is highest weight.
\item $P(\gamma^{(j)}\beta^{(j+1)}) = \tableau{\gamma^{(j+1)} \& \hat\beta^{(j+1)} \\}$
\end{itemize}
Therefore rewriting proceeds as follows:
\begin{equation}
  \label{eq:rightmultrewriting}
  \begin{aligned}
    &c_{\gamma^{(1)}}c_{\beta^{(2)}}\cdots c_{\beta^{(m)}} \\
    \imreduces{}&c_{\hat\beta^{(2)}}c_{\gamma^{(2)}}c_{\beta^{(3)}}\cdots c_{\beta^{(m)}} \\
    \imreduces{}&c_{\hat\beta^{(2)}}c_{\hat\beta^{(3)}}c_{\gamma^{(3)}}c_{\beta^{(4)}}\cdots c_{\beta^{(m)}} \\
    &\quad\vdots\\
    \imreduces{}&c_{\hat\beta^{(2)}}\cdots c_{\hat\beta^{(k)}}c_{\gamma^{(k)}}c_{\beta^{(k+1)}}\cdots c_{\beta^{(m)}},
  \end{aligned}
\end{equation}
The maximality of $k$ means either that $k=m$ or $\beta^{(k+1)} \preceq \gamma^{(k)}$; in either case the word
$c_{\hat\beta^{(2)}}\cdots c_{\hat\beta^{(k)}}c_{\gamma^{(k)}}c_{\beta^{(k+1)}}\cdots c_{\beta^{(m)}}$ is irreducible
since $\beta^{(j+1)} \preceq \beta^{(j)}$ for all $j$. The corresponding
tabloid is now a tableau of the form:
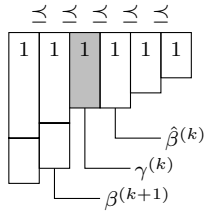
\[
\begin{tikzpicture}[x=4mm,y=-2mm]
  \draw (0,0) rectangle (1,7) rectangle (0,10);
  \draw (1,0) rectangle (2,6) rectangle (1,9);
  \draw[fill=lightgray] (2,0) rectangle (3,5);
  \draw (3,0) rectangle (4,5);
  \draw (4,0) rectangle (5,4);
  \draw (5,0) rectangle (6,3);
  \foreach \x in {1,2,...,5} {
    \node[font=\scriptsize,anchor=south] at (\x,0) {$\preceq$};
  };
  \foreach \x in {1,2,...,6} {
    \path (\x,0) ++(-.5,0) node[font=\scriptsize,anchor=north] {$1$};
  };
  \path[draw] (5,7) node[font=\scriptsize,anchor=west,inner sep=.5mm] {$\hat\beta^{(k)}$} -| (3.5,5);
  \path[draw] (4,9) node[font=\scriptsize,anchor=west,inner sep=.5mm] {$\gamma^{(k)}$} -| (2.5,5);
  \path[draw] (3,11) node[font=\scriptsize,anchor=west,inner sep=.5mm] {$\beta^{(k+1)}$} -| (1.5,9);
\end{tikzpicture}
\]

Thus far, we have analyzed the rewriting that occurs \emph{at highest weight} when a tableau is left-multiplied by a
generator. However, as we shall see, we can now deduce information about the rewriting that occurs in general.

Recall that a word $c_{\alpha^{(1)}}\cdots c_{\alpha^{(k)}} \in \Sigma^*$ represents the tabloid
$\tikz[tableau]\matrix{\alpha^{(k)} \& |[dottedentry]| \& \alpha^{(1)}\\};$. As discussed following
\fullref{Lemma}{lem:efonfactors}, we can think of applying the operators $\e_i$ and $\f_i$ to a tabloid. Thus we can
think of applying them to words in $\Sigma$: the result of applying $\e_i$ or $\f_i$ to
$c_{\alpha^{(1)}}\cdots c_{\alpha^{(k)}} \in \Sigma^*$ is $c_{\beta^{(1)}}\cdots c_{\beta^{(k)}}$, where
$\tikz[tableau]\matrix{\beta^{(k)} \& |[dottedentry]| \& \beta^{(1)}\\};$ is the result of applying the operator to
$\tikz[tableau]\matrix{\alpha^{(k)} \& |[dottedentry]| \& \alpha^{(1)}\\};$. Recall that the operators $\e_i$ and $\f_i$
preserve whether the $\preceq$ relation holds between adjacent columns. Thus the operators $\e_i$ and $\f_i$ preserve
whether the $\preceq$ relation holds between adjacent subscripts of a word in $\Sigma^*$.

Now let $\beta^{(1)},\ldots,\beta^{(m)}$ be admissible $X$ columns, such that $\beta^{(i+1)} \preceq \beta^{(i)}$ for
$i=1,\ldots,m-1$ (that is, $\tikz[tableau]\matrix{\beta^{(m)} \& |[dottedentry]| \& \beta^{(1)}\\};$ is an $X$ tableau),
and let $x \in \aX$. Let $\e_{i_1},\ldots,\e_{i_k}$ be such that
$w = \e_{i_1}\cdots\e_{i_k}(c_xc_{\beta^{(1)}}\cdots c_{\beta^{(m)}})$ is highest weight. The rewriting of the word $w$
to normal form proceeds as described above, via a single left-to-right pass through the word. In particular, until the
normal form is reached, there is exactly one pair of adjacent symbols where the relation $\preceq$ does not hold between
the adjacent subscripts. Now apply $\f_{i_k}\cdots\f_{i_1}$ to every word in the sequence of rewriting; this gives a
sequence of words starting at $c_xc_{\beta^{(1)}}\cdots c_{\beta^{(m)}}$. Furthemore, since the $\f_i$ preserve whether
the $\preceq$ relation holds between adjacent subscripts, there is exactly one place in each word where a rewriting rule
can be applied. Since rewriting rules correspond to crystal isomorphisms, which are also preserved by the $\f_i$, the
rewriting rule that can be applied results in the next word in the sequence of rewriting. Thus we have a sequence of
rewriting from $c_xc_{\beta^{(1)}}\cdots c_{\beta^{(m)}}$ that also proceeds via a single left-to-right pass.

The aim is now to show that a transducer can recognize the relation consisting of pairs $(u,v)$, where
$u,v \in \Sigma^*$ are irreducible and $c_xu \reduces v$, by essentially computing this rewriting. First, note that the
transducer can check that the words on both tapes are irreducible: it simply stores the previously-read symbol in its
state and checks that the previously-read and next symbols do not form the left-hand side of a rewriting rule. We will
assume henceforth that the transducer is doing this and focus on how it
computes the rewriting.

The computation is performed as follows. It reads the word $c_{\beta^{(1)}}\cdots c_{\beta^{(m)}}$ on its first tape. It
stores one symbol in its state, starting with $c_x$. (Note that $c_x$ is not read from input; the transducer is
recognizing pairs $(u,v)$ such that $c_xu \reduces v$.) At each later step, storing some $c_\alpha$ in its state, it
reads the next symbol $c_{\beta^{(j)}}$, applies the rewriting rule $c_\alpha c_{\beta^{(j)}} \to c_\gamma c_\delta$,
checks that the next symbol on its second tape is $c_\gamma$, and replaces the stored symbol $c_\alpha$ with
$c_\delta$. In the case where $c_\alpha c_{\beta^{(j)}}$ is not the left-hand side of a rewritin rule, the transducer
simply checks that the next symbols on its second tape are $c_\alpha$ and
$c_{\beta^{(j)}}$, then reads the rest of both
tapes, checking that symbols on both tapes match. Note that this relies on rewriting proceeding as described above, via
a single left-to-right pass.

In summary, we have proven the following lemma:

\begin{lemma}
  \label{lem:anbncndn:leftmultbytrans}
  Let $X$ be one of the types $A_n$, $B_n$, $C_n$, and $D_n$ and let $\aX$ be the corresponding alphabet from $\aA_n$,
  $\aB_n$, $\aC_n$, or $\aD_n$.  Let $\Sigma$ and $T$ be alphabet and set of rewriting rules constucted for type $X$ in
  \fullref{Subsection}{subsec:rewritingconstruction}. Let $x \in \aX$. Let $L \subseteq \Sigma^*$ be the languages of
  irreducible words. Then the relation
  \[
    {}_{c_x}L = \gset[\big]{(u,v) \in L\times L}{c_xu =_{\Pl(X)} v}
  \]
  is recognized by a transducer. Furthermore, if $(u,v)$ is a pair in this relation, then the lengths of $u$ and $v$
  differ by at most $1$.
\end{lemma}

\subsubsection{\texorpdfstring{$G_2$}{G2}}

Let $x \in \aG_2$ and let $\beta^{(1)},\cdots,\beta^{(m)}$ be admissible $G_2$ columns satisfying
$\beta^{(i+1)} \preceq \beta^{(i)}$ for $i=1,\cdots,m-1$ (that is,
$\tableau{\beta^{(m)} \& |[dottedentry]| \&\beta^{(1)}\\}$ is a tableau), such that $x\beta^{(1)}\cdots\beta^{(m)}$ is a
highest-weight word. Recall that $x\beta^{(1)}\cdots\beta^{(h)}$ is a highest-weight word for all $h \leq m$ by
\fullref{Lemma}{lem:highestweightfactors}. In particular, $x$ is a highest-weight word and so $x = 1$.

We are going to analyze how $c_xc_{\beta^{(1)}}\cdots c_{\beta^{(m)}}$ rewrites to normal form. Like for $A_n$, $B_n$,
$C_n$, and $D_n$, the aim is to prove that this rewriting involves a single left-to-right pass through the
word. However, we require a fairly complicated analysis of cases, shown in \fullref{Table}{tbl:g2:leftmultcases}. In the
table, every possible admissible column is listed as a possibility for $\beta^{(1)}$. In those cases where we also have
to consider $\beta^{(2)}$ or $\beta^{(3)}$, there are fewer possibilities because of the restriction
$\beta^{(3)} \preceq \beta^{(2)} \preceq \beta^{(1)}$. Most of these cases are ruled out by the requirement that
$x\beta^{(1)}\beta^{(2)}\beta^{(3)}$ is of highest weight. For example, the case where $\beta^{(1)} = 2$ and
$\beta^{(2)}=2$ is impossible, because
\[
\begin{rhoarray}{rcccll}
  \rho_{1}(x\beta^{(1)}\beta^{(2)}\cdots) ={}& \soverbrace{{+}}^{x} &\, \soverbrace{{-}}^{\beta^{(1)}} & \soverbrace{{-}}^{\beta^{(2)}} & \cdots &{}={-}\cdots; \\
\end{rhoarray}
\]
All other cases listed as `not highest weight' in \fullref{Table}{tbl:g2:leftmultcases} are ruled out in the same way,
by considering either $\rho_1$ or $\rho_2$.

\begin{table}
  \centering
  \caption{Cases for left-multiplication in $G_2$}
  \label{tbl:g2:leftmultcases}
  \begin{tabular}{lllll}
    \toprule
    $x=1$ & $\beta^{(1)}=1$              & \multicolumn{3}{l}{Case 1}                                                 \\ \cmidrule(lr) {2-5}
          & $\beta^{(1)}=2$              & $\beta^{(2)}=1$        & \multicolumn{2}{l}{Case 2}                        \\
          &                              & $\beta^{(2)}=2$        & \multicolumn{2}{l}{Not highest weight ($\rho_1$)} \\
          &                              & $\beta^{(2)}=12$       & \multicolumn{2}{l}{Case 3}                        \\
          &                              & $\beta^{(2)}=13$       & \multicolumn{2}{l}{Case 4}                        \\
          &                              & $\beta^{(2)}=23$       & \multicolumn{2}{l}{Not highest weight ($\rho_1$)} \\
          &                              & $\beta^{(2)}=20$       & \multicolumn{2}{l}{Not highest weight ($\rho_1$)} \\
          &                              & $\beta^{(2)}=2\bar{3}$ & \multicolumn{2}{l}{Not highest weight ($\rho_1$)} \\ \cmidrule(lr) {2-5}
          & $\beta^{(1)}=3$              & \multicolumn{3}{l}{Not highest weight ($\rho_2$)}                          \\ \cmidrule(lr) {2-5}
          & $\beta^{(1)}=0$              & $\beta^{(2)}=1$        & Case 5                                            \\ \cmidrule(lr) {3-5}
          &                              & $\beta^{(2)}=2$        & $\beta^{(3)}=1$  & Case 6                         \\
          &                              &                        & $\beta^{(3)}=2$  & Not highest weight ($\rho_1$)  \\
          &                              &                        & $\beta^{(3)}=12$ & Case 7                         \\
          &                              &                        & $\beta^{(3)}=13$ & Case 8                         \\
          &                              &                        & $\beta^{(3)}=23$ & Not highest weight ($\rho_1$)  \\
          &                              &                        & $\beta^{(3)}=20$ & Not highest weight ($\rho_1$)  \\ \cmidrule(lr) {3-5}
          &                              & $\beta^{(2)}=3$        & \multicolumn{2}{l}{Not highest weight ($\rho_2$)} \\
          &                              & $\beta^{(2)}=12$       & Case 9                                            \\
          &                              & $\beta^{(2)}=13$       & \multicolumn{2}{l}{Not highest weight ($\rho_2$)} \\
          &                              & $\beta^{(2)}=23$       & Case 10                                           \\
          &                              & $\beta^{(2)}=20$       & \multicolumn{2}{l}{Not highest weight ($\rho_1$)} \\ \cmidrule(lr) {2-5}
          & $\beta^{(1)}=\bar{3}$        & \multicolumn{3}{l}{Not highest weight ($\rho_1$)}                          \\
          & $\beta^{(1)}=\bar{2}$        & \multicolumn{3}{l}{Not highest weight ($\rho_2$)}                          \\
          & $\beta^{(1)}=\bar{1}$        & Case 11                                                                    \\
          & $\beta^{(1)}=12$             & Case 12                                                                    \\
          & $\beta^{(1)}=13$             & \multicolumn{3}{l}{Not highest weight ($\rho_2$)}                          \\
          & $\beta^{(1)}=23$             & Case 13                                                                    \\
          & $\beta^{(1)}=00$             & Case 14                                                                    \\
          & $\beta^{(1)}=20$             & \multicolumn{3}{l}{Not highest weight ($\rho_1$)}                          \\
          & $\beta^{(1)}=2\bar{3}$       & \multicolumn{3}{l}{Not highest weight ($\rho_1$)}                          \\
          & $\beta^{(1)}=0\bar{3}$       & \multicolumn{3}{l}{Not highest weight ($\rho_1$)}                          \\
          & $\beta^{(1)}=3\bar{3}$       & \multicolumn{3}{l}{Not highest weight ($\rho_2$)}                          \\
          & $\beta^{(1)}=30$             & \multicolumn{3}{l}{Not highest weight ($\rho_2$)}                          \\
          & $\beta^{(1)}=3\bar{2}$       & \multicolumn{3}{l}{Not highest weight ($\rho_2$)}                          \\
          & $\beta^{(1)}=0\bar{2}$       & \multicolumn{3}{l}{Not highest weight ($\rho_2$)}                          \\
          & $\beta^{(1)}=\bar{3}\bar{2}$ & \multicolumn{3}{l}{Not highest weight ($\rho_1$)}                          \\
          & $\beta^{(1)}=\bar{3}\bar{1}$ & \multicolumn{3}{l}{Not highest weight ($\rho_1$)}                          \\
          & $\beta^{(1)}=\bar{2}\bar{1}$ & \multicolumn{3}{l}{Not highest weight ($\rho_2$)}                          \\
    \bottomrule
  \end{tabular}
\end{table}

There are fourteen remaining cases in \fullref{Table}{tbl:g2:leftmultcases}, but we reassure the reader that many of
these quickly only result in one or two rewriting steps, and in the others the rewriting behaves in a straightforward
way. Let us consider each of these cases in turn.
\begin{itemize}

\item\textit{Case 1}. $\beta^{(1)} = 1$. Then $\beta^{(1)} \preceq x$ and so no rewriting occurs: the word
$c_xc_{\beta^{(1)}}\cdots c_{\beta^{(m)}}$ is in normal form.

\item\textit{Cases 2--4}. $\beta^{(1)} = 2$ and $\beta^{(2)} \in \set{1,12,13}$. Now, since
$\beta^{(j+1)} \preceq \beta^{(j)}$ for all $j$, the columns $\beta^{(2)}$, \ldots, $\beta^{(m)}$ consist of zero or
more columns $1$, followed by zero or more columns $13$, followed by zero or more columns $12$. Notice that this
subsumes the three possibilities for $\beta^{(2)}$.  Note that $P(x\beta^{(1)}) = \tableau{12\\}$, and
$P(121) = \tableau{12 \& 1\\}$, and $P(1213) = \tableau{1 \& 1 \& 1\\}$, as can be seen from
\fullref{Table}{tbl:g2:palphabeta} (see page~\pageref{tbl:g2:palphabeta}).

When there is at least one column $13$, rewriting begins
\begin{align*}
c_1c_2c_1^{p}c_{13}^{q}c_{12}^{r}
\imreduces{}& c_{12}c_1^{p}c_{13}^{q}c_{12}^{r} \\
\reduces{}& c_1^{p}c_{12}c_{13}^{q}c_{12}^{r} \\
\imreduces{}& \begin{cases}
c_1^{p}c_{12}^{r+1} & \text{if $q = 0$,} \\
c_1^{p+3}c_{13}^{q-1}c_{12}^{r}  & \text{if $q \geq 1$.}
\end{cases}
\end{align*}
In either case, the final word is in normal form since $1 \preceq 13$ and $1 \preceq 12$.

\item\textit{Case 5}. $\beta^{(1)}=0$ and $\beta^{(2)}=1$. Then $P(x\beta^{(1)}) = \tableau{1\\}$. Since $1 \preceq 1$, the rewriting to
normal form is simply
\[
c_1c_0c_{1}c_{\beta^{(3)}}\cdots c_{\beta^{(m)}}
\imreduces c_{1}c_{1}c_{\beta^{(3)}}\cdots c_{\beta^{(m)}}.
\]

\item\textit{Cases 6--8}. $\beta^{(1)}=0$, $\beta^{(2)}=2$, and $\beta^{(3)} \in \set{1,12,13}$. Since $\beta^{(j+1)} \preceq \beta^{(j)}$
for all $j$, the columns $\beta^{(3)}$, \ldots, $\beta^{(m)}$ consist of zero or more columns $1$, followed by zero or more columns
$13$, followed by zero or more columns $12$. Since $P(1\beta^{(1)}) = \tableau{1\\}$ and $P(1\beta^{(2)}) = \tableau{12\\}$, rewriting
proceeds in one of two ways, similarly to case~2. If there is a column $13$ present, rewriting proceeds
\begin{align*}
c_1c_0c_2c_1\cdots c_1c_{13}c_{\beta^{(k)}}\cdots c_{\beta^{(m)}}
\imreduces{}& c_1c_2c_1\cdots c_1c_{13}c_{\beta^{(k)}}\cdots c_{\beta^{(m)}} \\
\imreduces{}& c_{12}c_1\cdots c_1c_{13}c_{\beta^{(k)}}\cdots c_{\beta^{(m)}} \\
\reduces{}& c_1\cdots c_1c_{12}c_{13}c_{\beta^{(k)}}\cdots c_{\beta^{(m)}} \\
\imreduces{}& c_1\cdots c_1c_1c_1c_1c_{\beta^{(k)}}\cdots c_{\beta^{(m)}}.
\end{align*}
This word is in normal form since, regardless of whether $\beta^{(k)}$ is $12$ or $13$, we have $1 \preceq \beta^{(k)}$.
When there is no column $13$, the columns $1$ are followed immediately by columns $12$, and so rewriting begins
\begin{align*}
c_1c_0c_2c_1\cdots c_1c_{12}c_{\beta^{(k)}}\cdots c_{\beta^{(m)}}
\imreduces{}& c_1c_2c_1\cdots c_1c_{12}c_{\beta^{(k)}}\cdots c_{\beta^{(m)}} \\
\imreduces{}& c_{12}c_1\cdots c_1c_{12}c_{\beta^{(k)}}\cdots c_{\beta^{(m)}} \\
\reduces{}& c_1\cdots c_1c_{12}c_{12}c_{\beta^{(k)}}\cdots c_{\beta^{(m)}},
\end{align*}
whichs is in normal form.

\item\textit{Case 9}. $\beta^{(1)}=0$ and $\beta^{(2)}=12$. Then $P(x\beta^{(1)}) = \tableau{1\\}$. Since $12 \preceq 1$, the rewriting to
normal form is simply
\[
c_1c_0c_{12}c_{\beta^{(3)}}\cdots c_{\beta^{(m)}}
\imreduces c_{1}c_{12}c_{\beta^{(3)}}\cdots c_{\beta^{(m)}}.
\]

\item\textit{Case 10}. $\beta^{(1)}=0$ and $\beta^{(2)}=23$. Now, since $\beta^{(j+1)} \preceq \beta^{(j)}$ for all $j$, the remaining
columns $\beta^{(3)}$, \ldots, $\beta^{(m)}$ consist of zero or more columns $23$, zero or more columns $13$, and zero or
more columns $12$. Note that $P(x\beta^{(1)}) = \tableau{1\\}$, and $P(1\beta^{(2)}) = \tableau{1 \& 1\\}$, as
can be seen from \fullref{Table}{tbl:g2:palphabeta}. So rewriting proceeds as follows:
\begin{align*}
c_1c_0c_{23}c_{23}\cdots c_{23}c_{\beta^{(k)}}\cdots c_{\beta^{(m)}}
\imreduces{}& c_1c_{23}c_{23}\cdots c_{23}c_{\beta^{(k)}}\cdots c_{\beta^{(m)}} \\
\imreduces{}& c_1c_1c_{23}\cdots c_{23}c_{\beta^{(k)}}\cdots c_{\beta^{(m)}} \\
\reduces{}& c_1c_1c_1\cdots c_1c_{\beta^{(k)}}\cdots c_{\beta^{(m)}}.
\end{align*}
Regardless of whether $\beta^{(k)}$ is $13$ or $12$, we have $\beta^{(k)} \preceq 1$, so this word is in normal
form. Note that there is exactly one symbol $c_1$ in the final word for each symbol $c_{23}$ in the initial word.

\item\textit{Case 11}. $\beta^{(1)}=\bar{1}$. Since $P(1\bar{1}) = \emptyword$, as can be seen from
\fullref{Table}{tbl:g2:palphabeta}, the rewriting to normal form is simply
\[
c_1c_{\bar{1}}c_{\beta^{(3)}}\cdots c_{\beta^{(m)}}
\imreduces c_{\beta^{(3)}}\cdots c_{\beta^{(m)}}.
\]

\item\textit{Case 12}. $\beta^{(1)}=12$. Since $12 \preceq 1$, the word $c_1c_{12}c_{\beta^{(2)}}\cdots c_{\beta^{(m)}}$ is in normal form.

\item\textit{Case 13}. $\beta^{(1)}=23$. Since $\beta^{(j+1)} \preceq \beta^{(j)}$ for all $j$, the remaining
columns $\beta^{(2)}$, \ldots, $\beta^{(m)}$ consist of zero or more columns $23$, zero or more columns $13$, and zero or
more columns $12$. Note that $P(x\beta^{(1)}) = \tableau{1 \& 1\\}$, as
can be seen from \fullref{Table}{tbl:g2:palphabeta}. So rewriting proceeds as follows:
\begin{align*}
c_1c_{23}c_{23}\cdots c_{23}c_{\beta^{(k)}}\cdots c_{\beta^{(m)}}
\imreduces{}& c_1c_1c_{23}\cdots c_{23}c_{\beta^{(k)}}\cdots c_{\beta^{(m)}} \\
\reduces{}& c_1c_1c_1\cdots c_1c_{\beta^{(k)}}\cdots c_{\beta^{(m)}}.
\end{align*}
Regardless of whether $\beta^{(k)}$ is $13$ or $12$, we have $\beta^{(k)} \preceq 1$, so this word is in normal form.

\item\textit{Case 14}. $\beta^{(1)}=00$. Since $\beta^{(2)} \preceq \beta^{(1)}$, it follows that $\beta^{(2)}$ is either $13$ or
$12$ (note that $00 \not\preceq 00$). Since $P(x\beta^{(1)}) = \tableau{1\\}$, as can be seen from
\fullref{Table}{tbl:g2:palphabeta}, the rewriting to normal form is simply
\[
c_1c_{00}c_{\beta^{(2)}}\cdots c_{\beta^{(m)}}
\imreduces c_1c_{\beta^{(2)}}\cdots c_{\beta^{(m)}}.
\]
Regardless of whether $\beta^{(2)}$ is $13$ or $12$, we have $\beta^{(2)} \preceq 1$, so this word is in normal form.

\end{itemize}
This completes the case analysis. Note that in each case, the lengths of $c_x c_{\beta^{(1)}}\cdots c_{\beta^{(m)}}$ and
its corresponding normal form differ by at most $2$. (The maximum difference $2$ occurs in cases 5--7 and 10.)

As in the discussion before \fullref{Lemma}{lem:anbncndn:leftmultbytrans}, the
way rewriting proceeds at highest weight
is mirrored in how it proceeds in general for type $G_2$. Thus, with the analysis above, we can prove the following
analogue of \fullref{Lemma}{lem:anbncndn:leftmultbytrans}:

\begin{lemma}
  \label{lem:g2:leftmultbytrans}
  Let $\Sigma$ and $T$ be alphabet and set of rewriting rules constucted for type $G_2$ in
  \fullref{Subsection}{subsec:rewritingconstruction}. Let $x \in \aG_2$. Let $L \subseteq \Sigma^*$ be the languages of
  irreducible words. Then the relation
  \[
    {}_{c_x}L = \gset[\big]{(u,v) \in L\times L}{c_xu =_{\Pl(G_2)} v}
  \]
  is recognized by an transducer. Furthermore, if $(u,v)$ is a pair in this relation, then the lengths of $u$ and $v$
  differ by at most $1$.
\end{lemma}

\subsection{Right-multiplication by transducer}
\label{subsec:rightmult}

We now turn our attention to right-multiplication. Unlike left-multiplication, the cases $A_n$, $B_n$, $C_n$, and $D_n$
are sufficiently different that we have to consider them separately.

\subsubsection{\texorpdfstring{$A_n$}{An}}

Let $\beta^{(1)},\ldots,\beta^{(m)}$ be admissible $A_n$ columns and let $x \in \aA_n$ be such that
$\beta^{(i+1)} \preceq \beta^{(i)}$ for $i=1,\ldots,m-1$ (that is,
$\tikz[tableau]\matrix{\beta^{(m)} \& |[dottedentry]| \& \beta^{(1)} \\};$ is an $A_n$ tableau), and such that
$\beta^{(1)}\cdots\beta^{(m)}x$ is a highest-weight word. We are going to examine how the corresponding word over
$\Sigma$ (that is, $c_{\beta^{(1)}}\cdots c_{\beta^{(m)}}c_x$) rewrites to an irreducible word. The aim is to prove that
this rewriting involves a single right-to-left pass through the word.

First, note that since the prefix $\beta^{(1)}\cdots\beta^{(m)}$ is a highest-weight word by
\fullref{Lemma}{lem:highestweightfactors}, each column $\beta^{(i)}$ is of the form $1\cdots p_i$ for some
$p_i \in \aA[1,n]$ and $p_{i+1} \geq p_i$ for $i=1,\ldots,m-1$ by \fullref{Lemma}{lem:highestweightableauchar}. That is,
the tabloid corresponding to the word $\beta^{(1)}\cdots\beta^{(m)}x$ is of the form:
\[
\begin{tikzpicture}[x=4mm,y=-2mm]
  \begin{scope}
    \draw (1,0) rectangle (2,10);
    \draw (2,0) rectangle (3,9);
    \draw (3,0) rectangle (4,9);
    \draw (4,0) rectangle (5,7);
    \draw (5,0) rectangle (6,6);
    \draw (6,0) rectangle (7,5);
    \draw (7,0) rectangle (8,5);
    \draw (8,0) rectangle (9,2);
    \draw (9,0) rectangle (10,2);
    \draw[fill=lightgray] (0,0) rectangle (1,2);
    \foreach \x in {2,3,...,9} {
      \node[font=\scriptsize,anchor=south] at (\x,0) {$\preceq$};
    };
    \node[font=\scriptsize,anchor=south] at (1,0) {$\not\preceq$};
    \foreach \x in {2,3,...,10} {
      \path (\x,0) ++(-.5,0) node[font=\scriptsize,anchor=north] {$1$};
    };
    \path (1,0) ++(-.5,0) node[font=\scriptsize,anchor=north] {$x$};
    \node[font=\tiny,anchor=south,inner sep=.4mm] at (1.5,10) {$p_9$};
    \node[font=\tiny,anchor=south,inner sep=.4mm] at (2.5,9) {$p_8$};
    \node[font=\tiny,anchor=south,inner sep=.4mm] at (3.5,9) {$p_7$};
    \node[font=\tiny,anchor=south,inner sep=.4mm] at (4.5,7) {$p_6$};
    \node[font=\tiny,anchor=south,inner sep=.4mm] at (5.5,6) {$p_5$};
    \node[font=\tiny,anchor=south,inner sep=.4mm] at (6.5,5) {$p_4$};
    \node[font=\tiny,anchor=south,inner sep=.4mm] at (7.5,5) {$p_3$};
    %
  \end{scope}
\end{tikzpicture}
\]

The assumption that $\beta^{(1)}\cdots\beta^{(m)}x$ is of highest weight puts a restriction on $x$, as the following
lemma shows:

\begin{lemma}
\label{lem:an:rightmultgen}
Either $x=1$, or $x = p_k+1$ for some $k \in \set{1,\ldots,m}$ such that $p_k < n$.
\end{lemma}

\begin{proof}
  Suppose that $x \neq 1$ and $x \neq p_k+1$ for all $k$. Then for each $i$, either $\rho_{x-1}(\beta^{(k)}) =
  \emptyword$ (when $x-1 > p_k$) or $\rho_{x-1}(\beta^{(k)}) = {+}{-} = \emptyword$ (when $x-1 < p_k$), and so
  $\rho_{x-1}(\beta^{(1)}\cdots\beta^{(m)}x) = {-}$, contradicting the
assumption of highest weight.
\end{proof}

We consider the cases $x=p_k+1$ and $x=1$ separately.

First, suppose $x = p_k+1$, and assume that $k$ is maximal with this property. Then for $j > k$, we have
$x \not \preceq \beta^{(j)}$ and the symbol $x$ appears in $\beta^{(j)}$ and so by the \fullref{Commuting columns
  lemma}{lem:commutingcolumns},
$P(\tikz[tableau]\matrix{x \& \beta^{(j)} \\};) = \tikz[tableau]\matrix{\beta^{(j)} \& x\\};$, while
$P(\tikz[tableau]\matrix{x \& \beta^{(k)} \\};) = \tikz[tableau]\matrix{\beta^{(k)}x\\};$. That is, there are rewriting
rules $c_{\beta^{(j)}}c_x \to c_xc_{\beta^{(j)}}$ and $c_{\beta^{(k)}}c_x \to c_{\beta^{(k)}x}$. Further,
$\beta^{(k)}x \preceq \beta_{(k+1)}$. Therefore rewriting to normal form proceeds as follows:
\begin{align*}
&c_{\beta^{(1)}}\cdots c_{\beta^{(m)}}c_x \\
\imreduces{}&c_{\beta^{(1)}}\cdots c_{\beta^{(m-1)}}c_xc_{\beta^{(m)}} \\
&\vdots \\
\imreduces{}&c_{\beta^{(1)}}\cdots c_{\beta^{(k)}}c_xc_{\beta^{(k+1)}}\cdots c_{\beta^{(m)}} \\
\imreduces{}&c_{\beta^{(1)}}\cdots c_{\beta^{(k)}x}c_{\beta^{(k+1)}}\cdots c_{\beta^{(m)}}.
\end{align*}

Thus in the case $k=6$, in terms of the tabloid the column $x$ commutes past the columns $\beta^{(j)}$ with $j > 6$,
resulting in a tabloid of the following form:
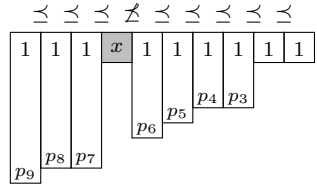
\[
\begin{tikzpicture}[x=4mm,y=-2mm]
  \begin{scope}[yshift=-30mm]
    \draw (0,0) rectangle (1,10);
    \draw (1,0) rectangle (2,9);
    \draw (2,0) rectangle (3,9);
    \draw (4,0) rectangle (5,7);
    \draw (5,0) rectangle (6,6);
    \draw (6,0) rectangle (7,5);
    \draw (7,0) rectangle (8,5);
    \draw (8,0) rectangle (9,2);
    \draw (9,0) rectangle (10,2);
    \draw[fill=lightgray] (3,0) rectangle (4,2);
    \foreach \x in {1,2,3,5,6,7,8,9} {
      \node[font=\scriptsize,anchor=south] at (\x,0) {$\preceq$};
    };
    \node[font=\scriptsize,anchor=south] at (4,0) {$\not\preceq$};
    \foreach \x in {1,2,3,5,6,7,8,9,10} {
      \path (\x,0) ++(-.5,0) node[font=\scriptsize,anchor=north] {$1$};
    };
    \path (4,0) ++(-.5,0) node[font=\scriptsize,anchor=north] {$x$};
    \node[font=\tiny,anchor=south,inner sep=.4mm] at (.5,10) {$p_9$};
    \node[font=\tiny,anchor=south,inner sep=.4mm] at (1.5,9) {$p_8$};
    \node[font=\tiny,anchor=south,inner sep=.4mm] at (2.5,9) {$p_7$};
    \node[font=\tiny,anchor=south,inner sep=.4mm] at (4.5,7) {$p_6$};
    \node[font=\tiny,anchor=south,inner sep=.4mm] at (5.5,6) {$p_5$};
    \node[font=\tiny,anchor=south,inner sep=.4mm] at (6.5,5) {$p_4$};
    \node[font=\tiny,anchor=south,inner sep=.4mm] at (7.5,5) {$p_3$};
    %
  \end{scope}
\end{tikzpicture}
\]
The final rewriting step appends $x$ to the bottom of the column $\beta^{(6)}$, giving a tableau of the following form:
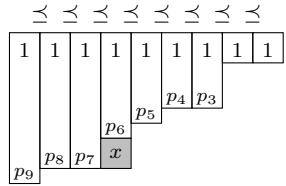
\[
\begin{tikzpicture}[x=4mm,y=-2mm]
  \begin{scope}[yshift=-60mm]
    \draw (0,0) rectangle (1,10);
    \draw (1,0) rectangle (2,9);
    \draw (2,0) rectangle (3,9);
    \draw (3,0) rectangle (4,7);
    \draw (4,0) rectangle (5,6);
    \draw (5,0) rectangle (6,5);
    \draw (6,0) rectangle (7,5);
    \draw (7,0) rectangle (8,2);
    \draw (8,0) rectangle (9,2);
    \draw[fill=lightgray] (3,7) rectangle (4,9);
    \foreach \x in {1,2,3,4,5,6,7,8} {
      \node[font=\scriptsize,anchor=south] at (\x,0) {$\preceq$};
    };
    \foreach \x in {1,2,3,4,5,6,7,8,9} {
      \path (\x,0) ++(-.5,0) node[font=\scriptsize,anchor=north] {$1$};
    };
    \path (4,0) ++(-.5,7) node[font=\scriptsize,anchor=north] {$x$};
    \node[font=\tiny,anchor=south,inner sep=.4mm] at (.5,10) {$p_9$};
    \node[font=\tiny,anchor=south,inner sep=.4mm] at (1.5,9) {$p_8$};
    \node[font=\tiny,anchor=south,inner sep=.4mm] at (2.5,9) {$p_7$};
    \node[font=\tiny,anchor=south,inner sep=.4mm] at (3.5,7) {$p_6$};
    \node[font=\tiny,anchor=south,inner sep=.4mm] at (4.5,6) {$p_5$};
    \node[font=\tiny,anchor=south,inner sep=.4mm] at (5.5,5) {$p_4$};
    \node[font=\tiny,anchor=south,inner sep=.4mm] at (6.5,5) {$p_3$};
    %
  \end{scope}
\end{tikzpicture}
\]

Now consider the other case, when $x=1$. Define $k$ to be maximal such that $\beta^{(k)} = 1$. Then for $j > k$, we have
$x \not \preceq \beta^{(j)}$ and the symbol $x$ appears in $\beta^{(j)}$ and so by the \fullref{Commuting columns
  lemma}{lem:commutingcolumns},
$P(\tikz[tableau]\matrix{x \& \beta^{(j)} \\};) = \tikz[tableau]\matrix{\beta^{(j)} \& x\\};$. That is, there are
rewriting rules $c_{\beta^{(j)}}c_x \to c_xc_{\beta^{(j)}}$. Further, $x \preceq \beta^{(k)}$. Therefore rewriting to
normal form proceeds as follows:
\begin{align*}
&c_{\beta^{(1)}}\cdots c_{\beta^{(m)}}c_x \\
\imreduces{}&c_{\beta^{(1)}}\cdots c_{\beta^{(m-1)}}c_xc_{\beta^{(m)}} \\
&\vdots \\
\imreduces{}&c_{\beta^{(1)}}\cdots c_{\beta^{(k)}}c_{\beta^{(k+1)}}c_xc_{\beta^{(k+2)}}\cdots c_{\beta^{(m)}} \\
\imreduces{}&c_{\beta^{(1)}}\cdots c_{\beta^{(k)}}c_xc_{\beta^{(k+1)}}\cdots c_{\beta^{(m)}}.
\end{align*}

Thus in the case $k=2$, in terms of the tabloid the column $x$ commutes past the columns $\beta^{(j)}$ with $j > 2$,
giving a tableau of the following form:
\[
  \begin{tikzpicture}[x=4mm,y=-2mm]
  \begin{scope}
    \draw (0,0) rectangle (1,10);
    \draw (1,0) rectangle (2,9);
    \draw (2,0) rectangle (3,9);
    \draw (3,0) rectangle (4,7);
    \draw (4,0) rectangle (5,6);
    \draw (5,0) rectangle (6,5);
    \draw (6,0) rectangle (7,5);
    \draw[fill=lightgray] (7,0) rectangle (8,2);
    \draw (8,0) rectangle (9,2);
    \draw (9,0) rectangle (10,2);
    \foreach \x in {1,2,3,4,5,6,7,8,9} {
      \node[font=\scriptsize,anchor=south] at (\x,0) {$\preceq$};
    };
    \foreach \x in {1,2,3,4,5,6,7,9,10} {
      \path (\x,0) ++(-.5,0) node[font=\scriptsize,anchor=north] {$1$};
    };
    \path (8,0) ++(-.5,0) node[font=\scriptsize,anchor=north] {$x$};
    \node[font=\tiny,anchor=south,inner sep=.4mm] at (.5,10) {$p_9$};
    \node[font=\tiny,anchor=south,inner sep=.4mm] at (1.5,9) {$p_8$};
    \node[font=\tiny,anchor=south,inner sep=.4mm] at (2.5,9) {$p_7$};
    \node[font=\tiny,anchor=south,inner sep=.4mm] at (3.5,7) {$p_6$};
    \node[font=\tiny,anchor=south,inner sep=.4mm] at (4.5,6) {$p_5$};
    \node[font=\tiny,anchor=south,inner sep=.4mm] at (5.5,5) {$p_4$};
    \node[font=\tiny,anchor=south,inner sep=.4mm] at (6.5,5) {$p_3$};
    %
  \end{scope}
\end{tikzpicture}
\]

Note that in both cases the length of the normal form word differs from $m$ by at most $1$.

As in the discussion before \fullref{Lemma}{lem:anbncndn:leftmultbytrans}, the
way rewriting proceeds at highest weight
is mirrored in how it proceeds in general. That is, rewriting can be carried out by a single right-to-left pass through
the word. The aim is now to show that a transducer can recognize the relation consisting of pairs $(u,v)$, where
$u,v \in \Sigma^*$ are irreducible and $uc_x \reduces v$, by essentially computing this rewriting. To see this, consider
a transducer that reads its input tapes right-to-left: such a transducer can carry out the rewriting in a way
symmetrical to that described in the discussion before \fullref{Lemma}{lem:anbncndn:leftmultbytrans}. Since the class of
rational relations is closed under reversal \cite[p.~65--66]{berstel_transductions}, it follows we have proven the
following analogue of \fullref{Lemma}{lem:anbncndn:leftmultbytrans} for right-multiplication in type $A_n$:

\begin{lemma}
  \label{lem:an:rightmultbytrans}
  Let $\Sigma$ and $T$ be the alphabet and set of rewriting rules constructed for type $A_n$ in
  \fullref{Subsection}{subsec:rewritingconstruction}. Let $x \in \aA_n$. Let $L \subseteq \Sigma^*$ be the languages of
  irreducible words. Then the relation
  \[
    L_{c_x} = \gset[\big]{(u,v) \in L\times L}{uc_x =_{\Pl(A_n)} v}
  \]
  is recognized by an transducer. Furthermore, if $(u,v)$ is a pair in this relation, then the lengths of $u$ and $v$
  differ by at most $1$.
\end{lemma}

\subsubsection{\texorpdfstring{$C_n$}{Cn}}

Let $\beta^{(1)},\ldots,\beta^{(m)}$ be admissible $C_n$ columns and let $x \in \aC_n$ be such that
$\beta^{(i+1)} \preceq \beta^{(i)}$ for $i=1,\ldots,m-1$ (that is,
$\tikz[tableau]\matrix{\beta^{(m)} \& |[dottedentry]| \& \beta^{(1)} \\};$ is a $C_n$ tableau), and such that
$\beta^{(1)}\ldots\beta^{(m)}x$ is a highest-weight word. As we did for type $A_n$, we are going to examine how the
corresponding word over $\Sigma$ (that is, $c_{\beta^{(1)}}\cdots c_{\beta^{(m)}}c_x$) rewrites to an irreducible
word. Again, the aim is to prove that this rewriting involves a single right-to-left pass through the word.

Since $\beta^{(1)}\ldots\beta^{(m)}$ is a highest-weight word by \fullref{Lemma}{lem:highestweightfactors}, each column
$\beta^{(i)}$ is of the form $1\cdots p_i$ for some $p_i \in \aC[1,n]$, and $p_{i+1} \geq p_i$ for $i=1,\ldots,m-1$ by
\fullref{Lemma}{lem:highestweightableauchar}.

The reasoning will proceed in a similar way to the $A_n$ case, except that there is the additional possibility that $x$
may be $\bar{p_k}$, as shown in the following lemma:

\begin{lemma}
\label{lem:cn:rightmultgen}
Either $x=1$, or $x = p_k+1$ for some $k \in \set{1,\ldots,m}$ such that $p_k < n$, or $x = \bar{p_k}$ for some
$k \in \set{1,\ldots,m}$.
\end{lemma}

\begin{proof}
  Suppose that $x \neq 1$ and $x \neq p_k+1$ and $x \neq \bar{p_k}$ for all $k$. If $x \in \aC[1,n]$, then the same
  contradiction arises as in the proof of \fullref{Lemma}{lem:an:rightmultgen}.  If $x \in \aC[\bar{n},\bar{1}]$, then
  for each $k$, either $\rho_{\bar{x}}(\beta^{(k)}) = \emptyword$ (when $\bar{x} > p_k$) or $\rho_{\bar{x}}(\beta^{(k)})
  = {+}{-} = \emptyword$ (when $\bar{x} < p_k$), and so $\rho_{\bar{x}}(\beta^{(1)}\cdots\beta^{(m)}x) = {-}$, contradicting
  the assumption of highest weight.
\end{proof}

If $x = 1$ or $x = p_k+1$, then the rewriting proceeds in the same way as in the $A_n$ case. So suppose $x = \bar{p_k}$.
If $p_k > 1$, we will assume that $k$ is minimal with this property; if $p_k = 1$, we will assume that $k$ is maximal
with this property.

Now, $P(\tikz[tableau]\matrix{\bar{p_k} \& \beta^{(m)}\\};) = \tikz[tableau]\matrix{1\cdots (p_k-1)(p_k+1)\cdots p_{m} \\};$, since
$\beta^{(m)}\bar{p_k} =_{\drel{R}_5^{C_n}} 1\cdots (p_k-1)(p_k+1)\cdots p_{m}$. Pictorially (using $k = 6$ as an example), we have:
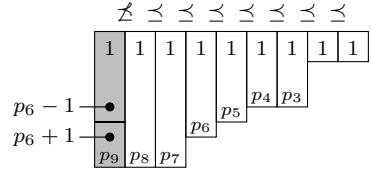
\[
\begin{tikzpicture}[x=4mm,y=-2mm]
  \begin{scope}
    \draw[fill=lightgray] (0,0) rectangle (1,6) rectangle (0,9);
    \draw (1,0) rectangle (2,9);
    \draw (2,0) rectangle (3,9);
    \draw (3,0) rectangle (4,7);
    \draw (4,0) rectangle (5,6);
    \draw (5,0) rectangle (6,5);
    \draw (6,0) rectangle (7,5);
    \draw (7,0) rectangle (8,2);
    \draw (8,0) rectangle (9,2);
    \draw[fill=black] (.5,5) circle[radius=1.5pt] -- +(-1,0) node[anchor=east,font=\scriptsize,inner sep=.5mm] {$p_6 - 1$};
    \draw[fill=black] (.5,7) circle[radius=1.5pt] -- +(-1,0) node[anchor=east,font=\scriptsize,inner sep=.5mm] {$p_6 + 1$};
    \foreach \x in {2,3,4,5,6,7,8} {
      \node[font=\scriptsize,anchor=south] at (\x,0) {$\preceq$};
    };
    \node[font=\scriptsize,anchor=south] at (1,0) {$\not\preceq$};
    \foreach \x in {1,2,3,4,5,6,7,8,9} {
      \path (\x,0) ++(-.5,0) node[font=\scriptsize,anchor=north] {$1$};
    };
    \node[font=\tiny,anchor=south,inner sep=.4mm] at (.5,9) {$p_9$};
    \node[font=\tiny,anchor=south,inner sep=.4mm] at (1.5,9) {$p_8$};
    \node[font=\tiny,anchor=south,inner sep=.4mm] at (2.5,9) {$p_7$};
    \node[font=\tiny,anchor=south,inner sep=.4mm] at (3.5,7) {$p_6$};
    \node[font=\tiny,anchor=south,inner sep=.4mm] at (4.5,6) {$p_5$};
    \node[font=\tiny,anchor=south,inner sep=.4mm] at (5.5,5) {$p_4$};
    \node[font=\tiny,anchor=south,inner sep=.4mm] at (6.5,5) {$p_3$};
    %
  \end{scope}
\end{tikzpicture}
\]

Now, for $m > j > k$, we have
\begin{align*}
&\beta^{(j)}1\cdots (p_k-1)(p_k+1)\cdots p_{j+1} \\
={}& \underbrace{1\cdots p_j}\underbrace{1\cdots (p_k-1)(p_k+1)\cdots p_j}(p_j+1)\cdots p_{j+1} \\
& \qquad\text{[by the \fullref{Commuting columns lemma}{lem:commutingcolumns}]} \\
=_{\Pl(C_n)}{}& \overbrace{1\cdots (p_k-1)(p_k+1)\cdots p_j}\overbrace{1\cdots p_j}(p_j+1)\cdots p_{j+1} \\
={}& 1\cdots (p_k-1)(p_k+1)\cdots p_j \beta^{(j+1)}.
\end{align*}
Write $\beta_*^{(j)}$ for $\beta^{(j)}$ with the symbol $p_k$ deleted. Then we have
$P(\tikz[tableau]\matrix{\beta_*^{(j+1)} \& \beta^{(j)}\\};) = \tikz[tableau]\matrix{\beta^{(j+1)} \& \beta_*^{(j)}\\};$
for all $j = k+1,\ldots,m-1$. (Note that when $p_k = 1$, we know from the maximality of $k$ that $\beta^{(j)} \neq 1$.)
Thus we have rewriting rules $c_{\beta^{(j)}}c_{\beta_*^{(j+1)}} \to c_{\beta_*^{(j)}}c_{\beta^{(j+1)}}$.

When $p_k \neq 1$, we have $\beta_*^{(k)} = 1\cdots (p_k-1).$ Thus $\beta_*^{(k)} \preceq \beta^{(k-1)}$ since by the
minimality of $k$ we have $p_{k-1} < p_k$. Thus in this case rewriting to normal form proceeds as follows:
\begin{align*}
&c_{\beta^{(1)}}\cdots c_{\beta^{(m)}}c_{\bar{p_k}} \\
\imreduces{}&c_{\beta^{(1)}}\cdots c_{\beta^{(m-1)}}c_{\beta_*^{(m)}} \\
\imreduces{}&c_{\beta^{(1)}}\cdots c_{\beta_*^{(m-1)}}c_{\beta^{(m)}} \\
&\vdots \\
\imreduces{}&c_{\beta^{(1)}}\cdots c_{\beta^{(k-1)}}c_{\beta_*^{(k)}}c_{\beta^{(k+1)}}\cdots c_{\beta^{(m)}}.
\end{align*}

In the case $k=6$ with $p_k \neq 1$, in terms of the tabloid the `gap' in the columns moves from left to right through
the tabloid:
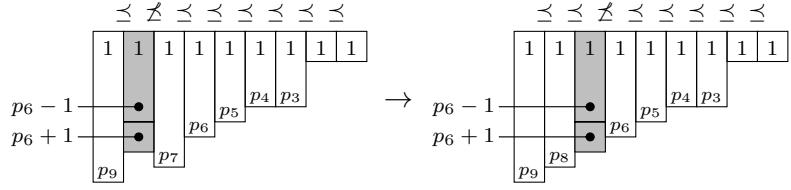
\[
\begin{tikzpicture}[x=4mm,y=-2mm,baseline=-10mm]
  \begin{scope}
    \draw (0,0) rectangle (1,10);
    \draw[fill=lightgray] (1,0) rectangle (2,6) rectangle (1,8);
    \draw (2,0) rectangle (3,9);
    \draw (3,0) rectangle (4,7);
    \draw (4,0) rectangle (5,6);
    \draw (5,0) rectangle (6,5);
    \draw (6,0) rectangle (7,5);
    \draw (7,0) rectangle (8,2);
    \draw (8,0) rectangle (9,2);
    \draw[fill=black] (1.5,5) circle[radius=1.5pt] -- +(-2,0) node[anchor=east,font=\scriptsize,inner sep=.5mm] {$p_6 - 1$};
    \draw[fill=black] (1.5,7) circle[radius=1.5pt] -- +(-2,0) node[anchor=east,font=\scriptsize,inner sep=.5mm] {$p_6 + 1$};
    \foreach \x in {1,3,4,5,6,7,8} {
      \node[font=\scriptsize,anchor=south] at (\x,0) {$\preceq$};
    };
    \node[font=\scriptsize,anchor=south] at (2,0) {$\not\preceq$};
    \foreach \x in {1,2,3,4,5,6,7,8,9} {
      \path (\x,0) ++(-.5,0) node[font=\scriptsize,anchor=north] {$1$};
    };
    \node[font=\tiny,anchor=south,inner sep=.4mm] at (.5,10) {$p_9$};
    \node[font=\tiny,anchor=south,inner sep=.4mm] at (2.5,9) {$p_7$};
    \node[font=\tiny,anchor=south,inner sep=.4mm] at (3.5,7) {$p_6$};
    \node[font=\tiny,anchor=south,inner sep=.4mm] at (4.5,6) {$p_5$};
    \node[font=\tiny,anchor=south,inner sep=.4mm] at (5.5,5) {$p_4$};
    \node[font=\tiny,anchor=south,inner sep=.4mm] at (6.5,5) {$p_3$};
    %
  \end{scope}
\end{tikzpicture}
\; \to \;
\begin{tikzpicture}[x=4mm,y=-2mm,baseline=-10mm]
  \begin{scope}
    \draw (0,0) rectangle (1,10);
    \draw (1,0) rectangle (2,9);
    \draw[fill=lightgray] (2,0) rectangle (3,6) rectangle (2,8);
    \draw (3,0) rectangle (4,7);
    \draw (4,0) rectangle (5,6);
    \draw (5,0) rectangle (6,5);
    \draw (6,0) rectangle (7,5);
    \draw (7,0) rectangle (8,2);
    \draw (8,0) rectangle (9,2);
    \draw[fill=black] (2.5,5) circle[radius=1.5pt] -- +(-3,0) node[anchor=east,font=\scriptsize,inner sep=.5mm] {$p_6 - 1$};
    \draw[fill=black] (2.5,7) circle[radius=1.5pt] -- +(-3,0) node[anchor=east,font=\scriptsize,inner sep=.5mm] {$p_6 + 1$};
    \foreach \x in {1,2,4,5,6,7,8} {
      \node[font=\scriptsize,anchor=south] at (\x,0) {$\preceq$};
    };
    \node[font=\scriptsize,anchor=south] at (3,0) {$\not\preceq$};
    \foreach \x in {1,2,3,4,5,6,7,8,9} {
      \path (\x,0) ++(-.5,0) node[font=\scriptsize,anchor=north] {$1$};
    };
    \node[font=\tiny,anchor=south,inner sep=.4mm] at (.5,10) {$p_9$};
    \node[font=\tiny,anchor=south,inner sep=.4mm] at (1.5,9) {$p_8$};
    \node[font=\tiny,anchor=south,inner sep=.4mm] at (3.5,7) {$p_6$};
    \node[font=\tiny,anchor=south,inner sep=.4mm] at (4.5,6) {$p_5$};
    \node[font=\tiny,anchor=south,inner sep=.4mm] at (5.5,5) {$p_4$};
    \node[font=\tiny,anchor=south,inner sep=.4mm] at (6.5,5) {$p_3$};
    %
  \end{scope}
\end{tikzpicture}
\]
The rewriting continues until the `gap' reaches the column $\beta^{(k)}$, at which point a tableau is obtained:
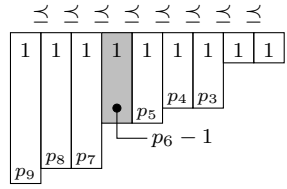
\[
\begin{tikzpicture}[x=4mm,y=-2mm]
  \begin{scope}
    \draw (0,0) rectangle (1,10);
    \draw (1,0) rectangle (2,9);
    \draw (2,0) rectangle (3,9);
    \draw[fill=lightgray] (3,0) rectangle (4,6);
    \draw (4,0) rectangle (5,6);
    \draw (5,0) rectangle (6,5);
    \draw (6,0) rectangle (7,5);
    \draw (7,0) rectangle (8,2);
    \draw (8,0) rectangle (9,2);
    \draw[fill=black] (3.5,5) circle[radius=1.5pt];
    \draw (3.5,5) |- +(1,2) node[anchor=west,font=\scriptsize,inner sep=.5mm] {$p_6 - 1$};
    \foreach \x in {1,2,...,8} {
      \node[font=\scriptsize,anchor=south] at (\x,0) {$\preceq$};
    };
    \foreach \x in {1,2,3,4,5,6,7,8,9} {
      \path (\x,0) ++(-.5,0) node[font=\scriptsize,anchor=north] {$1$};
    };
    \node[font=\tiny,anchor=south,inner sep=.4mm] at (.5,10) {$p_9$};
    \node[font=\tiny,anchor=south,inner sep=.4mm] at (1.5,9) {$p_8$};
    \node[font=\tiny,anchor=south,inner sep=.4mm] at (2.5,9) {$p_7$};
    \node[font=\tiny,anchor=south,inner sep=.4mm] at (4.5,6) {$p_5$};
    \node[font=\tiny,anchor=south,inner sep=.4mm] at (5.5,5) {$p_4$};
    \node[font=\tiny,anchor=south,inner sep=.4mm] at (6.5,5) {$p_3$};
    %
  \end{scope}
\end{tikzpicture}
\]

When $p_k = 1$, we have $\beta_*^{(k+1)} = 2\cdots p_{k+1}$ (and we know $p_{k+1} > 1$ by the maximality of $k$) and so
$P(\tikz[tableau]\matrix{\beta_*^{(k+1)} \& \beta^{(k)}\\};) = \tikz[tableau]\matrix{\beta^{(k+1)}\\};$ and so there is
a rewriting rule $c_{\beta^{(k)}}c_{\beta_*^{(k+1}} \to c_{\beta^{(k+1)}}$. Thus in this case rewriting to normal form proceeds as follows:
\begin{align*}
&c_{\beta^{(1)}}\cdots c_{\beta^{(m)}}c_{\bar{p_k}} \\
\imreduces{}&c_{\beta^{(1)}}\cdots c_{\beta^{(m-1)}}c_{\beta_*^{(m)}} \\
\imreduces{}&c_{\beta^{(1)}}\cdots c_{\beta_*^{(m-1)}}c_{\beta^{(m)}} \\
&\vdots \\
\imreduces{}&c_{\beta^{(1)}}\cdots c_{\beta^{(k-1)}}c_{\beta^{(k)}}c_{\beta_*^{(k+1)}}\cdots c_{\beta^{(m)}} \\
\imreduces{}&c_{\beta^{(1)}}\cdots c_{\beta^{(k-1)}}c_{\beta^{(k+1)}}\cdots c_{\beta^{(m)}}.
\end{align*}

In the case $k=2$ with $p_k = 1$, in terms of the tabloid the `gap' in the columns moves from left to right through
the tabloid, just as in the other case, but then there is a final rewriting step:
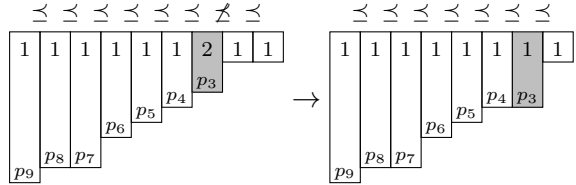
\[
\begin{tikzpicture}[x=4mm,y=-2mm,baseline=-10mm]
  \begin{scope}
    \draw (0,0) rectangle (1,10);
    \draw (1,0) rectangle (2,9);
    \draw (2,0) rectangle (3,9);
    \draw (3,0) rectangle (4,7);
    \draw (4,0) rectangle (5,6);
    \draw (5,0) rectangle (6,5);
    \draw[fill=lightgray] (6,0) rectangle (7,4);
    \draw (7,0) rectangle (8,2);
    \draw (8,0) rectangle (9,2);
    \foreach \x in {1,2,...,6,8} {
      \node[font=\scriptsize,anchor=south] at (\x,0) {$\preceq$};
    };
    \node[font=\scriptsize,anchor=south] at (7,0) {$\not\preceq$};
    \foreach \x in {1,2,3,4,5,6,8,9} {
      \path (\x,0) ++(-.5,0) node[font=\scriptsize,anchor=north] {$1$};
    };
    \path (7,0) ++(-.5,0) node[font=\scriptsize,anchor=north] {$2$};
    \node[font=\tiny,anchor=south,inner sep=.4mm] at (.5,10) {$p_9$};
    \node[font=\tiny,anchor=south,inner sep=.4mm] at (1.5,9) {$p_8$};
    \node[font=\tiny,anchor=south,inner sep=.4mm] at (2.5,9) {$p_7$};
    \node[font=\tiny,anchor=south,inner sep=.4mm] at (3.5,7) {$p_6$};
    \node[font=\tiny,anchor=south,inner sep=.4mm] at (4.5,6) {$p_5$};
    \node[font=\tiny,anchor=south,inner sep=.4mm] at (5.5,5) {$p_4$};
    \node[font=\tiny,anchor=south,inner sep=.4mm] at (6.5,4) {$p_3$};
    %
  \end{scope}
\end{tikzpicture}
\to
\begin{tikzpicture}[x=4mm,y=-2mm,baseline=-10mm]
  \begin{scope}
    \draw (0,0) rectangle (1,10);
    \draw (1,0) rectangle (2,9);
    \draw (2,0) rectangle (3,9);
    \draw (3,0) rectangle (4,7);
    \draw (4,0) rectangle (5,6);
    \draw (5,0) rectangle (6,5);
    \draw[fill=lightgray] (6,0) rectangle (7,5);
    \draw (7,0) rectangle (8,2);
    \foreach \x in {1,2,...,6,7} {
      \node[font=\scriptsize,anchor=south] at (\x,0) {$\preceq$};
    };
    \foreach \x in {1,2,3,4,5,6,7,8} {
      \path (\x,0) ++(-.5,0) node[font=\scriptsize,anchor=north] {$1$};
    };
    \node[font=\tiny,anchor=south,inner sep=.4mm] at (.5,10) {$p_9$};
    \node[font=\tiny,anchor=south,inner sep=.4mm] at (1.5,9) {$p_8$};
    \node[font=\tiny,anchor=south,inner sep=.4mm] at (2.5,9) {$p_7$};
    \node[font=\tiny,anchor=south,inner sep=.4mm] at (3.5,7) {$p_6$};
    \node[font=\tiny,anchor=south,inner sep=.4mm] at (4.5,6) {$p_5$};
    \node[font=\tiny,anchor=south,inner sep=.4mm] at (5.5,5) {$p_4$};
    \node[font=\tiny,anchor=south,inner sep=.4mm] at (6.5,5) {$p_3$};
    %
  \end{scope}
\end{tikzpicture}
\]

Note that in each case the length of the normal form word differs from $m$ by at most $1$.

As in the discussion before \fullref{Lemma}{lem:an:rightmultbytrans}, the way
rewriting proceeds at highest weight is
mirrored in how it proceeds in general and so, using the same argument, we have proven the following analogue of
\fullref{Lemma}{lem:an:rightmultbytrans} for type $C_n$:

\begin{lemma}
  \label{lem:cn:rightmultbytrans}
  Let $\Sigma$ and $T$ be the alphabet and set of rewriting rules constructed for type $C_n$ in
  \fullref{Subsection}{subsec:rewritingconstruction}. Let $x \in \aC_n$. Let $L \subseteq \Sigma^*$ be the languages of
  irreducible words. Then the relation
  \[
    L_{c_x} = \gset[\big]{(u,v) \in L\times L}{uc_x =_{\Pl(C_n)} v}
  \]
  is recognized by an transducer. Furthermore, if $(u,v)$ is a pair in this relation, then the lengths of $u$ and $v$
  differ by at most $1$.
\end{lemma}

\subsubsection{\texorpdfstring{$B_n$}{Bn}}

Let $\beta^{(1)},\ldots,\beta^{(m)}$ be admissible $B_n$ columns and let $x \in \aB_n$ be such that
$\beta^{(i+1)} \preceq \beta^{(i)}$ for $i=1,\ldots,m-1$ (that is,
$\tikz[tableau]\matrix{\beta^{(m)} \& |[dottedentry]| \& \beta^{(1)} \\};$ is a $B_n$ tableau), and such that
$\beta^{(1)}\ldots\beta^{(m)}x$ is a highest-weight word. As we did for types $A_n$ and $C_n$, we are going to examine
how the corresponding word over $\Sigma$ (that is, $c_{\beta^{(1)}}\cdots c_{\beta^{(m)}}c_x$) rewrites to an
irreducible word. In fact, the analysis reduces almost entirely to the $C_n$ case: there is only one easy extra
case. Again, the aim is to prove that this rewriting involves a single right-to-left pass through the word.

Since $\beta^{(1)}\ldots\beta^{(m)}$ is a highest-weight tableau, each column $\beta^{(i)}$ is of the form $1\cdots p_i$
for some $p_i \in \aB[1,n]$, and $p_{i+1} \geq p_i$ for $i=1,\ldots,m-1$ by
\fullref{Lemma}{lem:highestweightableauchar}.

\begin{lemma}
\label{lem:bn:rightmultgen}
One of the following holds:
\begin{enumerate}
\item $x=1$;
\item $x = p_k+1$ for some $k \in \set{1,\ldots,m}$ such that $p_k < n$;
\item $x=0$ (only if $p_m = n$);
\item $x = \bar{p_k}$ for some $k \in \set{1,\ldots,m}$.
\end{enumerate}
\end{lemma}

\begin{proof}
  Suppose that $x \neq 1$, $x \neq p_k+1$, $x \neq 0$, and $x \neq \bar{p_k}$ for all $k$. If $x \in \aB_n[1,n]$, then the
  same contradiction arises as in the proof of \fullref{Lemma}{lem:an:rightmultgen}.  If $x \in \aB_n[\bar{n},\bar{1}]$,
  then the same contradiction arises as in the proof of \fullref{Lemma}{lem:cn:rightmultgen}.

  Finally, suppose $x = 0$. If $p_m \neq n$, then $\rho_n(\beta^{(k)}) = \emptyword$ for each $k$ and so
  $\rho_{n}(\beta^{(1)}\cdots\beta^{(m)}0) = {-}{+}$, contradicting the
assumption of highest weight.
\end{proof}

If $x = 1$ or $x = p_k+1$, then the rewriting proceeds in the same way as in the $A_n$ case, and if $x = \bar{p_k}$,
then the rewriting proceeds in the same way as the $C_n$ case. So suppose $x=0$. Then $p_m = n$ and so
$\beta^{(m)}0 = 1\cdots n0 =_{\drel{R}_5^{B_n}} 1\cdots n = \beta^{(m)}$; thus
$P(\tikz[tableau]\matrix{0 \& \beta^{(m)}\\};) = \tikz[tableau]\matrix{\beta^{(m)}\\};$. So there is a rewriting rule
$c_{\beta^{(m)}}c_0 = c_{\beta^{(m)}}$ and so rewriting to normal form is as follows:
\[
c_{\beta^{(1)}}\cdots c_{\beta^{(m)}}c_{0} \imreduces c_{\beta^{(1)}}\cdots c_{\beta^{(m)}}.
\]

Note that in each case the length of the normal form word differs from $m$ by at most $1$.

As in the discussion before \fullref{Lemma}{lem:an:rightmultbytrans}, the way
rewriting proceeds at highest weight is
mirrored in how it proceeds in general and so, using the same argument, we have proven the following analogue of
\fullref{Lemma}{lem:an:rightmultbytrans} for type $B_n$:

\begin{lemma}
  \label{lem:bn:rightmultbytrans}
  Let $\Sigma$ and $T$ be the alphabet and set of rewriting rules constructed for type $B_n$ in
  \fullref{Subsection}{subsec:rewritingconstruction}. Let $x \in \aB_n$. Let $L \subseteq \Sigma^*$ be the languages of
  irreducible words. Then the relation
  \[
    L_{c_x} = \gset[\big]{(u,v) \in L\times L}{uc_x =_{\Pl(B_n)} v}
  \]
  is recognized by an transducer. Furthermore, if $(u,v)$ is a pair in this relation, then the lengths of $u$ and $v$
  differ by at most $1$.
\end{lemma}

\subsubsection{\texorpdfstring{$D_n$}{Dn}}

Let $\beta^{(1)},\ldots,\beta^{(m)}$ be admissible $D_n$ columns and let $x \in \aD_n$ be such that
$\beta^{(i+1)} \preceq \beta^{(i)}$ for $i=1,\ldots,m-1$ (that is,
$\tikz[tableau]\matrix{\beta^{(m)} \& |[dottedentry]| \& \beta^{(1)} \\};$ is a $D_n$ tableau), and such that
$\beta^{(1)}\ldots\beta^{(m)}x$ is a highest-weight word. As for the other types, we are going to examine how the
corresponding word over $\Sigma$ (that is, $c_{\beta^{(1)}}\cdots c_{\beta^{(m)}}c_x$) rewrites to an irreducible
word. As before, the aim is to prove that this rewriting involves a single right-to-left pass through the word.

Since $\beta^{(1)}\ldots\beta^{(m)}$ is a highest-weight word by \fullref{Lemma}{lem:highestweightfactors}, each column
$\beta^{(i)}$ is of the form $1\cdots p_i$ for some $p_i \in \aD[1,n] \cup \aD[1,\bar{n}]$, and $p_{i+1} \geq p_i$ for
$i=1,\ldots,m-1$ by \fullref{Lemma}{lem:highestweightableauchar}.

\begin{lemma}
\label{lem:dn:rightmultgen}
One of the following holds:
\begin{enumerate}
\item $x=1$;
\item $x = p_k+1$ for some $k \in \set{1,\ldots,m}$ such that $p_k < n-1$;
\item $x=n$ (only if $\beta^{(k)} = 1\cdots (n-1)$ for some $k$ or $\beta^{(m)} = 1\cdots
(n-1)\bar{n}$);
\item $x=\bar{n}$ (only if $\beta^{(k)} = 1\cdots (n-1)$ for some $k$ or $\beta^{(m)} = 1\cdots n$);
\item $x = \bar{p_k}$ for some $k \in \set{1,\ldots,m}$ such that $p_k \leq n-1$.
\end{enumerate}
\end{lemma}

\begin{proof}
  Suppose that $x \neq 1$, $x \neq p_k+1$, $x \neq n$, $x \neq \bar{n}$, and $x \neq \bar{p_k}$ for all $k$. If $x \in
  \aD_n[1,n-1]$ then the same contradiction arises as in the proof of \fullref{Lemma}{lem:an:rightmultgen}.  If $x \in
  \aD_n[\bar{n-1},\bar{1}]$, then the same contradiction arises as in the proof of \fullref{Lemma}{lem:cn:rightmultgen}.

  Now, suppose $x = n$. If $\beta^{(k)} \neq 1\cdots (n-1)$ for all $k$ and $\beta^{(m)} \neq 1\cdots (n-1)\bar{n}$, then
  $\rho_{n}(\beta^{(j)}) = \emptyword$ (when $\beta^{(j)} = 1\cdots p_j$ for $p_j \leq n-2$) and $\rho_{n-1}(\beta^{(j)})
  = {+}{-} = \emptyword$ (when $\beta^{(j)} = 1\cdots n$) and so $\rho_{n-1}(\beta^{(1)}\cdots\beta^{(m)}n) = {-}$,
  contradicting the assumption of highest weight.

  Similar reasoning shows that $x=\bar{n}$ only if $\beta^{(k)} = 1\cdots (n-1)$ for some $k$ or
  $\beta^{(m)} = 1\cdots n$, using $\rho_n$ to get the contradictions.
\end{proof}

If cases (1) or (2) of \fullref{Lemma}{lem:dn:rightmultgen} hold, or case (3) holds with $\beta^{(k)} = 1\cdots (n-1)$
for some $k$, then the rewriting proceeds in the same way as in the $A_n$ case. If case (5) holds, or case (4) holds
with $\beta^{(m)} = 1\cdots n$, then the rewriting proceeds in the same way as the $C_n$ case.

We thus have two remaining cases: case (3) with $x=n$ and $\beta^{(m)} = 1\cdots (n-1)\bar{n}$, or case (4) with
$x=\bar{n}$ and $\beta^{(k)} = 1\cdots (n-1)$ for some $k$.

Suppose $x = \bar{n}$ and $\beta^{(k)} = 1\cdots (n-1)$ for some $k$. In the case where $\beta^{(m)} = 1\cdots n$,
rewriting proceeds as in the $C_n$ case. So, by the definition of $\preceq$, either $\beta^{(m)} = 1\cdots (n-1)$ or
$\beta^{(m)} = 1\cdots (n-1)\bar{n}$. Consider these cases separately:
\begin{enumerate}
\item $\beta^{(m)} = 1\cdots (n-1)$. So $P(\beta^{(m)}\bar{n})$ is the single column $\beta^{(m)}\bar{n}$ and so there
  is a rewriting rule $c_{\beta^{(m)}}c_{\bar{n}} \to c_{\beta^{(m)}\bar{n}}$ and so rewriting to normal form proceeds
as follows:
  \[
  c_{\beta^{(1)}}\cdots c_{\beta^{(m)}}c_{\bar{n}} \imreduces c_{\beta^{(1)}}\cdots c_{\beta^{(m)}\bar{n}}.
  \]
\item $\beta^{(m)} = 1\cdots (n-1)\bar{n}$. Then rewriting proceeds in the same way as in the $A_n$ case, but with
  $\bar{n}$ in place of $n$.
\end{enumerate}

Finally, suppose $x = n$ and $\beta^{(m)} = 1\cdots (n-1)\bar{n}$. It is easy to see that rewriting is symmetric to
the case $C_n$ where $x = \bar{n}$ and $\beta^{(m)} = 1\cdots n$.

Note that in each case the length of the normal form word differs from $m$ by at most $1$.

As in the discussion before \fullref{Lemma}{lem:an:rightmultbytrans}, the way
rewriting proceeds at highest weight is
mirrored in how it proceeds in general and so, using the same argument, we have proven the following analogue of
\fullref{Lemma}{lem:an:rightmultbytrans} for type $D_n$:

\begin{lemma}
  \label{lem:dn:rightmultbytrans}
  Let $\Sigma$ and $T$ be the alphabet and set of rewriting rules constructed for type $D_n$ in
  \fullref{Subsection}{subsec:rewritingconstruction}. Let $x \in \aD_n$. Let $L \subseteq \Sigma^*$ be the languages of
  irreducible words. Then the relation
  \[
    L_{c_x} = \gset[\big]{(u,v) \in L\times L}{uc_x =_{\Pl(D_2)} v}
  \]
  is recognized by an transducer. Furthermore, if $(u,v)$ is a pair in this relation, then the lengths of $u$ and $v$
  differ by at most $1$.
\end{lemma}

\subsubsection{\texorpdfstring{$G_2$}{G2}}
\label{subsubsec:g2:rightmult}

Let $\beta^{(1)},\ldots,\beta^{(m)}$ be admissible $G_2$ columns and let $x \in \aG_2$ be such that
$\beta^{(i+1)} \preceq \beta^{(i)}$ for $i=1,\ldots,m-1$ (that is,
$\tikz[tableau]\matrix{\beta^{(m)} \& |[dottedentry]| \& \beta^{(1)} \\};$ is a $G_2$ tableau), and such that
$\beta^{(1)}\ldots\beta^{(m)}x$ is a highest-weight word. Since $\beta^{(1)}\ldots\beta^{(m)}$ is a highest-weight word,
each column $\beta^{(i)}$ is either $1$ or $12$ by \fullref{Lemma}{lem:highestweightableauchar}. Notice that, by the
definition of $\preceq$ for type $G_2$, some $\beta^{(j)}$ is $12$ if and only if the leftmost column $\beta^{(m)}$ is
$12$, and some $\beta^{(j)}$ is $1$ if and only if the rightmost column $\beta^{(1)}$ is $1$. As for the other types, we
are going to examine how the corresponding word over $\Sigma$ (that is, $c_{\beta^{(1)}}\cdots c_{\beta^{(m)}}c_x$)
rewrites to an irreducible word. As before, the aim is to prove that this rewriting involves a single right-to-left pass
through the word.

We first prove the following lemma, which tells us about the possible cases for $x$ and the restrictions this puts on
the columns $\beta^{(j)}$. We will then consider separately the rewriting that occurs according to whether some
$\beta^{(j)}$ is the column $12$.

\begin{lemma}
  \label{lem:g2:rightmultgen}
  The generator $x$ can be
  \begin{enumerate}
  \item $1$;
  \item $2$, only if there is at least one column $1$ among the $\beta^{(j)}$;
  \item $3$, only if there is at least one column $12$ among the $\beta^{(j)}$;
  \item $0$, only if there is at least one column $1$ among the $\beta^{(j)}$;
  \item $\bar{3}$, only if there are at least two columns $1$ among the $\beta^{(j)}$;
  \item $\bar{2}$, only if there is at least one column $12$ among the $\beta^{(j)}$;
  \item $\bar{1}$, only if there is at least one column $1$ among the $\beta^{(j)}$.
  \end{enumerate}
\end{lemma}

\begin{proof}
  Note first that $\rho_1(1) = {+}$, $\rho_1(12) = {+}{-} = \emptyword$, $\rho_2(1) = \emptyword$, $\rho_2(12) = {+}$,
  so $\rho_1(\beta^{(1)}\cdots\beta^{(m)})$ consists of a string of symbols ${+}$ whose length is the number of
  columns $1$ among the $\beta^{(j)}$, and $\rho_2(\beta^{(1)}\cdots\beta^{(m)})$ consists of a string of symbols
  ${+}$ whose length is the number of columns $12$ among the $\beta^{(j)}$. The result now follows by considering how
  many symbols ${+}$ are required to cancel symbols ${-}$ in $\rho_i(x)$:
  \begin{enumerate}
  \item Nothing to prove.
  \item Since $\rho_1(2) = {-}$, there must be at least one column $1$ among the $\beta^{(j)}$;
  \item Since $\rho_2(3) = {-}$, there must be at least one column $12$ among the $\beta^{(j)}$;
  \item Since $\rho_1(0) = {-}{+}$, there must be at least one column $1$ among the $\beta^{(j)}$;
  \item Since $\rho_1(\bar{3}) = {-}{-}$, there must be at least two columns $1$ among the $\beta^{(j)}$;
  \item Since $\rho_2(\bar{2}) = {-}$, there must be at least one column $12$ among the $\beta^{(j)}$;
  \item Since $\rho_1(\bar{1}) = {-}$, there must be at least one column $1$ among the $\beta^{(j)}$. \qedhere
  \end{enumerate}
\end{proof}

Consider first the case where there is no column $12$ among the $\beta^{(j)}$. That is,
$\beta^{(1)}\cdots\beta^{(m)}x = 1^m x$. In this case, $x$ can be $1$, $2$, $0$, $\bar{3}$ (only if $m \geq
2$), or $\bar{1}$ by \fullref{Lemma}{lem:g2:rightmultgen}, and so:
\begin{align*}
&\tikz[tableau]\matrix{x \& 1 \& 1 \& 1 \& |[dottedentry]| \& 1 \\}; \\
&\qquad=_{\Pl(G_2)}
\begin{cases}
\tikz[tableau]\matrix{1 \& 1 \& 1 \& 1 \& |[dottedentry]| \& 1 \\}; & \text{if $x = 1$;} \\[1mm]
\tikz[tableau]\matrix{1 \& 1 \& 1 \& |[dottedentry]| \& 1 \\2\\}; & \text{if $x = 2$;} \\[3.5mm]
\tikz[tableau]\matrix{1 \& 1 \& 1 \& |[dottedentry]| \& 1 \\}; & \text{if $x = 0$, since $P(10) = \tableau{1\\}$;} \\[1mm]
\tikz[tableau]\matrix{1 \& 1 \& |[dottedentry]| \& 1 \\2\\}; & \text{if $x = \bar{3}$, since $P(1\bar{3}) = \tableau{2\\}$;} \\[3.5mm]
\tikz[tableau]\matrix{1 \& 1 \& |[dottedentry]| \& 1 \\}; & \text{if $x = \bar{1}$ since $P(1\bar{1})$ is empty.} \\
\end{cases}
\end{align*}
In the first case, $c_{\beta^{(1)}}\cdots
  c_{\beta^{(m)}}c_{x}$ is in normal form; in the other four cases (respectively) the rewriting to normal form proceeds as follows:
\begin{align*}
  &c_{\beta^{(1)}}\cdots c_{\beta^{(m)}}c_{x} = c_1\cdots c_1c_1c_1c_x\\
  &\qquad\to
    \begin{cases}
      c_1\cdots c_1c_1c_{12} & \text{using $c_1c_2 \to c_{12}$ since $P(12) = \tableau{12\\}$;}\\
      c_1\cdots c_1c_1c_1 & \text{using $c_1c_0 \to c_{1}$ since $P(10) = \tableau{1\\}$;}\\
      c_1\cdots c_1c_1c_2 & \text{using $c_1c_{\bar{3}} \to c_{2}$ since $P(1\bar{3}) = \tableau{2\\}$;}\\
      \to c_1\cdots c_1c_{12} & \text{using $c_1c_2 \to c_{12}$ since $P(12) = \tableau{12\\}$;}\\
      c_1\cdots c_1c_1 & \text{using $c_1c_{\bar{1}} \to \emptyword$ since $P(1\bar)$ is empty.}\\
    \end{cases}
\end{align*}
In each case, rewriting $c_{\beta^{(1)}}\cdots c_{\beta^{(m)}}c_{x}$ to normal form involves at most two rewriting steps
at the right-hand end of the word. Note that the length of the normal form differs from $m$ by at most $2$.

Next consider the case where there is at least one column $12$ among the $\beta^{(j)}$. That is,
$\beta^{(1)}\cdots\beta^{(m)}x = 1^h(12)^k x$, with $k \geq 1$ and $h \geq 0$. By \fullref{Lemma}{lem:g2:rightmultgen},
$x$ can be $1$, $2$ (only if $h\geq 1$), $3$, $0$ (only if $h\geq 1$), $\bar{3}$ (only if $h \geq 2$), $\bar{2}$, or
$\bar{1}$ (only if $h\geq 1$). Consider each case in turn:
\begin{enumerate}

\item $x=1$. Then since $121 =_{\drel{R}_3^{G_2}} 112$, we have
  $P(\tableau[topalign]{1 \& 1 \\ \& 2\\}) = \tableau[topalign]{1 \& 1 \\ 2\\}$ and so $c_{12}c_1 \to c_1c_{12}$. Thus,
  using this rule at each step, rewriting to normal form is as follows:
\begin{align*}
c_{1}\cdots c_{1}c_{12}\cdots c_{12}c_{12}c_1
&\imreduces c_{1}\cdots c_{1}c_{12}\cdots c_{12}c_1c_{12} \\
&\qquad\vdots \\
&\imreduces c_{1}\cdots c_{1}c_1c_{12}\cdots c_{12}c_{12}.
\end{align*}


\item $x=2$. Then since $122 =_{\drel{R}_3^{G_2}} 212$, we have
  $P(\tableau[topalign]{2 \& 1 \\ \& 2\\}) = \tableau[topalign]{1 \& 2\\ 2
\\}$, and so $c_{12}c_2 \to c_2c_{12}$ is a rewriting
  rule. As noted above, there is at least one column $1$ present. So the rewriting to normal form proceeds as follows:
\begin{align*}
c_{1}\cdots c_{1}c_{1}c_{12}\cdots c_{12}c_{12}c_2
&\imreduces c_{1}\cdots c_{1}c_{1}c_{12}\cdots c_{12}c_2c_{12} && \text{using $c_{12}c_2 \to c_2c_{12}$}\\
&\qquad\vdots \\
&\imreduces c_{1}\cdots c_{1}c_{1}c_2c_{12}\cdots c_{12}c_{12} && \text{using $c_{12}c_2 \to c_2c_{12}$}\\
&\imreduces c_{1}\cdots c_{1}c_{12}c_{12}\cdots c_{12}c_{12}. && \text{using $c_1c_2 \to c_{12}$}\\
\end{align*}

\item $x=3$. Then since $123 =_{\drel{R}_4^{G_2}} 110 =_{\drel{R}_1^{G_2}} 11$ and
  $1211 =_{\drel{R}_3^{G_2}} 1121 =_{\drel{R}_3^{G_2}} 1112$, we have
  $P(\tikz[tableau,topalign]\matrix{3 \& 1 \\ \& 2\\};) = \tikz[tableau,topalign]\matrix{1 \& 1 \\};)$ So
  $c_{12}c_3 \to c_1c_1$ is a rewriting rule. Furthermore,
  $P(\tikz[tableau,topalign]\matrix{1 \& 1 \& 1 \\ \& \& 2\\};) = \tikz[tableau,topalign]\matrix{1 \& 1 \& 1 \\ 2\\};$.
  Thus add the extra rewriting rule $c_{12}c_1c_1 \imreduces c_1c_1c_{12}$. Now rewriting to normal form is
\begin{align*}
c_{1}\cdots c_{1}c_{12}\cdots c_{12}c_{12}c_{12}c_3
&\imreduces c_{1}\cdots c_{1}c_{12}\cdots c_{12}c_{12}c_1c_1 && \text{using $c_{12}c_3 \to c_1c_1$} \\
&\imreduces c_{1}\cdots c_{1}c_{12}\cdots c_{12}c_1c_1c_{12} && \text{using $c_1c_1c_{12} \to c_{12}c_1c_1$} \\
&\qquad\vdots \\
&\imreduces c_{1}\cdots c_{1}c_1c_1c_{12}\cdots c_{12}c_{12}. && \text{using $c_1c_1c_{12} \to c_{12}c_1c_1$} \\
\end{align*}

\item $x=0$. Then since $120 =_{\drel{R}_4^{G_2}} 210 =_{\drel{R}_1^{G_2}} 21$ and
  $1221 =_{\drel{R}_3^{G_2}} 2121 =_{\drel{R}_3^{G_2}} 2112$ and $121 =_{\drel{R}_3^{G_2}} 112$, we have
  $P(\tikz[tableau,topalign]\matrix{0 \& 1 \\ \& 2\\};) = \tikz[tableau,topalign]\matrix{1 \& 2\\};$ and so
  $c_{12}c_0 \to c_2c_1$ is a rewriting rule. Furthermore,
  $P(\tikz[tableau,topalign]\matrix{1 \& 2 \& 1\\ \& \& 2\\};) = \tikz[tableau,topalign]\matrix{1 \& 1 \& 2 \\ 2\\};$
  and $P(\tikz[tableau,topalign]\matrix{1 \& 2 \& 1 \\};) = \tikz[tableau,topalign]\matrix{1 \& 1 \\ 2\\};$. Thus, we
  add the extra rewriting rules $c_{12}c_2c_1 \imreduces c_2c_1c_{12}$ and $c_1c_2c_1 \imreduces c_1c_{12}$,

  As noted above, there is at least one column $1$ present. Rewriting to normal form is therefore as follows:
  \begin{align*}
    c_{1}\cdots c_{1}c_{12}\cdots c_{12}c_{12}c_{12}c_0
    &\imreduces c_{1}\cdots c_{1}c_{12}\cdots c_{12}c_{12}c_2c_1 && \text{using $c_{12}c_0 \to c_2c_1$}\\
    &\imreduces c_{1}\cdots c_{1}c_{12}\cdots c_{12}c_2c_1c_{12} && \text{using $c_{12}c_2c_1 \imreduces c_2c_1c_{12}$} \\
    &\qquad\vdots \\
    &\imreduces c_{1}\cdots c_{1}c_2c_1c_{12}\cdots c_{12}c_{12} && \text{using $c_{12}c_2c_1 \imreduces c_2c_1c_{12}$} \\
    &\imreduces c_{1}\cdots c_{1}c_{12}c_{12}\cdots c_{12}c_{12}. && \text{using $c_1c_2c_1 \imreduces c_1c_{12}$} \\
  \end{align*}

\item $x=\bar{3}$. Then since $12\bar{3} =_{\drel{R}_4^{G_2}} 21\bar{3} =_{\drel{R}_1^{G_2}} 22$, so
  $c_{12}c_{\bar{3}} \imreduces c_2c_2$ is a rewriting rule. Furthermore,
  $1222 =_{\drel{R}_3^{G_2}} 2122 =_{\drel{R}_3^{G_2}} 2212$ and $1122 =_{\drel{R}_3^{G_2}} 1212$, we have
  $P(\tikz[tableau,topalign]\matrix{\bar{3} \& 1 \\ \& 2\\};) = \tikz[tableau,topalign]\matrix{2 \& 2 \\};)$ and
  $P(\tikz[tableau,topalign]\matrix{2 \& 2 \& 1 \\ \& \& 2\\};) = \tikz[tableau,topalign]\matrix{1 \& 2 \& 2 \\ 2\\};$
  and $P(\tikz[tableau,topalign]\matrix{2 \& 2 \& 1 \& 1 \\};) = \tikz[tableau,topalign]\matrix{1 \& 1 \\ 2 \&
    2\\};$. Thus, we add the extra rewriting rules $c_{12}c_2c_2 \imreduces c_2c_2c_{12}$ and
  $c_2c_2c_1c_1 \imreduces c_{12}c_{12}$.

  As noted above, there are at least two columns $1$ present. Rewriting to normal form is therefore as follows:
  \begin{align*}
    c_{1}\cdots c_1c_1c_{1}c_{12}\cdots c_{12}c_{12}c_{12}c_{\bar{3}}
    &\imreduces c_{1}\cdots c_1c_1c_{1}c_{12}\cdots c_{12}c_{12}c_2c_2 && \text{using $c_{12}c_{\bar{3}} \imreduces c_2c_2$} \\
    &\imreduces c_{1}\cdots c_1c_1c_{1}c_{12}\cdots c_{12}c_2c_2c_{12} && \text{using $c_{12}c_2c_2 \imreduces c_2c_2c_{12}$}\\
    &\qquad\vdots \\
    &\imreduces c_{1}\cdots c_1c_1c_{1}c_2c_2c_{12}\cdots c_{12}c_{12} && \text{using $c_{12}c_2c_2 \imreduces c_2c_2c_{12}$}  \\
    &\imreduces c_{1}\cdots c_1c_{12}c_{12}c_{12}\cdots c_{12}c_{12} && \text{using $c_2c_2c_1c_1 \imreduces c_{12}c_{12}$}.
  \end{align*}

\item $x=\bar{2}$. Then since $12\bar{2} =_{\drel{R}_1^{G_2}} 10 =_{\drel{R}_1^{G_2}} 1$ and
  $121 =_{\drel{R}_3^{G_2}} 112$, we have
  $P(\tikz[tableau,topalign]\matrix{\bar{2} \& 1 \\ \& 2\\};) = \tikz[tableau,topalign]\matrix{1\\};$ and
  $P(\tikz[tableau,topalign]\matrix{1 \& 1 \\ \& 2\\};) = \tikz[tableau,topalign]\matrix{1 \& 1 \\ 2\\};$ and so
  $c_{12}c_{\bar{2}} \imreduces c_1$ and $c_{12}c_1 \imreduces c_1c_{12}$ are rewriting rules.

  Thus rewriting to normal form proceeds as follows:
  \begin{align*}
    c_{1}\cdots c_{1}c_{12}\cdots c_{12}c_{12}c_{12}c_{\bar{2}}
    &\imreduces c_{1}\cdots c_{1}c_{12}\cdots c_{12}c_{12}c_1 && \text{using $c_{12}c_{\bar{2}} \imreduces c_1$} \\
    &\imreduces c_{1}\cdots c_{1}c_{12}\cdots c_{12}c_1c_{12} && \text{using $c_{12}c_1 \imreduces c_1c_{12}$} \\
    &\qquad\vdots \\
    &\imreduces c_{1}\cdots c_{1}c_1c_{12}\cdots c_{12}c_{12}. && \text{using $c_{12}c_1 \imreduces c_1c_{12}$}
  \end{align*}

\item $x=\bar{1}$. Then since $12\bar{1} =_{\drel{R}_1^{G_2}} = 1\bar{3} =_{\drel{R}_1^{G_2}} 2$ and
  $122 =_{\drel{R}_3^{G_2}} 212$, we have
  $P(\tikz[tableau,topalign]\matrix{\bar{1} \& 1 \\ \& 2\\};) = \tikz[tableau,topalign]\matrix{2\\};$ and
  $P(\tikz[tableau,topalign]\matrix{2 \& 1 \\ \& 2\\};) = \tikz[tableau,topalign]\matrix{1 \& 2 \\ 2\\};$, so
  $c_{12}c_{\bar{1}} \imreduces c_2$ and $c_{12}c_2 \imreduces c_2c_{12}$.

  As noted above, there is at least one column $1$ present. Rewriting to normal form therefore proceeds as follows:
  \begin{align*}
    c_{1}\cdots c_1c_{1}c_{12}\cdots c_{12}c_{12}c_{12}c_{\bar{1}}
    &\imreduces c_{1}\cdots c_1c_{1}c_{12}\cdots c_{12}c_{12}c_{2} && \text{using $c_{12}c_{\bar{1}} \imreduces c_2$} \\
    &\imreduces c_{1}\cdots c_1c_{1}c_{12}\cdots c_{12}c_2c_{12} && \text{using $c_{12}c_2 \imreduces c_2c_{12}$} \\
    &\qquad\vdots \\
    &\imreduces c_{1}\cdots c_1c_{1}c_2c_{12}\cdots c_{12}c_{12} && \text{using $c_{12}c_2 \imreduces c_2c_{12}$} \\
    &\imreduces c_{1}\cdots c_1c_{12}c_{12}\cdots c_{12}c_{12}.&& \text{using $c_{1}c_2 \imreduces c_{12}$}
  \end{align*}

\end{enumerate}

Let $\Sigma$ and $T$ be the alphabet and set of rewriting rules constucted for type $G_2$ in
\fullref{Subsection}{subsec:rewritingconstruction}

Let $T'$ consist of the rules in $T$ and by rules corresponding to:
\begin{itemize}
\item  tabloids with shape $\tikz[shapetableau]\matrix{\null \& \null \& \null \\ \& \& \null\\};$ rewriting to tableaux with shape $\tikz[shapetableau]\matrix{\null \& \null \& \null
  \\ \null\\};$ (corresponding to extra rules in cases 3, 4, and 5 above);
\item tabloids with shape $\tikz[shapetableau]\matrix{\null \& \null \& \null \\};$ rewriting to tableaux with shape
  $\tikz[shapetableau]\matrix{\null \& \null \\ \null\\};$ (corresponding to an extra rule in case 4 above);
\item  tabloids with shape $\tikz[shapetableau]\matrix{\null \& \null \& \null \& \null\\};$ rewriting to tableaux with shape $\tikz[shapetableau]\matrix{\null \& \null \\ \null \& \null\\};$ (corresponding to an extra rule in case 5 above).
\end{itemize}
Note that the language of irreducible words is the same for the sets of rules
$T$ and $T'$, since a left-hand side of
some rule in $T$ must appear as a subword of the left-hand side of each rule in $T'$.

Let $u \in \Sigma^*$ and $c_x \in G_2$. By the analysis above, rewriting $uc_x$ to normal form using $T'$
proceeds via a single right-to-left pass, since rewriting at highest weight using $T'$ mirrors how rewriting
proceeds in general. Note that in each case the length of the normal form word differs from $m$ by at most $2$.

As in the discussion before \fullref{Lemma}{lem:an:rightmultbytrans}, the
relation consisting of pairs $(u,v)$ such that
$uc_x$ rewrites to $v$ can be recognized by a transducer. The only modification to the argument is that the transducer
that reads its input tapes right-to-left must stores the previous three symbols read from its first tape, so as to apply
the rule in $T' \setminus T$, and will always give these new rules precedence. With that change, the same argument proves the following
analogue of \fullref{Lemma}{lem:an:rightmultbytrans} for type $G_2$:

\begin{lemma}
  \label{lem:g2:rightmultbytrans}
  Let $\Sigma$ and $T'$ be as above. Let $x \in \aG_2$. Let $L \subseteq \Sigma^*$ be the language of irreducible words
  with respect to $T'$ (which is equal to the language of irreducible words with respect to $T$). Then the relation
  \[
    L_{c_x} = \gset[\big]{(u,v) \in L\times L}{uc_x =_{\Pl(G_2)} v}
  \]
  is recognized by an transducer. Furthermore, if $(u,v)$ is a pair in this relation, then the lengths of $u$ and $v$
  differ by at most $2$.
\end{lemma}

\section{Building the biautomatic structure}
\label{sec:biautomaticity}

Equipped with the lemmata from \fullref{Subsections}{subsec:leftmult} and \ref{subsec:rightmult}, we are now ready to
prove biautomaticity for the plactic monoids. First, we recall the essential definitions in
\fullref{Subsection}{subsec:biautopreliminaries}. We also state a result that allows us to discuss rational relations
rather than synchronous rational relations, which helps avoids some technical reasoning
(\fullref{Proposition}{prop:rationalbounded}). In \fullref{Subsection}{subsec:biautoconstruct}, we then proceed to build
the biautomatic structures and to examine some consequences and applications of biautomaticity.

\subsection{Preliminaries}
\label{subsec:biautopreliminaries}

This subsection contains the definitions and basic results from the theory of automatic and biautomatic monoids needed
hereafter. For further information on automatic semigroups, see~\cite{campbell_autsg}. We assume familiarity with basic
notions of automata and regular languages (see, for example, \cite{hopcroft_automata}).

\begin{definition}
Let $A$ be an alphabet and let $\$$ be a new symbol not in $A$. Define
the mapping $\rpad : A^* \times A^* \to ((A\cup\{\$\})\times (A\cup
\{\$\}))^*$ by
\[
(u_1\cdots u_m,v_1\cdots v_n) \mapsto
\begin{cases}
(u_1,v_1)\cdots(u_m,v_n) & \text{if }m=n,\\
(u_1,v_1)\cdots(u_n,v_n)(u_{n+1},\$)\cdots(u_m,\$) & \text{if }m>n,\\
(u_1,v_1)\cdots(u_m,v_m)(\$,v_{m+1})\cdots(\$,v_n) & \text{if }m<n,
\end{cases}
\]
and the mapping $\lpad : A^* \times A^* \to ((A\cup\{\$\})\times (A\cup \{\$\}))^*$ by
\[
(u_1\cdots u_m,v_1\cdots v_n) \mapsto
\begin{cases}
(u_1,v_1)\cdots(u_m,v_n) & \text{if }m=n,\\
(u_1,\$)\cdots(u_{m-n},\$)(u_{m-n+1},v_1)\cdots(u_m,v_n) & \text{if }m>n,\\
(\$,v_1)\cdots(\$,v_{n-m})(u_1,v_{n-m+1})\cdots(u_m,v_n) & \text{if }m<n,
\end{cases}
\]
where $u_i,v_i \in A$.
\end{definition}

\begin{definition}
\label{def:autstruct}
Let $M$ be a monoid. Let $A$ be a finite alphabet representing a set
of generators for $M$ and let $L \subseteq A^*$ be a regular language such
that every element of $M$ has at least one representative in $L$.  For
each $a \in A \cup \{\emptyword\}$, define the relations
\begin{align*}
L_a &= \{(u,v): u,v \in L, {ua} =_M {v}\}\\
{}_aL &= \{(u,v) : u,v \in L, {au} =_M {v}\}.
\end{align*}
The pair $(A,L)$ is an \defterm{automatic structure} for $M$ if $L_a\rpad$ is a regular language over $(A\cup\{\$\})
\times (A\cup\{\$\})$ for all $a \in A \cup \{\emptyword\}$. A monoid $M$ is \defterm{automatic} if it admits an
automatic structure with respect to some generating set.

The pair $(A,L)$ is a \defterm{biautomatic structure} for $M$ if $L_a\rpad$, ${}_aL\rpad$, $L_a\lpad$, and ${}_aL\lpad$
are regular languages over $(A\cup\{\$\}) \times (A\cup\{\$\})$ for all $a \in A \cup \{\emptyword\}$. A monoid $M$ is
\defterm{biautomatic} if it admits a biautomatic structure with respect to some generating set. [Note that
biautomaticity implies automaticity.]
\end{definition}

Unlike the situation for groups, biautomaticity for monoids and semigroups, like automaticity, is dependent on the
choice of generating set \cite[Example~4.5]{campbell_autsg}. However, for monoids, biautomaticity and automaticity are
independent of the choice of \emph{semigroup} generating sets \cite[Theorem~1.1]{duncan_change}.

Hoffmann \& Thomas have made a careful study of biautomaticity for semigroups \cite{hoffmann_biautomatic}. They
distinguish four notions of biautomaticity for semigroups, which are all equivalent for groups and more generally for
cancellative semigroups \cite[Theorem~1]{hoffmann_biautomatic} but distinct for semigroups \cite[Remark~1 \&
\S~4]{hoffmann_biautomatic}. In the sense used in this paper, `biautomaticity' implies \emph{all four} of these notions
of biautomaticity.

In proving that $R\rpad$ or $R\lpad$ is regular, where $R$ is a relation on $A^*$, a useful strategy is to prove that
$R$ is a rational relation (that is, is recognized by a transducer) and then apply the following result, which is a
combination of \cite[Corollary~2.5]{frougny_synchronized} and \cite[Proposition~4]{hoffmann_biautomatic}:

\begin{proposition}
\label{prop:rationalbounded}
If $R \subseteq A^* \times A^*$ is rational relation and there is a constant $k$ such that $\bigl||u|-|v|\bigr| \leq k$
for all $(u,v) \in R$, then $R\rpad$ and $R\lpad$ are regular.
\end{proposition}

\subsection{Construction}
\label{subsec:biautoconstruct}

In \fullref{Subsections}{subsec:leftmult} and \ref{subsec:rightmult}, we studied the rewriting that occurs when a normal
form word is left- or right-multiplied by a generator. We now turn to building biautomatic structures for the plactic
monoids of each type. Most of the work has been done; all that remains is to put together the pieces.

\begin{theorem}
  \label{thm:biautomaticity}
  The plactic monoids $\Pl(A_n)$, $\Pl(B_n)$, $\Pl(C_n)$, $\Pl(D_n)$, and $\Pl(G_2)$ are biautomatic.
\end{theorem}

\begin{proof}
Let $X$ be one of the types $A_n$, $B_n$, $C_n$, $D_n$, and $G_2$ and let $\aX$ be the corresponding alphabet from
$\aA_n$, $\aB_n$, $\aC_n$, $\aD_n$, or $\aG_2$. Let $(\Sigma,T)$ be the rewriting system constructed in
\fullref{Section}{sec:rewriting} for $\Pl(X)$. Let $L \subseteq \Sigma^*$ be the language of irreducible words.

Let $x \in \aX$. By \fullref{Lemmata}{lem:anbncndn:leftmultbytrans} and \ref{lem:g2:leftmultbytrans}, the relation
\[
  {}_{c_x}L = \gset[\big]{(u,v) \in L\times L}{c_xu =_{\Pl(X)} v}
\]
is a rational relation. By \fullref{Lemmata}{lem:an:rightmultbytrans}, \ref{lem:cn:rightmultbytrans}, \ref{lem:bn:rightmultbytrans},
\ref{lem:dn:rightmultbytrans}, and \ref{lem:g2:rightmultbytrans}, the relation
\[
  L_{c_x} = \gset[\big]{(u,v) \in L\times L}{uc_x =_{\Pl(X)} v}
\]
is a rational relation.

Now let $c_\sigma \in \Sigma$. So $\sigma$ is an admissible $X$ column and $\sigma = \sigma_1\cdots\sigma_k$ for some
$\sigma_i \in \aX$, with $k \leq n$ when $X \in \set{A_n,B_n,C_n,D_n}$ and $k \leq 2$ when $X = G_2$. So
\begin{equation}
  \label{eq:composingrationalrelations}
  \begin{aligned}
    {}_{c_\sigma}L &= {}_{c_{\sigma_1}}L\circ \cdots \circ {}_{c_{\sigma_k}}L, \\
    L_{c_\sigma} &= L_{c_{\sigma_1}}\circ \cdots \circ L_{c_{\sigma_k}}.
  \end{aligned}
\end{equation}
Since the composition of a rational relation is a rational relation, $L_{c_\sigma}$ and ${}_{c_\sigma}L$ are rational
relations for any $c_\sigma \in \Sigma$.

By \fullref{Lemmata}{lem:anbncndn:leftmultbytrans} and \ref{lem:g2:leftmultbytrans}, if $(u,v) \in {}_{c_x}L$ then
$\bigl||u| - |v|\bigr| \leq 1$. Hence if $(u,v) \in {}_{c_\sigma}L$ them $\bigl||u| - |v|\bigr| \leq n$. Therefore
${}_{c_\sigma}L\rpad$ and ${}_{c_\sigma}L\lpad$ are both regular.

By \fullref{Lemmata}{lem:an:rightmultbytrans}, \ref{lem:cn:rightmultbytrans}, \ref{lem:bn:rightmultbytrans},
\ref{lem:dn:rightmultbytrans}, and \ref{lem:g2:rightmultbytrans}, if $(u,v) \in L_{c_x}$ then
$\bigl||u| - |v|\bigr| \leq 1$. Hence if $(u,v) \in L_{c_\sigma}$ them $\bigl||u| - |v|\bigr| \leq n$. Therefore
$L_{c_\sigma}\rpad$ and $L_{c_\sigma}\lpad$ are both regular.

Therefore $(\Sigma,L)$ is a biautomatic structure for $\Pl(X)$.
\end{proof}

\fullref{Theorem}{thm:biautomaticity} has several important consequences for the plactic monoids $\Pl(A_n)$, $\Pl(B_n)$,
$\Pl(C_n)$, $\Pl(D_n)$, and $\Pl(G_2)$. First, an automatic monoid has decidable right-divisibility problem. (This
result is well-known but does not seem to be explicitly stated in the literature; it follows from the decidability of
the first-order theory of the left Cayley graph of an automatic monoid \cite[\S~5]{lohrey_decidability}.) Combining this
result and its dual with \fullref{Theorem}{thm:biautomaticity} proves the following result:

\begin{corollary}
  \label{corol:divisibility}
  The plactic monoids $\Pl(A_n)$, $\Pl(B_n)$, $\Pl(C_n)$, $\Pl(D_n)$, and $\Pl(G_2)$ have soluble left- and
  right-divisibilty problems.
\end{corollary}

An immediate consequence of \fullref{Corollary}{corol:divisibility} is the following:

\begin{corollary}
  The plactic monoids $\Pl(A_n)$, $\Pl(B_n)$, $\Pl(C_n)$, $\Pl(D_n)$, and $\Pl(G_2)$ have soluble Green's relation
  $\mathcal{L}$ and $\mathcal{R}$.
\end{corollary}

There are also several very important crystal-theoretic consequences of the biautomaticity of the plactic monoids:

\begin{corollary}
  \label{corol:wordproblem}
  For the crystal graphs of types $A_n$, $B_n$, $C_n$, $D_n$, or $G_2$, there is a quadratic-time algorithm that takes
  as input two vertices and decides whether they lie in the same position in isomorphic components.
\end{corollary}

\begin{proof}
  Two vertices lie in the same position in isomorphic connected components if and only if they represent the same
  element of the plactic monoid of the given type. This monoid is biautomatic by \fullref{Theorem}{thm:biautomaticity},
  and biautomatic (and automatic) monoids have word problem soluble in quadradic time
  \cite[Corollary~3.7]{campbell_autsg}.
\end{proof}

Note in passing that \fullref{Corollary}{corol:wordproblem} cannot be deduced directly from tableaux insertion
algorithms except in the $A_n$ case. Schensted's insertion algorithm (see \cite[Chapter~5]{lothaire_algebraic}) can
solve the word problem in $\Pl(A_n)$ in quadratic time because inserting a single symbol into a tableau takes linear
time. However, in types $B_n$, $C_n$, and $D_n$ inserting a single symbol into a tableau may take more that linear time
(see \cite[\S~4]{lecouvey_cn} and \cite[\S~3.3]{lecouvey_bndn}), because in certain cases a recursion arises that
requires inserting an entire column symbol by symbol into the remainder of the tableau.

\begin{corollary}
  \label{corol:isomorphiccomponents}
  For the crystal graphs of types $A_n$, $B_n$, $C_n$, $D_n$, or $G_2$, there is a
  quadratic-time algorithm that takes as input two vertices and decides that whether they lie in isomorphic components.
\end{corollary}

\begin{proof}
  Let $B(u_1)$ and $B(u_2)$ be two components of the crystal graph, where $u_1$ and $u_2$ are any vertices of these
  components. Apply operators $\e_i$ to transform $u_1$ and $u_2$ to highest-weight words $v_1$ and $v_2$
  respectively. It is easy to see that each application of $\e_i$ takes linear time in the length of the word. Each
  symbol of the word can be altered a bounded number of times by the various $\e_i$, so computing $v_1$ and $v_2$ takes
  at most quadratic time in the lengths of $u_1$ and $v_1$. Then $B(u_1)$ and $B(u_2)$ are isomorphic if and only if $v_1$ and
  $v_2$ lie in the same position in $B(u_1)$ and $B(u_2)$, which can be decided in quadratic time by
  \fullref{Corollary}{corol:wordproblem}.
\end{proof}

\bibliographystyle{alphaabbrv}
\bibliography{\jobname}

\end{document}